\documentclass[reqno, 11pt]{amsart}
\usepackage{amsmath, amssymb, amsthm}
\usepackage[margin=2.75cm, heightrounded]{geometry}
\usepackage{mathrsfs}
\usepackage{graphicx, subcaption}
\usepackage[dvipsnames]{xcolor}
\usepackage{bm}
\usepackage[colorlinks=true, linkcolor=blue, citecolor=blue, urlcolor=blue]{hyperref}

\newcommand{\be}{\begin{equation}}
\newcommand{\ee}{\end{equation}}

\DeclareMathOperator{\Id}{Id}
\DeclareMathOperator{\Ric}{Ric}
\DeclareMathOperator{\Rm}{Rm}
\DeclareMathOperator{\cro}{cr}
\DeclareMathOperator{\dist}{dist}
\DeclareMathOperator{\diam}{diam}
\DeclareMathOperator{\Vol}{Vol}
\DeclareMathOperator*{\esssup}{ess\,sup}
\DeclareMathOperator{\arctanh}{arctanh}
\DeclareMathOperator{\Stab}{Stab}

\newcommand{\ga}{\gamma}
\newcommand{\ep}{\epsilon}
\newcommand{\vep}{\varepsilon}
\newcommand{\vsi}{\varsigma}
\newcommand{\vph}{\varphi}
\newcommand{\vrh}{\varrho}
\newcommand{\vth}{\vartheta}
\newcommand{\pp}{\partial}

\newcommand{\mone}{\mathbf{1}}
\newcommand{\mbC}{\mathbf{C}}

\newcommand{\N}{\mathbb{N}}
\newcommand{\R}{\mathbb{R}}
\newcommand{\Ss}{\mathbb{S}}

\newcommand{\mca}{\mathcal{A}}
\newcommand{\mcb}{\mathcal{B}}
\newcommand{\mcc}{\mathcal{C}}
\newcommand{\mcd}{\mathcal{D}}
\newcommand{\mce}{\mathcal{E}}
\newcommand{\mcg}{\mathcal{G}}
\newcommand{\mch}{\mathcal{H}}
\newcommand{\mci}{\mathcal{I}}
\newcommand{\mck}{\mathcal{K}}
\newcommand{\mco}{\mathcal{O}}
\newcommand{\mcp}{\mathcal{P}}
\newcommand{\mcq}{\mathcal{Q}}
\newcommand{\mcr}{\mathcal{R}}
\newcommand{\mcs}{\mathcal{S}}
\newcommand{\mct}{\mathcal{T}}
\newcommand{\mcu}{\mathcal{U}}
\newcommand{\mcv}{\mathcal{V}}
\newcommand{\mcw}{\mathcal{W}}
\newcommand{\mcy}{\mathcal{Y}}
\newcommand{\mcz}{\mathcal{Z}}

\newcommand{\mfc}{\mathfrak{c}}
\newcommand{\mfC}{\mathfrak{C}}
\newcommand{\mfm}{\mathfrak{m}}
\newcommand{\mfM}{\mathfrak{M}}
\newcommand{\mfs}{\mathfrak{s}}

\newcommand{\mscA}{\mathscr{A}}
\newcommand{\mscL}{\mathscr{L}}
\newcommand{\mscR}{\mathscr{R}}
\newcommand{\mscS}{\mathscr{S}}

\newcommand{\msa}{\mathsf{a}}
\newcommand{\msA}{\mathsf{A}}
\newcommand{\msb}{\mathsf{b}}
\newcommand{\msB}{\mathsf{B}}
\newcommand{\msc}{\mathsf{c}}
\newcommand{\msd}{\mathsf{d}}
\newcommand{\msM}{\mathsf{M}}
\newcommand{\msN}{\mathsf{N}}
\newcommand{\msP}{\mathsf{P}}
\newcommand{\msR}{\mathsf{R}}
\newcommand{\mss}{\mathsf{s}}
\newcommand{\msS}{\mathsf{S}}
\newcommand{\mst}{\mathsf{t}}
\newcommand{\msT}{\mathsf{T}}

\newcommand{\baf}{\bar{f}}
\newcommand{\bg}{\bar{g}}
\newcommand{\bs}{\bar{s}}
\newcommand{\bu}{\bar{u}}
\newcommand{\bx}{\bar{x}}
\newcommand{\bz}{\bar{z}}
\newcommand{\bmu}{\bar{\mu}}
\newcommand{\bps}{\bar{\psi}}

\newcommand{\hf}{\hat{f}}
\newcommand{\hg}{\hat{g}}
\newcommand{\hu}{\hat{u}}
\newcommand{\hph}{\hat{\phi}}
\newcommand{\hps}{\hat{\psi}}
\newcommand{\whe}{\widehat{E}}
\newcommand{\whmch}{\widehat{\mathcal{H}}}
\newcommand{\whmcz}{\widehat{\mathcal{Z}}}

\newcommand{\te}{\tilde{e}}
\newcommand{\tf}{\tilde{f}}
\newcommand{\tg}{\tilde{g}}
\newcommand{\tit}{\tilde{t}}
\newcommand{\tu}{\tilde{u}}
\newcommand{\tx}{\tilde{x}}
\newcommand{\ty}{\tilde{y}}
\newcommand{\tz}{\tilde{z}}
\newcommand{\tmu}{\tilde{\mu}}
\newcommand{\tnu}{\tilde{\nu}}
\newcommand{\txi}{\tilde{\xi}}
\newcommand{\tze}{\tilde{\zeta}}
\newcommand{\tph}{\tilde{\phi}}
\newcommand{\tvph}{\tilde{\varphi}}
\newcommand{\trh}{\tilde{\rho}}
\newcommand{\tps}{\tilde{\psi}}

\newcommand{\wte}{\widetilde{E}}
\newcommand{\wtq}{\widetilde{Q}}
\newcommand{\wtw}{\widetilde{W}}
\newcommand{\wtz}{\widetilde{Z}}
\newcommand{\wtom}{\widetilde{\Omega}}
\newcommand{\wtsh}{\widetilde{\sharp}}
\newcommand{\wtshs}{\widetilde{\sharp\sharp}}
\newcommand{\wtmch}{\widetilde{\mathcal{H}}}

\newcommand{\bDelta}{\boldsymbol\Delta}
\newcommand{\bvth}{\boldsymbol\vth}
\newcommand{\bmcl}{\boldsymbol{\mathcal{L}}}

\newcommand{\bub}{\mathrm{bub}}
\newcommand{\ch}{\mathrm{ch}}
\newcommand{\mint}{\mathrm{int}}
\newcommand{\loc}{\mathrm{loc}}
\newcommand{\met}{\mathrm{met}}

\renewcommand{\(}{\left(}
\renewcommand{\)}{\right)}
\newcommand{\la}{\left\langle}
\newcommand{\ra}{\right\rangle}

\newtheorem{theorem}{Theorem}[section]
\newtheorem{lemma}[theorem]{Lemma}
\newtheorem{prop}[theorem]{Proposition}
\newtheorem{corollary}[theorem]{Corollary}

\theoremstyle{definition}
\newtheorem{remark}[theorem]{Remark}
\newtheorem{definition}[theorem]{Definition}

\numberwithin{equation}{section}
\allowdisplaybreaks
\begin{document}
\title[Ancient Yamabe Flows on the Sphere]{Symmetry Breaking and Moduli \\ for Ancient Yamabe Flows on the Sphere}

\author{Haixia Chen}
\address[Haixia Chen]{School of Mathematics and Statistics, Central China Normal University, Wuhan 430079, P. R. China/Department of Mathematics and Research Institute for Natural Sciences, College of Natural Sciences, Hanyang University, 222 Wangsimni-ro Seongdong-gu, Seoul 04763, Republic of Korea}
\email{hxchen@ccnu.edu.cn}
\email{chenhaixia157@gmail.com}

\author{Seunghyeok Kim}
\address[Seunghyeok Kim]{Department of Mathematics and Research Institute for Natural Sciences, College of Natural Sciences, Hanyang University, 222 Wangsimni-ro Seongdong-gu, Seoul 04763, Republic of Korea}
\email{shkim0401@hanyang.ac.kr}
\email{shkim0401@gmail.com}

\author{Monica Musso}
\address[Monica Musso]{Department of Mathematical Sciences, University of Bath, Bath BA2 7AY, United Kingdom}
\email{mm2683@bath.ac.uk}

\begin{abstract}
For every $n\ge3$, we construct compact Type II ancient Yamabe flows in the round conformal class of the unit $n$-sphere $\Ss^n$.
After stereographic projection and rescaling of time and conformal factor, the solutions are asymptotic backward in time to finite superpositions of Aubin--Talenti bubbles with mutually comparable scales.
For integers $h\ge1$ and $k\ge2$, the limiting centers form $h$ coaxial regular $k$-gons; when $h\ge2$, their distinct layer heights may be prescribed arbitrarily in the interval $(-1,1)$.
If $hk\ge3$, the resulting flows are not geometrically equivalent to any rotationally symmetric ancient flow.
For each flow, there is a fixed point of $\Ss^n$ at which the Ricci tensor has both positive and negative eigenvalues on every time slice.

The corresponding volume-normalized Yamabe flows are also ancient.
After rescaling each time slice to a fixed volume, the metrics converge to a wedge of $hk$ isometric round $n$-spheres in the Gromov--Hausdorff topology, as time tends to $-\infty$.
For fixed $h,k\ge2$, the concentration configurations admit an explicit $(h-1)$-parameter description up to M\"obius transformations.
Consequently, the construction yields geometrically inequivalent flows indexed by a nonempty open subset of $\R^{h-1}$, all with the same fixed-volume backward Gromov--Hausdorff limit.

The proof uses inner--outer gluing for the quasilinear parabolic equation in $n$-dimensional space that exhibits Kelvin symmetry.
Weighted Euclidean and spherical H\"older estimates control the nonlinear terms and, together with parabolic regularity, yield smooth extension of the conformal metrics across the north pole of $\Ss^n$.
\end{abstract}

\date{\today}
\maketitle
\enlargethispage{2pt}

\section{Introduction}
\subsection{Ancient solutions, rigidity, and symmetry breaking}\label{subsec:intro1}
Ancient solutions, defined on intervals $(-\infty,T)$, arise in the analysis of singularities of geometric evolution equations. 
In many important cases, combining compactness with conditions such as noncollapsing, convexity, or curvature yields strong rigidity.
For example, classification results for compact ancient Ricci flows and ancient mean curvature flows are available in \cite{DHS12,BDS21,BDNS23,DHS10,ADS20}.

The Yamabe flow, introduced by Hamilton \cite{ha89}, is
\be \label{YF-ori}
\pp_s g(s)=-R_{g(s)}g(s),
\ee
where $R_g$ denotes scalar curvature and $s$ is the time variable. It preserves the conformal class. The corresponding volume-normalized flow is
\be \label{YF-nori}
\pp_s g(s)=-\(R_{g(s)}-r_{g(s)}\)g(s),
\ee
where $r_g$ is the average scalar curvature. The global existence and convergence theory is by now well developed \cite{chow92,ye94,SS03,brendle05,brendle07}.
By contrast, the ancient problem remains largely open, even for the conformal class on the unit $n$-sphere $\Ss^n := \{z\in\R^{n+1}:|z|=1\}$ ($n\ge3$) with its round metric $g_{\Ss^n}$.

\medskip
An ancient solution of \eqref{YF-ori} is of \emph{Type I} if
\be \label{eq:typeI}
\limsup_{s\to-\infty}|s|\max_{z\in\Ss^n}|\Rm_{g(s)}(z)|_{g(s)}<\infty,
\ee
and of \emph{Type II} otherwise.

The compact ancient Yamabe flows in \cite{King93,ddks,dds} are rotationally symmetric. King obtained explicit Type I solutions \cite{King93}.
Daskalopoulos, del Pino, King, and Sesum found five-parameter Type I families that, going backward in time, approach two traveling waves separated by a cylindrical region \cite{ddks}.
Daskalopoulos, del Pino, and Sesum constructed Type II bubble towers with spherical components at different scales along one axis \cite{dds}.
Their rotational symmetry reduces the equation to one spatial variable.

Here we construct finite-bubble ancient Yamabe flows in every dimension $n\ge3$ with distinct angular configurations of concentration points, freely prescribed distinct layer heights, and uniformly bounded Yamabe quotient \eqref{eq:Yamabequo}; see Theorems \ref{thm:main-geo}--\ref{thm:main-an} below.
The main feature is not solely the lack of rotational symmetry. Appendix \ref{app:quadratic-S4} presents a straightforward construction of a compact, non-rotational Type II ancient flow on $\Ss^4$ based on a quadratic ansatz; however, its Yamabe quotient diverges as $s\to-\infty$.
The solutions presented in this paper exhibit finite-bubble concentration patterns with bounded Yamabe quotients, and their relative arrangements encode significant conformal information.

For comparison, the two-dimensional analog of \eqref{YF-ori}, namely, compact ancient Ricci flows on $\Ss^2$, is rotationally symmetric, as established in \cite{DHS12}. These flows are classified as either shrinking round spheres or King--Rosenau solutions \cite{King93,Ro}.

\medskip
As shown in Subsection \ref{subsec:para}, the flow \eqref{YF-ori} can be reduced to the quasilinear energy-critical parabolic equation \eqref{pb}.
With the normalization used there, the elliptic Yamabe equation on Euclidean space $\R^n$ is
\be \label{eq:critical-yamabe}
\begin{cases}
-\Delta w=\dfrac{n+2}{4}w^p &\text{in } \R^n,\\
w>0 &\text{in } \R^n,
\end{cases}
\ee
where $p:=\frac{n+2}{n-2}$. We define
\be \label{eq:bubble}
U(y) := \frac{\msc_n}{(1+|y|^2)^{\frac{n-2}{2}}},
\qquad
\msc_n := \left[\frac{4n(n-2)}{n+2}\right]^{\frac{n-2}{4}},
\ee
and
\be \label{eq:bubble2}
U_{\mu,\xi}(x) := \mu^{-\frac{n-2}{2}}U\big(\mu^{-1}(x-\xi)\big) \quad \text{for } (\mu,\xi)\in(0,\infty)\times\R^n.
\ee
The function $U=U_{1,0}$ is called the \emph{normalized Aubin--Talenti bubble}, and each $U_{\mu,\xi}$ is known as an \emph{Aubin--Talenti bubble}, or simply a \emph{bubble}.
These functions exhaust the positive entire solutions of \eqref{eq:critical-yamabe}; see \cite{CGS89}.
Thus every positive stationary solution is a single bubble, radial about its center, and no positive multi-bubble solutions exist in the elliptic case. The symmetry breaking found here is therefore a genuinely parabolic phenomenon.

\medskip
We write ${\bf e}_j$ for the $j$-th standard basis vector in the Euclidean space under consideration. We define
\[
G_0 := \{\operatorname{diag}(A,1):A\in O(n)\} \subset O(n+1).
\]
The group $G_0$ acts on $\Ss^n$ by fixing only the two points $\{\pm {\bf e}_{n+1}\}$, and every other orbit is an $(n-1)$-sphere.
For any $Q\in O(n+1)$, we define $G_Q:=QG_0Q^{-1}$. A Riemannian metric $g$ on $\Ss^n$ is called \emph{rotationally symmetric} if $\Psi^*g=g$ for all $\Psi\in G_Q$, for some $Q\in O(n+1)$.
A Yamabe flow $g(s)$ on an interval $I$ is rotationally symmetric if a fixed subgroup $G_Q$, where $Q$ is independent of $s$, preserves each time slice:
\be \label{eq:rotational-symmetry-definition}
\Psi^*g(s)=g(s) \quad \text{for every }\Psi\in G_Q \text{ and every } s\in I.
\ee
The line $Q(\R {\bf e}_{n+1})$ is the axis of rotation, and $\{\pm Q{\bf e}_{n+1}\}$ is the fixed-point set on $\Ss^n$.
Indeed, stereographic projection from either fixed point, after identifying its target with $\R^n$, conjugates the $G_Q$-action to the standard linear action of $O(n)$ on $\R^n$.
For metrics in the round conformal class, \eqref{eq:rotational-symmetry-definition} is therefore equivalent to radiality of the Euclidean conformal factor in that chart.

We say that two ancient Yamabe flows $g_1$ and $g_2$ on $\Ss^n$ are \emph{geometrically equivalent} if there exist $c>0$, $s_0\in\R$, and a fixed conformal diffeomorphism $\Phi\in\operatorname{Conf}(\Ss^n)$ such that
\be \label{eq:geo-equiv-sphere}
g_2(s)=c\,\Phi^*g_1\big(c^{-1}(s-s_0)\big)
\ee
whenever both sides are defined.

For a parameterized family $\{g_\lambda\}_{\lambda\in\Lambda}$, a \emph{continuous modulus} is a continuous real-valued function on $\Lambda$ that takes the same value on parameters corresponding to geometrically equivalent flows and whose image contains a nondegenerate interval.
Here, continuity concerns the invariant as a function of the prescribed concentration configuration.

\subsection{Main geometric result}
Assuming that $n \ge 3$, $h \ge 1$, and $k \ge 2$ are integers, we set
\be \label{eq:mcah}
\mca_h := \begin{cases}
\{(0)\} &\text{if } h=1,\\
\{(\vth_1^*,\dots,\vth_h^*)\in(-1,1)^h: \vth_1^*<\dots<\vth_h^*\} &\text{if } h\ge 2,
\end{cases}
\ee
and $\bvth^* := (\vth_1^*,\dots,\vth_h^*) \in \mca_h$. Let
\be \label{eq:omegaj}
\omega_j := e^{\frac{2\pi(j-1)}{k} i} = \(\cos\frac{2\pi(j-1)}k,\sin\frac{2\pi(j-1)}k\) \in \Ss^1 \quad \text{for } j=1,\dots,k,
\ee
and define the limiting Euclidean centers
\be \label{eq:bfq}
{\bf q}_{jl}^* := \(\sqrt{1-(\vth_l^*)^2}\,\omega_j,\vth_l^*,0_{\R^{n-3}}\)
\in\Ss^{n-1}\subset\R^n,
\ee
where the final zero block is absent when $n=3$. We then set
\be \label{eq:Phk}
P_{h,k}(\bvth^*) := \{{\bf q}_{jl}^*:j=1,\dots,k,\ l=1,\dots,h\} \subset \Ss^{n-1} \subset \R^n,
\ee
which will serve as the limiting configuration of the bubble centers for equation \eqref{pb} as $t\to-\infty$.
For each $l \in \{1,\dots,h\}$, the $k$ points corresponding to that value of $l$ form a regular polygon in the affine plane
\[
\{(x_1,x_2,x_3,0_{\R^{n-3}}):x_3=\vth_l^*\}\subset\R^n.
\]
All layers have the same angular phase and the common axis $\R {\bf e}_3\subset\R^n$.
As a result, $P_{1,k}(0)$ is a regular $k$-gon in an equatorial circle of $\Ss^{n-1}$, whereas $P_{h,k}(\bvth^*)$ is the union of $h$ coaxial regular $k$-gons.

To identify the corresponding points on  $\Ss^n$, we define $\msN_{\infty}:={\bf e}_{n+1} \in \R^{n+1}$ as the north pole of $\Ss^n$ and let
\be \label{stereo-proj}
\pi:\R^n \to \Ss^n\setminus\{\msN_{\infty}\}, \qquad z=\pi(y) := \(\frac{2y}{1+|y|^2}, \frac{|y|^2-1}{1+|y|^2}\)
\ee
denote inverse stereographic projection from $\msN_{\infty}$. Every point $q \in P_{h,k}(\bvth^*)$ has $|q|=1$, so $\pi(q) = (q, 0)$ belongs to $\Ss^n$. We denote the corresponding spherical configuration by
\be \label{eq:Sigmahk}
\Sigma_{h,k}(\bvth^*) := \pi(P_{h,k}(\bvth^*)) \subset\Ss^n\cap\{z_{n+1}=0\}.
\ee

\medskip
In the following, $\bigvee_{j=1}^{hk}\Ss^n$ denotes the quotient space obtained from the disjoint union of $hk$ pointed round $n$-spheres of the same radius by identifying all their marked points.
We equip this space with the path metric, whose precise definition is given in \eqref{eq:geo-wedge-dist}.
In particular, each spherical component is embedded isometrically. In Theorem \ref{thm:main-geo}\textup{(vi)}, the radius of each component is selected such that its volume equals $V_0/(hk)$, where $V_0>0$ is chosen arbitrarily.

\begin{theorem}[Main geometric result]\label{thm:main-geo}
Suppose that $n \ge 3$, $h \ge 1$, and $k \ge 2$ are integers, and $\bvth^* \in \mca_h$. 
Then there exists a compact ancient Yamabe flow $g_{h,k,\bvth^*}(s)$ that solves equation \eqref{YF-ori} on $\Ss^n \times (-\infty,0)$ and has the following properties.
\begin{enumerate}
\item[(i)] We define the normalized volume measure at time $s$ as
\be \label{eq:nvm}
\nu_s := \frac{dv_{g_{h,k,\bvth^*}(s)}} {\Vol_{g_{h,k,\bvth^*}(s)}(\Ss^n)}.
\ee
The backward concentration set on $\Ss^n$ is $\Sigma_{h,k}(\bvth^*)$, in the sense that
\[
\nu_s\rightharpoonup \frac1{hk} \sum_{z\in \Sigma_{h,k}(\bvth^*)} \delta_z \quad \text{as } s\to-\infty
\]
weakly in the space of probability measures on $\Ss^n$.
\item[(ii)] We define the scale-invariant Yamabe quotient by
    \be \label{eq:Yamabequo}
    \mcy(g) := \frac{\int_{\Ss^n}R_g\,dv_g} {\Vol_g(\Ss^n)^{(n-2)/n}} \quad \text{so that } \mcy(g_{\Ss^n})=n(n-1)|\Ss^n|^{2/n},
    \ee
    where $|\Ss^n|$ is the volume (i.e., surface area) of the unit sphere $\Ss^n$ with respect to the standard round metric $g_{\Ss^n}$. Then
    \[
    \mcy(g_{h,k,\bvth^*}(s)) \to (hk)^{2/n}\mcy(g_{\Ss^n}) \quad \text{as } s\to-\infty.
    \]
\item[(iii)] The flow is of Type II.
\item[(iv)] There is a fixed point $z_0\in\Ss^n$ at which the Ricci curvature is indefinite for every $s<0$.
\item[(v)] If $hk\ge3$, then the flow is not geometrically equivalent to any rotationally symmetric ancient Yamabe flow.
\item[(vi)] Fix an arbitrary constant $V_0>0$, and define
    \be \label{eq:tg}
    \tg_{h,k,\bvth^*}(s) := \(\frac{V_0}{\Vol_{g_{h,k,\bvth^*}(s)}(\Ss^n)}\)^{2/n} g_{h,k,\bvth^*}(s)
    \ee
    so that $\Vol_{\tg_{h,k,\bvth^*}(s)}(\Ss^n)=V_0$ for every $s<0$.
    As $s\to-\infty$, the compact metric spaces $(\Ss^n,\tg_{h,k,\bvth^*}(s))$ converge to $\bigvee_{j=1}^{hk}\Ss^n$ in the Gromov--Hausdorff topology (see Definition \ref{def:GH}), where the spherical components are mutually isometric and each has volume $V_0/(hk)$.
\item[(vii)] Flows with different numbers of bubbles are geometrically inequivalent.
    For fixed $h,k\ge2$, there is a nonempty relatively open subset $\mcv\subset\{\bvth^*\in\mca_h:\vth_1^*=0\}$ 
    such that any two flows constructed for distinct $\bvth^*,\bvth^{*\prime}\in\mcv$ are geometrically inequivalent.
\end{enumerate}
\end{theorem}
The finite concentration set imposes an intrinsic obstruction to rotational symmetry. Specifically, for $Q\in O(n+1)$, any finite $G_Q$-invariant set contains at most the two fixed points of $G_Q$.
Furthermore, geometric equivalence preserves the number of concentration points, which explains why Theorem \ref{thm:main-geo}\textup{(v)} holds.

Subsection \ref{subsec:equimod} presents Proposition \ref{prop:cont-moduli}, which establishes that the spherical configurations \eqref{eq:Sigmahk} arising in the construction can be parametrized by $(0,\infty)^{h-1}$, up to label-preserving M\"obius equivalence.
If label permutations are allowed, then each M\"obius-equivalence class can be represented by at most $(hk)!$ different parameter values.
Moreover, distinct parameters within the nonempty open subset $\mcv$ of this parameter space yield geometrically inequivalent flows.
Remark \ref{rem:continuous-modulus} defines a continuous modulus as the minimum chordal cross-ratio \eqref{eq:ccr} over all ordered quadruples of distinct concentration points.
The classification concerns concentration configurations, not flows with a prescribed concentration set; see Remark \ref{rem:config-versus-flow-moduli}.

By Theorem \ref{thm:main-geo}\textup{(vi)}, flows constructed for fixed $h,k$ have the same fixed-volume backward Gromov--Hausdorff limit even when their concentration sets are not M\"obius-equivalent.
Therefore, the limiting wedge does not retain all conformal information of the concentration configuration.
\begin{figure}[htbp]
    \centering
    \includegraphics[width=0.85\textwidth]{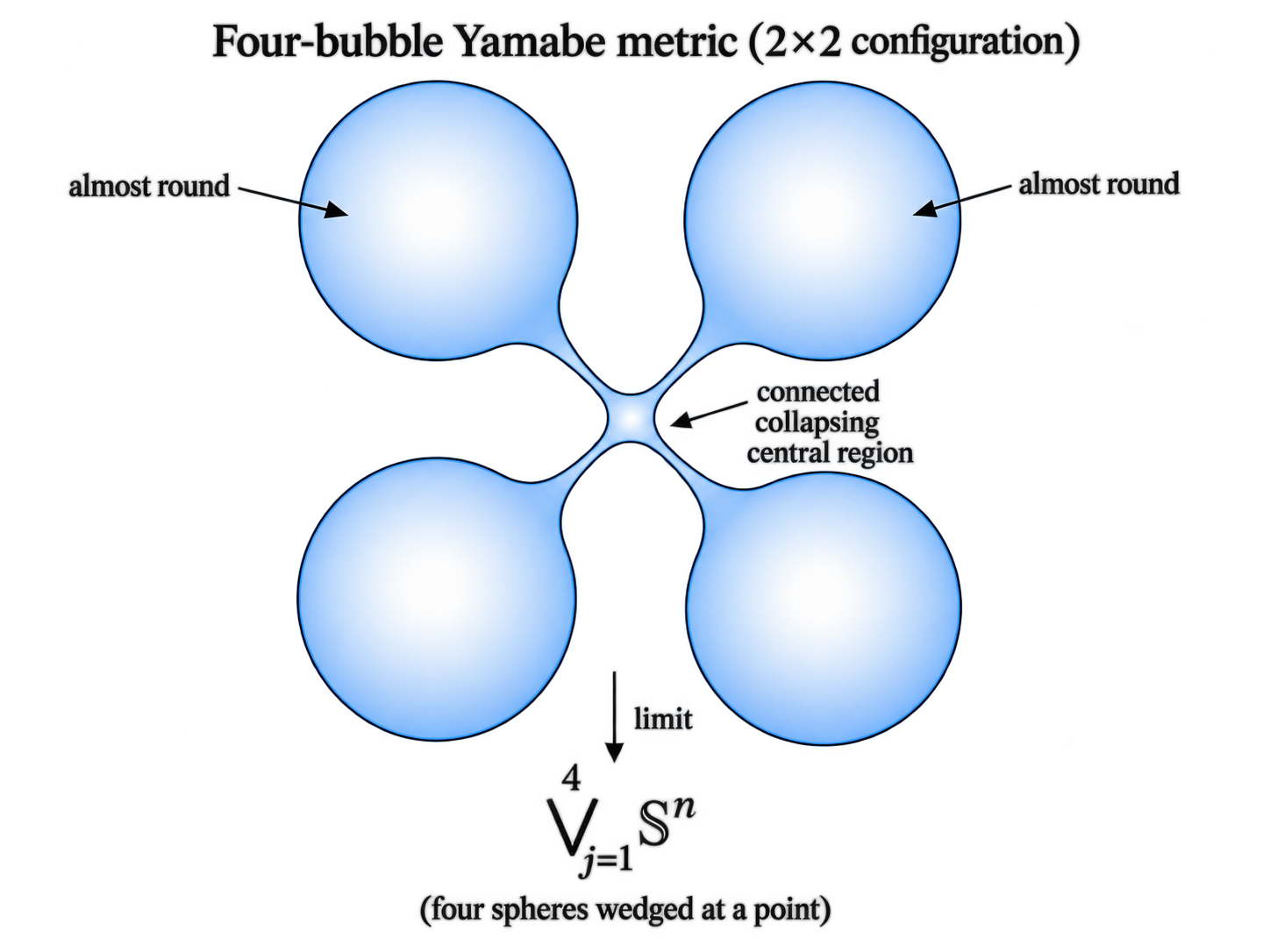}
    \caption{A schematic representation of an ancient Yamabe flow with two layers and two limiting bubble centers per layer at a given time $s \in (-\infty,0)$. 
    Upon normalization to a fixed volume, each bubble core converges to a round sphere, whereas the shared outer region collapses to the wedge point of $\bigvee_{j=1}^{hk}\Ss^n$.}
    \label{pic}
\end{figure}

\subsection{Parabolic formulation and analytic construction}\label{subsec:para}
Given a Yamabe flow $g(s)$ satisfying \eqref{YF-ori}, we write $g(s)=v(\cdot,s)^{4/(n-2)}g_{\Ss^n}$. The conformal scalar-curvature formula gives
\[
(v^p)_s=p(n-1)\(\Delta_{\Ss^n}v-\frac{n(n-2)}4v\).
\]
We set $\tau=p(n-1)s$ and define the Euclidean conformal factor by
\[
\bar v(x,\tau):=\(\frac{2}{1+|x|^2}\)^{\frac{n-2}{2}}
v\(\pi(x),\frac{\tau}{p(n-1)}\).
\]
Then $\pi^*g(\tau/[p(n-1)])=\bar v(\cdot,\tau)^{4/(n-2)}g_{\R^n}$ and
\be \label{pb1}
(\bar v^p)_\tau=\Delta\bar v,\qquad \bar v>0.
\ee
We define
\be \label{ubarvrel}
\tau = 1-e^{-t},\qquad
u(x,t) := (1-\tau)^{-\frac{1}{p-1}}\bar v(x,\tau).
\ee
Since $p/(p-1)=(n+2)/4$, equation \eqref{pb1} becomes
\be \label{pb}
p u^{p-1}\pp_tu=\Delta u+\frac{n+2}{4}u^p,\qquad u>0.
\ee
The bubbles \eqref{eq:bubble2} are stationary solutions of \eqref{pb}.

\begin{theorem}[Main analytic result]\label{thm:main-an}
Suppose that $n\ge3$.

\medskip \noindent \textup{\bf (1) Single layer.}
Let $k \ge 2$ be an integer. Then there exists a solution of \eqref{pb} on $\R^n\times(-\infty,0)$ of the form
\be \label{so1.2}
u(x,t)=\sum_{j=1}^k U_{\mu(t),\xi_j(t)}(x)+\Xi(x,t)
\ee
for $x = (x_1,x_2,\dots,x_n) \in \R^n$ and $t \in (-\infty,0)$. Here,
\be \label{eq:parasym0}
\mu(t)=\kappa|t|^{-\frac1{n-2}}\bigl(1+O(|t|^{-\frac{\ga}{n-2}})\bigr), \qquad d(t)=\sqrt{1-\mu(t)^2}, \qquad \xi_j(t)=d(t){\bf q}_j,
\ee
and
\be \label{eq:bfqs}
{\bf q}_j := (\omega_j,0_{\R^{n-2}}) = \big(e^{\frac{2\pi(j-1)}{k} i}, 0_{\R^{n-2}}\big) \in \Ss^1 \times \{0_{\R^{n-2}}\} \subset \Ss^{n-1} \subset \R^n.
\ee
Here, $\kappa>0$ and $\ga>0$ are constants, the asymptotic relation holds as $t\to-\infty$, and $P_{1,k}(0)=\{{\bf q}_j:j=1,\dots,k\}$.

\medskip \noindent \textup{\bf (2) Multiple layers.}
Let $h\ge2$ and $k\ge2$ be integers, and $-1<\vth_1^*<\dots<\vth_h^*<1$. Then there exists a solution of \eqref{pb} on $\R^n\times(-\infty,0)$ of the form
\[
u(x,t)=\sum_{l=1}^h\sum_{j=1}^k U_{\mu_l(t),\xi_{jl}(t)}(x)+\Xi(x,t)
\]
for $x = (x_1,x_2,\dots,x_n) \in \R^n$ and $t \in (-\infty,0)$. Here,
\be \label{eq:parasym1}
\begin{aligned}
\mu_l(t) &= \kappa_l|t|^{-\frac1{n-2}}\bigl(1+O(|t|^{-\frac{\ga}{n-2}})\bigr),\qquad d_l(t)=\sqrt{1-\mu_l(t)^2},\\
\xi_{jl}(t) &= d_l(t)\big(\sqrt{1-\vth_l(t)^2}\,\omega_j,\vth_l(t),0_{\R^{n-3}}\big).
\end{aligned}
\ee
The limiting point ${\bf q}_{jl}^* \in \Ss^{n-1}$ is defined in \eqref{eq:bfq}, and
\be \label{eq:parasym2}
\vth_l(t)=\vth_l^*+O\big(|t|^{-\frac{1+\ga}{n-2}}\big)
\ee
as $t\to-\infty$, for constants $\kappa_l>0$ and $\ga>0$.

\medskip
In both cases, writing $u_*$ for the displayed sum of bubbles, the correction satisfies
\[
\sup_{x\in\R^n}\frac{|\Xi(x,t)|}{u_*(x,t)} \to 0 \quad \text{as } t\to-\infty.
\]
In addition, the solution $u$ has the following properties:
\begin{enumerate}
\item[(i)] In the single-layer case, $\xi_j(t)\to{\bf q}_j\in P_{1,k}(0)$. In the multi-layer case, $\xi_{jl}(t)\to{\bf q}_{jl}^*\in P_{h,k}(\bvth^*)$.
\item[(ii)] For every $t<0$, $u(\cdot,t)$ is nonradial in $\R^n$, meaning that it is not invariant under the standard linear action of $O(n)$ on $\R^n$.
\item[(iii)] The solution is Kelvin invariant, namely,
\[
u(x,t)=|x|^{2-n}u\(\frac{x}{|x|^2},t\) \quad \text{for } (x,t) \in (\R^n \setminus \{0\}) \times (-\infty,0).
\]
\item[(iv)] Let $\hu$ be the conformal lift of $u$ to $\Ss^n\setminus\{\msN_{\infty}\}$ defined as
\be \label{stereo-proj-1}
\hu(\pi(y),t) := \(\frac{1+|y|^2}{2}\)^{\frac{n-2}{2}}u(y,t) \quad \text{for } y\in\R^n \text{ and } t<0.
\ee
It extends smoothly and positively across $\msN_{\infty}$. Reversing \eqref{ubarvrel} and setting $s=\tau/[p(n-1)]$ gives a compact ancient solution of \eqref{YF-ori} on $\Ss^n$.
\end{enumerate}
\end{theorem}

\begin{remark}
The following observations complement Theorem \ref{thm:main-an}.

\begin{enumerate}
\item[(i)] If $\msT$ is a M\"obius transformation of $\R^n\cup\{\infty\}$ and $\mst_0\in\R$, then
    \[
    u_{\msT,\mst_0}(x,t) := |\det D\msT(x)|^{\frac{n-2}{2n}} u\big(\msT(x),t-\mst_0\big)
    \]
    solves \eqref{pb} wherever it is defined. For the smoothly compactifiable solutions considered here, this induces \eqref{eq:geo-equiv-sphere} with $c=e^{-\mst_0}$ and $s_0=(1-e^{-\mst_0})/[p(n-1)]$, together with the conformal map induced by $\msT$.
    Thus a translation of $t$ by $\mst_0$ induces a geometric time change satisfying $s_0=(1-c)/[p(n-1)]$, whereas \eqref{eq:geo-equiv-sphere} allows $c>0$ and $s_0\in\R$ to be chosen independently.

\item[(ii)] The identity $\mu^2+|\xi|^2=1$ is precisely the condition under which the bubble $U_{\mu,\xi}$ is invariant under the Kelvin transformation \eqref{fke}. It determines the radial position of every center in Theorem \ref{thm:main-an}.

\item[(iii)] The energy limit in Theorem \ref{thm:main-geo}\textup{(ii)}, proved in Proposition \ref{prop:concentration}, also implies that the unit-volume normalization, after reparametrizing time, is an ancient solution of \eqref{YF-nori}. Indeed, setting $V(s):=\Vol_{g(s)}(\Ss^n)$, we have
    \be \label{eq:Vs}
    \frac{d}{ds}V(s)^{2/n}=-\mcy(g(s)).
    \ee
    Because $\mcy(g(s)) \to (hk)^{2/n}\mcy(g_{\Ss^n}) =: \mcy_\infty>0$ as $s\to-\infty$, integration from $s$ to a fixed $s_0<0$ yields
    \[
    V(s)^{2/n}=V(s_0)^{2/n}+\int_s^{s_0}\mcy(g(r))\,dr=\mcy_\infty|s|(1+o(1))
    \]
    as $s\to-\infty$. We define the new time variable and the unit-volume metric by
    \[
    \bs(s) := \int_{s_0}^s V(r)^{-2/n}\,dr \quad \text{and} \quad \bg(\bs(s)):=V(s)^{-2/n}g(s),
    \]
    respectively. Then $\Vol_{\bg(\bs)}(\Ss^n)=1$, and $\bg(\bs)$ satisfies the normalized Yamabe flow equation \eqref{YF-nori} with respect to the time variable $\bs$. Besides, $\bs(s)\to-\infty$ as $s\to-\infty$, since   
    $V(s)^{-2/n}=(1+o(1))(\mcy_\infty|s|)^{-1}$ as $s\to-\infty$. Hence, the flow $\bg$ is ancient.
\end{enumerate}
\end{remark}

\subsection{Proof strategy and relation to previous work}
The proof uses inner--outer gluing for the quasilinear parabolic equation \eqref{pb} in the full $n$-dimensional space $\R^n$. Related constructions and estimates can be found in \cite{CDM20,KME,KME2,DDW20,km}.
In particular, in \cite{km}, the second and third authors developed weighted integral and H\"older spaces adapted to the non-uniform coefficient of the time derivative arising in perturbed Yamabe flow.
The points specific to the ancient, unperturbed flow are as follows.

\medskip\noindent\textbf{(1) Parabolic balance of positive interactions and prescribed layers.}
All bubbles are positive. Their interactions are balanced by the time derivative, in contrast to the interactions between bubbles of opposite signs in the elliptic constructions \cite{dmpp,MM21,MMW19,LS26}.
The leading-order system for the scale constants in Proposition \ref{prop:inner1-mult} has a unique positive solution, obtained by minimizing a strictly convex function of the logarithms of these constants.
In the modulation system \eqref{ode2}, the logarithmic scale derivatives $\mu_l^{-1}\dot\mu_l$ are of order $|t|^{-1}$, whereas the right-hand sides of the layer-height equations in \eqref{modep} are integrable near $t=-\infty$.
Thus the limiting layer heights may be prescribed freely as in \eqref{eq:mcah}, without requiring them to satisfy an interaction equation.

\medskip\noindent\textbf{(2) Inner problems and Kelvin symmetry.}
For the inner problems in Section \ref{sec:inner}, the constant mode is fixed by a terminal condition at $t_0$, the first harmonics are removed by orthogonality, and the higher modes are obtained by integration from $-\infty$.
Kelvin symmetry requires additional care because $\mcz_{n+1}$ in \eqref{eq:mczn1} is a time-dependent combination of the kernel elements $Z_{n+1}$ and $Z_1$ in \eqref{eq:Zi}.

In Step 4 of the proof of Proposition \ref{prop:lin}, the Kelvin-induced isometry $\mck_t$ in \eqref{eq:lin-Kelvin-map} generally moves the north pole $\msN_\infty$, so the weighted norm $\|\cdot\|_{\wtshs,a+2,b}$ of Definition \ref{def:inner-norm} is not preserved.
We instead use uniqueness under the isometry-invariant time-weighted $L^1(\Ss^n)$ condition \eqref{eq:N1} to identify the projected first-harmonic term in \eqref{eq:bmc}.

\medskip\noindent\textbf{(3) Nonlinear closure and compactification.}
The outer operator $pu_*^{p-1}\pp_t-\Delta$ is not uniformly parabolic on $\R^n$. The outer problem is controlled by energy estimates, the maximum principle, and barriers, while local Schauder estimates are applied on the natural time scales in both the inner and outer regions.
For the outer problem, the estimates on a fixed spherical neighborhood of the north pole $\msN_{\infty}$ give the regularity across $\msN_{\infty}$ needed for smooth compactification;
in the $x$-variables, the same estimates yield the rapid decay of spatial derivatives through order two needed for curvature control.
Appendix \ref{app} controls the quasilinear term using the smallness of $u_*^{-1}\Xi$, including the time H\"older estimates on the moving charts.
For fixed parameter curves, the inner and outer problems are solved by the contraction mapping principle,
while the modulation parameters are determined in Proposition \ref{prop:inner1-mult} by applying the Schauder--Tychonoff fixed-point theorem in the local $C^1$ topology.

\medskip\noindent\textbf{(4) Geometric results.}
In Section \ref{sec:geo}, we establish geometric properties of the flows constructed in Theorem \ref{thm:main-an}, including concentration of the normalized volume measures at finitely many points, 
the existence of a nonflat scalar-flat outer limit, and collapse of the region outside the concentration neighborhoods.
Subsection \ref{subsec:equimod} proves Theorem \ref{thm:main-geo}\textup{(vii)} and provides a more precise description of the M\"obius-equivalence classes of concentration configurations, as outlined following Theorem \ref{thm:main-geo}.
The resulting flows can be geometrically inequivalent while sharing the same fixed-volume backward Gromov--Hausdorff limit described in Theorem \ref{thm:main-geo}\textup{(vi)}.

The symmetry-breaking and moduli results establish that the rotational families \cite{King93,ddks,dds} do not exhaust the finite-bubble ancient flows with bounded Yamabe quotient.
For related rigidity results and more general ancient solutions in Ricci and mean curvature flows, see \cite{DHS12,BDS21,BDNS23,DHS10,ADS20,BKN12,BK17}.

\subsection{Organization of the paper}
Section \ref{sec:sing} develops the single-layer construction in Theorem \ref{thm:main-an}, addressing the polygonal ansatz, weighted energy and H\"older spaces, the nonlinear inner and outer solvability statements, and the scale modulation equations.
Section \ref{sec:outer} studies the outer problem \eqref{outer} and proves Proposition \ref{prop:outer}, using energy methods, a maximum principle, barriers, and parabolic Schauder theory.
Section \ref{sec:inner} states and proves the linear inner solvability result in Proposition \ref{prop:lin}, using the spherical formulation and Euclidean Schauder estimates, 
and then proves nonlinear solvability in Proposition \ref{prop:nonlin} by a contraction argument.
Section \ref{sec:mult} treats the multi-layer construction in Theorem \ref{thm:main-an}, with particular attention to the coupled modulation equations for the scales and layer heights.
Section \ref{sec:geo} proves Theorem \ref{thm:main-geo}. Appendix \ref{app} proves the cutoff estimates, nonlinear estimates, and parameter continuity.
In Appendix \ref{app:quadratic-S4}, we construct a compact quadratic Type II ancient flow on $\Ss^4$ whose Yamabe quotient diverges backward in time and whose limiting conformal factor is singular along a circle.

\subsection{Notation and conventions}
Throughout the paper, $C>0$ denotes a constant that may change from line to line and may depend on the fixed dimension, polygonal data, auxiliary exponents, and parameter bounds, but not on $t$.
Unless explicitly indicated, $C$ is uniform for all sufficiently negative $t_0$.
We write $a\lesssim b$ if $a\le Cb$, $a\gtrsim b$ if $a\ge C^{-1}b$, $a\asymp b$ if both estimates hold, and $a\sim b$ if $a=b(1+o_t(1))$ as $t\to-\infty$.
The exponent $p=(n+2)/(n-2)$ is fixed throughout. In the flow transformations, the geometric time is $s$, the fast-diffusion time in \eqref{pb1} is $\tau=p(n-1)s$, and $t$ is the logarithmic time in \eqref{ubarvrel}.
Local integration variables and finite time endpoints are specified where they occur.
A dot denotes differentiation with respect to $t$, and $2^*:=2n/(n-2)=p+1$. We write $\Id_m$ for the $m\times m$ identity matrix.
For $y\in\R^n$, we define the Japanese bracket by $\la y\ra := (1+|y|^2)^{1/2}$.
For a set or condition $\Omega$, $\mone_\Omega$ denotes its characteristic function.
We write $B(x,r)$ for the open Euclidean ball of radius $r$ centered at $x$ and $B_{\Ss^n}(z,r)$ for the open geodesic ball of radius $r$ centered at $z$ in the round metric.
For subsets $A$ and $B$, we write $A\Subset B$ if $\overline A$ is a compact subset of $B$.
The symbols $\nabla_{\Ss^n}$, $\Delta_{\Ss^n}$, and $dv_{\Ss^n}$ denote the round gradient, Laplace--Beltrami operator, and volume form, respectively.
We denote by $d_{\Ss^n}(\tz_1,\tz_2)$ the geodesic distance between $\tz_1,\tz_2\in\Ss^n$, and abbreviate
\[
d_{\Ss^n}(z) := d_{\Ss^n}(\msN_{\infty},z)=\arccos(z_{n+1}) \quad \text{for } z=(z_1,\dots,z_{n+1})\in\Ss^n,
\]
where $\msN_{\infty}={\bf e}_{n+1}$ is the north pole.

\medskip \noindent \textbf{Acknowledgements.}
H. Chen and S. Kim were supported by the Basic Science Research Program through the National Research Foundation of Korea (NRF) funded by the Ministry of Science and ICT (RS2025-00558417).
S. Kim was also supported by the Samsung Science and Technology Foundation under Project Number SSTF-BA2601-01.

\medskip \noindent \textbf{Declaration on the use of AI-assisted tools.}
The human authors developed Theorems \ref{thm:main-geo} and \ref{thm:main-an}, performed the principal computations, and prepared the initial manuscript.
ChatGPT contributed to mathematical discussions, literature searches, selected calculations, verification of proofs, and revision of the exposition.
Specifically, ChatGPT generated Figure \ref{pic} and proposed the quadratic Type II ancient flow on $\Ss^4$, as detailed in Appendix \ref{app:quadratic-S4}.
The human authors independently reviewed and revised all AI-assisted material and retain responsibility for the mathematical accuracy, references, and the final submitted manuscript.

\section{Proof of Theorem \ref{thm:main-an}: Single-layer Case}\label{sec:sing}
\subsection{Setup}
We construct the approximate solution $u_*$ based on the following configuration:
For a fixed integer $k \ge 2$, consider a set of $k$ concentration points $\xi_j$ ($j=1,\dots,k$), arranged as the vertices of a regular polygon in a plane.
To describe these points, we decompose $x \in \R^n$ as $x = (\bx,x')$ with $\bx = (x_1,x_2) \in \R^2$. Writing the $\bx$-plane in complex notation, we set
\be \label{defpoints}
\xi_j(t) := d(t)\, {\bf q}_j \quad \text{and} \quad {\bf q}_j := \Big(e^{\frac{2\pi(j-1)}{k} i}, 0_{\R^{n-2}}\Big) \quad \text{for } j=1,\dots,k.
\ee
The scaling parameter $\mu(t)$ and the circumradius $d(t)$ of the regular polygon are of class $C^{1,\sigma/2}_{\loc}$ in time and satisfy
\be \label{mu}
\msM_0^{-1}|t|^{-\frac1{n-2}} \le \mu(t)\le \msM_0|t|^{-\frac1{n-2}}, \quad
\msM_0^{-1}|t|^{-\frac{n-1}{n-2}} \le \dot\mu(t)\le \msM_0|t|^{-\frac{n-1}{n-2}},
\ee
\be \label{dt}
\mu(t)^2+d(t)^2=1,\quad d(t)>0
\ee
for some fixed constant $\msM_0>1$. Throughout the paper, we write $\dot\mu=\pp_t\mu$. Under assumptions \eqref{mu}--\eqref{dt}, $d(t) \to 1$ as $t \to -\infty$. Hence, if we set 
\be \label{eq:delta0}
\delta_0 := \frac{1}{4}|{\bf q}_1-{\bf q}_2| = \frac{1}{2} \sin\frac{\pi}{k},
\ee
then we can find $t_0 < 0$ with $|t_0|$ large enough such that
\be \label{eq:delta01}
|\xi_{j_1}(t)-\xi_{j_2}(t)| \ge 2\delta_0 \quad \text{for all } 1 \le j_1 \ne j_2 \le k \text{ and } t \le t_0.
\ee

Recalling \eqref{eq:bubble}--\eqref{eq:bubble2} for the definition of $U_{\mu(t),\xi_j(t)}$, let $U_j(x,t)$ be the bubble
\[U_j (x,t) := U_{\mu(t),\xi_j(t)}(x).\]
The formal approximation $u_*$ for a solution to \eqref{pb} is defined by the superposition of these bubbles:
\be \label{approx*}
u_*(x,t) = u_* [\mu] (x,t) := \sum_{j=1}^k U_j (x,t) \quad \text{for } (x,t) \in \R^n \times (-\infty,t_0].
\ee
Here, $t_0<0$ is a large negative number to be chosen later. Rather than solving \eqref{pb} on $(-\infty,0]$, we will consider the equivalent problem on $(-\infty,t_0]$.
The two formulations are interchangeable, as one can be obtained from the other by a simple time translation. 

The symmetries of \eqref{pb} will play a crucial role. We consider a function $f$ even in the variables $x_3,\dots,x_n$ and invariant under a rotation of angle $\frac{2\pi}{k}$ in the $x_1x_2$-plane, that is,
\begin{align}
f(x_1, \dots, x_{\ell}, \dots, x_n, t) &= f(x_1, \dots, -x_{\ell}, \dots, x_n, t) \quad \text{for } \ell=3, \dots, n, \label{pp1} \\
f(e^{\frac{2\pi}{k} i} \bx, x', t) &= f(\bx, x', t). \label{pp2}
\end{align}
Also, we treat the Kelvin invariance of $f$ in the $x$-variable, expressed as
\be \label{fke}
f(x,t) = f^*(x,t) := \frac{1}{|x|^{n-2}} f \(\frac{x}{|x|^2}, t\).
\ee
In view of \eqref{defpoints} and the geometric constraint \eqref{dt}, the approximation $u_*$ satisfies the symmetry properties \eqref{pp1}--\eqref{fke}.

The error associated with this approximation is defined by the parabolic operator
\be \label{err}
E(x,t) = E[\mu,\dot{\mu}](x,t) := -\pp_t (u_*(x,t))^p + \Delta u_*(x,t) + \frac{n+2}{4} (u_*(x,t))^p
\ee
for $(x,t) \in \R^n \times (-\infty,t_0]$. Direct computation shows that $E(x,t)$ inherits the symmetry properties \eqref{pp1}--\eqref{pp2} and satisfies 
\[E(x,t) = \frac{1}{|x|^{n+2}}\, E\(\frac{x}{|x|^2}, t\).\]
Quantitative estimates for $E$ will be carried out in Subsection \ref{subsec:error}.

\medskip
For later use, we introduce functions
\be \label{eq:Zi}
\begin{aligned}
Z_i(y) &:= \pp_{y_i} U(y) = -(n-2)\msc_n \frac{y_i}{(1+|y|^2)^{\frac{n}{2}}}
\quad \text{for } i=1, \dots, n, \\
Z_{n+1} (y) &:= \frac{n-2}{2} U(y) + y \cdot \nabla U(y) = -\frac{(n-2)\msc_n}{2} \frac{|y|^2-1}{(1+|y|^2)^{\frac{n}{2}}}
\end{aligned}
\ee
for $y = (y_1,\dots,y_n) \in \R^n$. These functions span the solution space of the linear problem
\be \label{eq:Zieq}
\Delta z+\frac{n+2}{4} pU^{p-1}z = 0 \quad \text{in } \R^n, \quad z \in \dot{W}^{1,2}(\R^n),
\ee
where $\dot{W}^{1,2}(\R^n)$ stands for the standard homogeneous Sobolev space.

To facilitate the subsequent analysis, we also record the following elementary inequalities.
\begin{lemma}\label{yan}
Given $\vsi>0$, there exists a constant $C(\vsi)>0$, depending only on $\vsi$, such that for all $a>0$ and $b\in\R$, the following estimates hold:
\[
\big| |a+b|^\vsi - a^\vsi \big| \le 
\begin{cases}
C(\vsi)\min\{a^{\vsi-1}|b|, |b|^\vsi\} & \text{if } 0 < \vsi < 1, \\
C(\vsi)(a^{\vsi-1}|b| + |b|^\vsi) & \text{if } \vsi \ge 1
\end{cases}
\]
and
\[
\big| |a+b|^\vsi(a+b) - a^{\vsi+1} - (\vsi+1)a^\vsi b \big| \le 
\begin{cases}
C(\vsi)\min\{|b|^{\vsi+1}, a^{\vsi-1}b^2\} & \text{if } 0 < \vsi < 1, \\
C(\vsi)\max\{|b|^{\vsi+1}, a^{\vsi-1}b^2\} & \text{if } \vsi \ge 1.
\end{cases}
\]
\end{lemma}

\subsection{Functional framework}\label{subsec:norms}
Given the approximation $u_*$ defined in \eqref{approx*}, we seek a correction term $\Xi$ so that $u_* + \Xi$ solves \eqref{pb} and exhibits the qualitative features stated in Theorem \ref{thm:main-an}.
The construction of $\Xi$ is based on the inner--outer gluing scheme presented in Subsection \ref{scheme}.
In this subsection, we establish the functional-analytic framework needed to conduct the gluing procedure.

\subsubsection{Norms involving the time variable}
Let $t_0 < 0$ be a number of sufficiently large magnitude.
\begin{definition}
Let $\lambda: (-\infty,t_0] \to \R$ be a function. For any $b \ge 0$ and $\sigma \in (0,1)$, we set
\[
\|\lambda\|_{b;\sigma} = \|\lambda\|_{b;\sigma;t_0} := \sup_{t \in (-\infty,t_0]} |t|^{b} \left[ |\lambda(t)| + [\lambda]_{C^{\sigma/2}_t}(t) \right],
\]
where the local H\"older seminorm is given by
\[
[\lambda]_{C^{\sigma/2}_t}(t) := \sup \left\{ \frac{|\lambda(t_1) - \lambda(t_2)|}{|t_1-t_2|^{\sigma/2}} : \, t_1, t_2 \in (t, \min\{t_0, t+1\}), \, t_1 \ne t_2 \right\}.
\]
All empty suprema below are taken to be zero.
\end{definition}
\begin{definition}
Let $\Omega$ be a domain in $\R^n$ or $\Ss^n$, and let $\psi:\Omega\times(-\infty,t_0]\to\R$ be a function. For any $\sigma \in (0,1)$, we set a local time H\"older seminorm:
\be \label{eq:Holdert}
[\psi]_{C^{\sigma/2}_t}(x,t) := \sup\bigg\{\frac{|\psi(x,t_1)-\psi(x,t_2)|}{|t_1-t_2|^{\sigma/2}}: t_1, t_2 \in(t,\min\{t_0,t+1\}),\ t_1\ne t_2\bigg\}.
\ee
\end{definition}

\subsubsection{Norms for the outer problem}
We introduce several norms that will be used in the analysis of the outer problem \eqref{outer}. Recall the parameters $\mu>0$ and $\xi_1,\dots,\xi_k \in \R^n$ defined in \eqref{defpoints}--\eqref{dt}.
\begin{definition}[Auxiliary quantities]
Let $\alpha,\iota\in\R$, $\beta,\vrh>0$, and $\ell\ge0$. Let $\msR_0>10$ be a large constant. We set the weight functions by
\be \label{eq:omega}
\omega^1_{\alpha,\beta,\ell}(x,t) := \frac{|t|^{-\beta}\mu^{-\ell}} {(\min_{j=1,\dots,k} \la y_j\ra)^{\alpha+\ell}} \mone_{\{|x|\le\msR_0\}}, \qquad
\omega^2_{\vrh,\iota}(x,t) := \frac{|t|^{-\iota}}{|x|^{\vrh}} \mone_{\{|x|\ge\msR_0\}},
\ee
where $y_j := \mu^{-1}(x-\xi_j)$. We also set the space--time cylinders $\Lambda_\tau:=\R^n\times[\tau,\tau+1]$ for $\tau\le t_0-1$ and $\Omega_{t_0}:=\Omega\times(-\infty,t_0]$, where $\Omega$ is a domain in $\R^n$.

Let $\psi:\Omega_{t_0}\to\R$ be a function. For $\sigma \in (0,1)$, we set a local spatial H\"older seminorm: 
\[
[\psi]_{C^{\sigma}_{\Omega}}(x,t) := \sup\bigg\{\frac{|\psi(\tx_1,t)-\psi(\tx_2,t)|}{|\tx_1-\tx_2|^{\sigma}}: \tx_1,\tx_2\in B(x,r_0(x,t))\cap\Omega,\ \tx_1\ne\tx_2\bigg\}.
\]
For bounded $\Omega$, we set
\be \label{eq:r0}
r_0(x,t) := \frac14\min\left\{\delta_0,\min_{j=1,\dots,k} \sqrt{\mu(t)^2+|x-\xi_j(t)|^2}\right\}.
\ee
For $\Omega=\R^n$, we use the same formula when $|x|<\msR_0$ and set
$r_0(x,t)=|x|/2$ when $|x|\ge\msR_0$.
\end{definition}

\begin{definition}[Integral norms]
In the Sobolev-type spaces below we additionally require $\psi(\cdot,t)\in\dot W^{1,2}(\R^n)$ for almost every $t$, where $\dot W^{1,2}(\R^n)$ denotes the completion of $C_c^\infty(\R^n)$ in the gradient norm.

For a function $\psi: \R^n \times (-\infty,t_0] \to \R$, we define the local-in-time weighted $L^2$ and Sobolev-type norms:
\begin{align*}
\|\psi\|_{L^2(\Lambda_\tau)} &:= \(\int_\tau^{\tau +1}\!\!\int_{\R^n} u_*^{p-1} \psi^2 dx dt\)^{\frac{1}{2}}, \\
\|\psi\|_{H^1(\Lambda_\tau)} &:= \|\psi\|_{L^2(\Lambda_\tau)} + \Big\| u_*^{-\frac{p-1}{2}} \nabla \psi \Big\|_{L^2(\Lambda_\tau)}, \\
\|\psi\|_{H^2(\Lambda_\tau)} &:= \|\pp_t \psi\|_{L^2(\Lambda_\tau)} + \|\psi\|_{L^2(\Lambda_\tau)} + \Big\|u_*^{-\frac{p-1}{2}} \nabla \psi \Big\|_{L^2(\Lambda_\tau)} + \|u_*^{1-p} \Delta \psi\|_{L^2(\Lambda_\tau)}.
\end{align*}

Given a finite initial time $t_-<t_0-1$ and $\nu \in [0,1)$, we also define the global-in-time weighted $L^2$ and Sobolev-type norms:
\begin{align*}
\|\psi\|_{L^2_{\nu;t_-,t_0}} &:= \sup_{\tau\in[t_-,t_0-1]}|\tau|^\nu \|\psi\|_{L^2(\Lambda_\tau)},\\
\|\psi\|_{H^i_{\nu;t_-,t_0}} &:= \sup_{\tau\in[t_-,t_0-1]}|\tau|^\nu \|\psi\|_{H^i(\Lambda_\tau)} \quad \text{for } i=1,2.
\end{align*}
If the supremum is taken over $\tau\in(-\infty,t_0-1]$ instead, we simply write these norms as $\|\psi\|_{L^2_{\nu; t_0}}$ and $\|\psi\|_{H^i_{\nu; t_0}}$, respectively.
\end{definition}

\begin{definition}[Pointwise and H\"older norms]\label{defpr}
For functions $f,\, \psi:\Omega_{t_0}\to\R$, we define the global-in-time weighted $L^{\infty}$-norms:
\begin{align*}
\|f\|_{*,\alpha,\beta(\Omega_{t_0})} &:=\big\|(\mcw^*_{\alpha,\beta,2})^{-1}u_*^{p-1}f \big\|_{L^\infty(\Omega_{t_0})},\\
\|\psi\|_{**,\alpha,\beta(\Omega_{t_0})} &:=\big\|(\mcw^{**}_{\alpha,\beta,0})^{-1}\psi \big\|_{L^\infty(\Omega_{t_0})},
\end{align*}
where
\[
\mcw^*_{\alpha,\beta,\ell} :=\omega^1_{\alpha,\beta,\ell} +\omega^2_{n+\ell,\beta+\frac{\alpha}{n-2}} \quad \text{and} \quad
\mcw^{**}_{\alpha,\beta,\ell} :=\omega^1_{\alpha,\beta,\ell} +\omega^2_{n-2+\ell,\beta+\frac{\alpha}{n-2}}.
\]

Given $\sigma \in (0,1)$, we also define the global-in-time weighted H\"older norms:
\[
\begin{gathered}
\begin{aligned}
\|f\|_{*,\alpha,\beta;\sigma(\Omega_{t_0})} &:= \|f\|_{*,\alpha,\beta(\Omega_{t_0})} +\big\|(\mcw^*_{\alpha,\beta,2+\sigma})^{-1} [u_*^{p-1}f]_{C^\sigma_\Omega}\big\|_{L^\infty(\Omega_{t_0})}\\
&\ +\big\|(\mcw^*_{\alpha-\sigma,\beta,2})^{-1} [u_*^{p-1}f]_{C^{\sigma/2}_t}\big\|_{L^\infty(\Omega_{t_0})},
\end{aligned} \\
\begin{aligned}
\|\psi\|_{**,\alpha,\beta;\sigma(\Omega_{t_0})} &:= \sum_{\ell=0}^2\big\|(\mcw^{**}_{\alpha,\beta,\ell})^{-1} |\nabla_x^\ell\psi|\big\|_{L^\infty(\Omega_{t_0})}
+\sum_{\ell=0}^2\big\|(\mcw^{**}_{\alpha,\beta,\ell+\sigma})^{-1} [\nabla_x^\ell\psi]_{C^\sigma_\Omega}\big\|_{L^\infty(\Omega_{t_0})}\\
&\ +\sum_{\ell=0}^2\big\|(\mcw^{**}_{\alpha-\sigma,\beta,\ell})^{-1} [\nabla_x^\ell\psi]_{C^{\sigma/2}_t}\big\|_{L^\infty(\Omega_{t_0})} +\big\|(\mcw^{**}_{\alpha-2,\beta,0})^{-1}\pp_t\psi \big\|_{L^\infty(\Omega_{t_0})}\\
&\ +\big\|(\mcw^{**}_{\alpha-2,\beta,\sigma})^{-1} [\pp_t\psi]_{C^\sigma_\Omega}\big\|_{L^\infty(\Omega_{t_0})} +\big\|(\mcw^{**}_{\alpha-2-\sigma,\beta,0})^{-1} [\pp_t\psi]_{C^{\sigma/2}_t}\big\|_{L^\infty(\Omega_{t_0})}.
\end{aligned}
\end{gathered}
\]
We write $\|f\|_{*,\alpha,\beta}:=\|f\|_{*,\alpha,\beta(\R^n_{t_0})}$, $\|\psi\|_{**,\alpha,\beta} :=\|\psi\|_{**,\alpha,\beta(\R^n_{t_0})},$ $\|f\|_{*,\alpha,\beta;\sigma}:=\|f\|_{*,\alpha,\beta;\sigma(\R^n_{t_0})}$,
and $\|\psi\|_{**,\alpha,\beta;\sigma} :=\|\psi\|_{**,\alpha,\beta;\sigma(\R^n_{t_0})}$.
\end{definition}

\begin{definition}[Equivalence between $\R^n$ and $\Ss^n \setminus\{\msN_{\infty}\}$]
Let $\msN_{\infty}={\bf e}_{n+1} \in \R^{n+1}$ be the north pole of $\Ss^n$, and $\pi$ the map defined in \eqref{stereo-proj}. Then its inverse $\pi^{-1}: \Ss^n \setminus \{\msN_{\infty}\} \to \R^n$ is the stereographic projection given by
\[
y=\pi^{-1}(z) := \frac{\bz}{1-z_{n+1}} \quad \text{for } z=(\bz,z_{n+1})=(z_1,\dots,z_n,z_{n+1}).
\]

Let $\mcb^{\ep} := B_{\Ss^n}(\msN_{\infty},\ep)$. We fix a large number $\msR_0 > 10$ and a small number $\ep_0 \in (0,\frac{\pi}{4})$ such that
\[
\pi(B(0,\tfrac{\msR_0}{2})) \cup \mcb^{\ep_0/2} = \Ss^n \quad \text{and} \quad B(0,\tfrac{\msR_0}{2}) \cup \pi^{-1}(\mcb^{\ep_0/2}\setminus\{\msN_{\infty}\}) = \R^n.
\] 

Recalling the conformal lift $\hf$ of a function $f: \R^n \times (-\infty,t_0] \to \R$ in \eqref{stereo-proj-1}, we also set
\[
[\hf]_{C^\sigma_{\mcb^\ep}}(z,t) := \sup\bigg\{\frac{|\hf(\tz_1,t)-\hf(\tz_2,t)|}{d_{\Ss^n}(\tz_1,\tz_2)^\sigma}: \tz_1,\tz_2\in B_{\Ss^n}(z,\ep_0)\cap\mcb^\ep,\ \tz_1\ne\tz_2\bigg\}
\]
for any $\sigma \in (0,1)$ and $(z,t) \in \mcb^{\ep}_{t_0} := \mcb^{\ep} \times (-\infty,t_0]$.
For tensor-valued derivatives, we compute the spatial H\"older seminorms componentwise in the fixed normal coordinates centered at $\msN_\infty$.
\end{definition}

\begin{definition}[Refined pointwise and H\"older norms]\label{def:norm-refined}
For $\sigma \in (0,1)$, let
\begin{align*}
\|\hf\|_{*',\alpha,\beta(\mcb^\ep_{t_0})} &:= \big\||t|^{\beta+\frac{\alpha-2}{n-2}}\hf \big\|_{L^\infty(\mcb^\ep_{t_0})},\\
\|\hps\|_{**',\alpha,\beta(\mcb^\ep_{t_0})} &:= \big\||t|^{\beta+\frac{\alpha}{n-2}}\hps \big\|_{L^\infty(\mcb^\ep_{t_0})},\\
\|\hf\|_{*',\alpha,\beta;\sigma(\mcb^\ep_{t_0})} &:= \left\||t|^{\beta+\frac{\alpha-2}{n-2}} \(|\hf| + [\hf]_{C^\sigma_{\mcb^\ep}} + |t|^{-\frac\sigma{n-2}}[\hf]_{C^{\sigma/2}_t}\)\right\|_{L^\infty(\mcb^\ep_{t_0})},\\
\|\hps\|_{**',\alpha,\beta;\sigma(\mcb^\ep_{t_0})}
&:= \left\||t|^{\beta+\frac\alpha{n-2}}
\sum_{\ell=0}^2\(\begin{aligned}
|\nabla_{\Ss^n}^\ell\hps|
&+[\nabla_{\Ss^n}^\ell\hps]_{C^\sigma_{\mcb^\ep}}\\
&+|t|^{-\frac\sigma{n-2}}
[\nabla_{\Ss^n}^\ell\hps]_{C^{\sigma/2}_t}
\end{aligned}\)\right\|_{L^\infty(\mcb^\ep_{t_0})}\\
&\ +\left\||t|^{\beta+\frac{\alpha-2}{n-2}}
\(\begin{aligned}
|\pp_t\hps|&+[\pp_t\hps]_{C^\sigma_{\mcb^\ep}}\\
& +|t|^{-\frac\sigma{n-2}}[\pp_t\hps]_{C^{\sigma/2}_t}
\end{aligned}\)\right\|_{L^\infty(\mcb^\ep_{t_0})}.
\end{align*}

Then, we define
\begin{align*}
\|f\|_{\tilde *,\alpha,\beta;t_0} &:= \|f\|_{*,\alpha,\beta(B(0,\msR_0)_{t_0})} + \|\hf\|_{*',\alpha,\beta(\mcb^{\ep_0}_{t_0})},\\
\|\psi\|_{\widetilde{**},\alpha,\beta;t_0} &:= \|\psi\|_{**,\alpha,\beta(B(0,\msR_0)_{t_0})} + \|\hps\|_{**',\alpha,\beta(\mcb^{\ep_0}_{t_0})},\\
\|f\|_{\tilde *,\alpha,\beta;\sigma;t_0} &:= \|f\|_{*,\alpha,\beta;\sigma(B(0,\msR_0)_{t_0})} + \|\hf\|_{*',\alpha,\beta;\sigma(\mcb^{\ep_0}_{t_0})},\\
\|\psi\|_{\widetilde{**},\alpha,\beta;\sigma;t_0} &:= \|\psi\|_{**,\alpha,\beta;\sigma(B(0,\msR_0)_{t_0})} + \|\hps\|_{**',\alpha,\beta;\sigma(\mcb^{\ep_0}_{t_0})}.
\end{align*}
We suppress the subscript $t_0$ when no confusion can arise.
\end{definition}
\begin{remark}
\leavevmode

\begin{enumerate}
\item[(i)] The refined pointwise norms, namely, the $\tilde{*}$- and $\widetilde{**}$-norms, are designed to capture the rapid decay, for each $t$,
    of the derivatives of the solution $\psi(x,t)$ to the outer problem \eqref{outer} as $|x| \to \infty$; see Remark \ref{rmk:psitdecay}.
\item[(ii)] In the definition of the $\tilde{*}$- and $\widetilde{**}$-norms, the $*$- and $**$-norms are applied only in the spatially bounded region $B(0,\msR_0)_{t_0}$ for each fixed $t_0$,
    while the decay of the relevant functions as $|x|\to\infty$ is instead tracked by the $*'$- and $**'$-norms of their conformal lifts.
    Therefore, the weight $\omega^2_{\varrho,\iota}$ in Definition \ref{defpr} is not strictly necessary.
    Nevertheless, we retain $\omega^2_{\varrho,\iota}$ in the $*$- and $**$-norms to explicitly encode spatial decay in $\R^n$, which is not directly visible from the $*'$- and $**'$-norms. It holds that
    \[
    \|\psi\|_{**,\alpha,\beta;\sigma}\le C\|\psi\|_{\widetilde{**},\alpha,\beta;\sigma} \quad \text{and} \quad \|f\|_{*,\alpha,\beta;\sigma}\le C\|f\|_{\tilde *, \alpha,\beta;\sigma}.
    \]
\end{enumerate}
\end{remark}

\subsubsection{Norms for the inner problem}
We introduce the norms needed for the analysis of the inner problem \eqref{inner0}.
\begin{definition}[Auxiliary quantities]\label{def:inner-radius}
For a domain $\Omega\subset\R^n$ and numbers $t_- < t_0$, we write $\Omega_{t_-,t_0}:=\Omega\times[t_-,t_0]$.

Let $\phi:\Omega_{t_0}\to\R$ be a function. For $\sigma \in (0,1)$, we set a local spatial H\"older seminorm:
\[
[\phi]_{C^\sigma_\Omega}(y,t) = \sup\bigg\{\frac{|\phi(\ty_1,t)-\phi(\ty_2,t)|} {|\ty_1-\ty_2|^\sigma}: \ty_1,\ty_2\in B(y,r_1(y))\cap\Omega,\ \ty_1\ne\ty_2\bigg\},
\]
where $r_1(y):=\la y\ra/2=(1+|y|^2)^{1/2}/2$.
\end{definition}

\begin{definition}[Pointwise and H\"older norms]\label{def:inner-norm}
We fix $a \in (0, n-2)$, $b > 0$, and $t_- < \frac{3t_0}{2}$.

For functions $\mch:\R^n_{t_-,t_0}\to\R$ and $\phi:\Omega_{t_-,t_0}\to\R$, we define the global-in-time weighted $L^{\infty}$-norms:
\begin{align*}
\|\mch\|_{\sharp,a+2,b; t_-,t_0} &:= \left\| |t|^b \la y\ra^{a+2} U^{p-1}\mch \right\|_{L^{\infty}(\R^n_{t_-,t_0})}, \\
\|\phi\|_{\sharp\sharp,a,b(\Omega_{t_-,t_0})} &:= \left\| |t|^b \la y\ra^a \phi \right\|_{L^{\infty}(\Omega_{t_-,t_0})}.
\end{align*}
If $\R^n_{t_-,t_0}$ and $\Omega_{t_-,t_0}$ are replaced by $\R^n_{t_0}$ and $\Omega_{t_0}$, respectively, we simply write these norms as $\|\mch\|_{\sharp,a+2,b} = \|\mch\|_{\sharp,a+2,b;t_0}$ and $\|\phi\|_{\sharp\sharp,a,b(\Omega_{t_0})}$, respectively.

Given $\sigma \in (0,1)$, we also define the global-in-time weighted H\"older norms:
\begin{align*}
\|\mch\|_{\sharp,a+2,b;\sigma} &:= \|\mch\|_{\sharp,a+2,b;t_0} + \left\| |t|^b \la y\ra^{a+2+\sigma} [U^{p-1} \mch]_{C^{\sigma}_{\R^n}} \right\|_{L^{\infty} (\R^n_{t_0})}\\
&\ + \left\| |t|^b \la y\ra^{a+2-\sigma} [U^{p-1} \mch]_{C^{\sigma/2}_t} \right\|_{L^{\infty} (\R^n_{t_0})},
\end{align*}
\begin{align*}
\|\phi\|_{\sharp\sharp,a,b;\sigma(\Omega_{t_0})} &:= \sum_{\ell=0}^2\left\||t|^b \times
\left\{\begin{aligned}
& \la y\ra^{a+\ell}|\nabla_y^\ell\phi| + \la y\ra^{a+\ell+\sigma} [\nabla_y^\ell\phi]_{C^\sigma_\Omega}\\
&\hspace{60pt} + \la y\ra^{a+\ell-\sigma} [\nabla_y^\ell\phi]_{C^{\sigma/2}_t}
\end{aligned}\right\}\right\|_{L^\infty(\Omega_{t_0})}\\
&\ + \left\||t|^b \times
\left\{\begin{aligned}
& \la y\ra^{a-2}|\pp_t\phi| + \la y\ra^{a-2+\sigma} [\pp_t\phi]_{C^\sigma_\Omega}\\
&\hspace{55pt} + \la y\ra^{a-2-\sigma} [\pp_t\phi]_{C^{\sigma/2}_t}
\end{aligned}\right\}\right\|_{L^\infty(\Omega_{t_0})}.
\end{align*}
We write $\|\phi\|_{\sharp\sharp,a,b;\sigma} = \|\phi\|_{\sharp\sharp,a,b;\sigma;t_0} = \|\phi\|_{\sharp\sharp,a,b;\sigma(\R^n_{t_0})}$.
\end{definition}

\subsection{Estimates of the errors}\label{subsec:error}
We now estimate the error $E$ in \eqref{err}. By using the asymptotic decay of the bubble, we obtain the exterior estimate:
If $x \in \R^n$ satisfies $|x-\xi_j|>\delta_0$ for all $j=1,\dots,k$, then
\be \label{est-error-out-0}
|E(x,t)| \le C u_*^{p-1} \frac{\mu^{\frac{n-2}{2}}}{1+|x|^{n-2}} \quad \text{for all } t \le t_0.
\ee
By \eqref{pp2}, it remains to estimate $E$ in the region $|x-\xi_1|<\delta_0$:
\begin{lemma}[Error estimates] \label{size error}
Suppose that the parameters $\mu = \mu(t)$, $\dot{\mu} = \dot{\mu}(t)$, and $d = d(t)$ satisfy \eqref{mu}--\eqref{dt} for all $t \in (-\infty,t_0]$, and the points $\xi_j = \xi_j(t)$ for $j=1,\dots,k$ are defined in \eqref{defpoints}.
Let $y=\mu^{-1}(x-\xi_1)$. Then, for $|y|\le \mu^{-1}\delta_0$ and negative $t$ of sufficiently large magnitude, the following estimate holds:
\be \label{est-error-in-0}
\begin{aligned}
\mu^{\frac{n+2}{2}}E(\mu y+\xi_1,t) &= p\mu^{-1}\dot\mu(U^{p-1}Z_{n+1})(y)+\frac{n+2}{4}pA_k\mu^{n-2}U^{p-1}(y) \\
&\ +\mu^{\frac{n+2}{2}} \big(\wte^1[\mu,\dot\mu](y,t)+\wte^2[\mu](y,t)\big). \end{aligned}
\ee
Here,
\be \label{eq:Ak}
A_k :=\sum_{j=2}^k \frac{\msc_n}{|(1,0,0_{\R^{n-2}}) - {\bf q}_j|^{n-2}} = \sum_{j=2}^k \frac{\msc_n}{2^{\frac{n-2}{2}} \big(1-\cos \frac{2\pi}{k}(j-1)\big)^{\frac{n-2}{2}}} > 0
\ee
and the remainders $\wte^1 = \wte^1[\mu,\dot{\mu}]$ and $\wte^2 = \wte^2[\mu]$ (whose precise definitions are given in \eqref{tildeE1} and \eqref{tildeE2_def}, respectively) satisfy the bounds
\be \label{boe1}
\mu^{\frac{n+2}{2}} |\wte^1|(y,t) \le C\mu^{n-1}\la y\ra^{-4},
\ee
\be \label{boe2}
\begin{aligned}
\mu^{\frac{n+2}{2}}|\wte^2|(y,t) &\le C\Big[\mu^{2(n-2)}U^{p-2}(y)\mone_{\{3\le n\le5\}}+\mu^{n+2} + U^{p-1}(y)\big\{\mu^n(1+|y|^2)+\mu^{n-1}|y|\big\}\Big]. 
\end{aligned}
\ee
\end{lemma}
\begin{proof}
We decompose the error $E$ into two parts, $E=E^1+E^2$, where
\[
E^1 := -p \bigg(\sum_{j=1}^kU_j\bigg)^{p-1} \sum_{j=1}^k\pp_tU_j, \quad
E^2 := \frac{n+2}{4} \bigg[\bigg(\sum_{j=1}^kU_j\bigg)^p - \sum_{j=1}^kU_j^p\bigg].
\]

\medskip
We note that the time derivative of a bubble $U_j$ is given by
\be \label{partial-t-ustar}
\pp_tU_j = -\mu^{-\frac n2} \left[\dot\mu Z_{n+1}\!\(\frac{x-\xi_j}{\mu}\) +(\nabla U)\!\(\frac{x-\xi_j}{\mu}\)\cdot\dot\xi_j\right], \quad \dot\xi_j=-\frac{\mu\dot\mu}{d}{\bf q}_j,
\ee
where the second relation follows from \eqref{defpoints} and \eqref{dt}. Substitution into $E^1$ gives the leading time-derivative term in
\eqref{est-error-in-0}, with
\be \label{tildeE1}
\begin{aligned}
&\wte^1(y,t) := p \mu^{-\frac{n+4}{2}} \dot \mu \, \bigg[\bigg(U(y)+\sum_{j=2}^k U\Big(y+\frac{\xi_1-\xi_j}{\mu}\Big)\bigg)^{p-1}-U^{p-1}(y)\bigg] Z_{n+1}(y) \\
&\quad -p \frac{\mu^{-\frac{n+2}{2}} \dot \mu}{d} \bigg(U(y)+\sum_{j=2}^kU\Big(y+\frac{\xi_1-\xi_j}{\mu}\Big)\bigg)^{p-1}\times\nabla U(y)\cdot{\bf q}_1 \\
&\quad + p \mu^{-\frac{n+4}{2}} \dot \mu \bigg(U(y)+\sum_{j=2}^kU\Big(y+\frac{\xi_1-\xi_j}{\mu}\Big)\bigg)^{p-1}\times \sum_{j=2}^kZ_{n+1}\Big(y+\frac{\xi_1-\xi_j}{\mu}\Big) \\
&\quad - p \frac{\mu^{-\frac{n+2}{2}} \dot \mu}{d} \bigg(U(y)+\sum_{j=2}^kU\Big(y+\frac{\xi_1-\xi_j}{\mu}\Big)\bigg)^{p-1}\times\sum_{j=2}^k (\nabla U)\Big(y+\frac{\xi_1-\xi_j}{\mu}\Big)\cdot{\bf q}_j.
\end{aligned}
\ee
For $|y|\le \mu^{-1}\delta_0$, we see from \eqref{eq:delta01} that $|\mu y+\xi_1-\xi_j| \ge |\xi_1-\xi_j|-\mu|y| \ge \delta_0$ for $j=2,\dots,k$.
Hence, we have that $\sum_{j=2}^k U(y+\mu^{-1}(\xi_1-\xi_j)) \le C\mu^{n-2} \le CU(y)$,
and the shifted $Z_{n+1}$ and $\nabla U$ have the same upper bound.
Also, $|Z_{n+1}|+|\nabla U|\le CU$,
$\mu^{-1}|\dot\mu|\le C\mu^{n-2}$, and $\mu^{-1}|\dot\xi_j|\le C\mu^{n-1}$.
The mean-value formula applied to the first line of \eqref{tildeE1}
and direct estimation of the other lines yield
\[
\mu^{\frac{n+2}{2}}|\wte^1| \le C(\mu^{2(n-2)}+\mu^{n-1})U^{p-1} \le C\mu^{n-1}\la y\ra^{-4},
\]
which is \eqref{boe1}.

\medskip
To analyze $E^2$, we consider
\begin{align}
U\bigg(y+\frac{\xi_1-\xi_j}{\mu}\bigg) &= \msc_n\frac{\mu^{n-2}}{|\xi_1-\xi_j|^{n-2}} \bigg[1-(n-2)\mu \frac{(\xi_1-\xi_j)\cdot y}{|\xi_1-\xi_j|^2} \nonumber \\
&\ +\mu^2 \bigg\{-\frac{n-2}{2}\frac{|y|^2}{|\xi_1-\xi_j|^2}+\frac{(n-2)n}{2}\frac{(y\cdot(\xi_1-\xi_j))^2}{|\xi_1-\xi_j|^4} \label{uijep} \\
&\hspace{2.5cm} - \frac{n-2}{2}\frac1{|\xi_1-\xi_j|^2}\bigg\} + O\big(\mu^3(1+|y|^3)\big)\bigg] \nonumber
\end{align}
for $|y|\le\mu^{-1}\delta_0$. We define the remainder term $\wte^2(y,t)$ by
\be \label{tildeE2_def}
\begin{aligned}
\frac4{n+2}\wte^2(y,t) &:= p\mu^{-\frac{n+2}{2}}U^{p-1}(y) \sum_{j=2}^kU\!\(y+\frac{\xi_1-\xi_j}{\mu}\) -pA_k\mu^{\frac n2-3}U^{p-1}(y) \\
&\ -\mu^{-\frac{n+2}{2}} \sum_{j=2}^k U^p\!\(y+\frac{\xi_1-\xi_j}{\mu}\) \\
&\ + \mu^{-\frac{n+2}{2}} \bigg[\bigg(U(y)+\sum_{j=2}^kU\!\(y+\frac{\xi_1-\xi_j}{\mu}\)\bigg)^p - U^p(y) \\
&\ \hspace{110pt} -pU^{p-1}(y)\sum_{j=2}^kU\!\(y+\frac{\xi_1-\xi_j}{\mu}\)\bigg]. 
\end{aligned}
\ee
Direct application of Lemma \ref{yan}, \eqref{uijep}, and \eqref{eq:delta01} leads to
\begin{align*}
\mu^{\frac{n+2}{2}}|\wte^2(y,t)| &\le C\bigg(\frac{\mu^n}{1+|y|^4}+\frac{\mu^{n-1}|y|}{1+|y|^4}+\left|\frac1{d^{n-2}}-1\right|\mu^{n-2}U^{p-1}+\mu^{n+2}\bigg)\\
&\ +\begin{cases}
C\max\{\mu^{n+2}, U^{p-2} \mu^{2(n-2)}\} &\text{if } n=3,4,5, \\
C\min\{\mu^{n+2}, U^{p-2} \mu^{2(n-2)}\} &\text{if } n \ge 6.
\end{cases}
\end{align*}
From this, the bound in \eqref{boe2} immediately follows. 
\end{proof}

After increasing $\msM_0>1$ if necessary, we also assume the temporal H\"older bound
\be \label{mu2}
\|\mu\|_{\frac1{n-2};\sigma} +\|\dot\mu\|_{\frac{n-1}{n-2};\sigma} \le \msM_0.
\ee

We estimate the integral of $\mu^{\frac{n+2}{2}}E(\mu y+\xi_1,t)$ against the scaling kernel element $Z_{n+1}$ defined in \eqref{eq:Zi}.

\begin{lemma}[Projected error estimates] \label{le2.3}
Suppose that the parameters $\mu=\mu(t)$, $\dot\mu=\dot\mu(t)$, and $d=d(t)$ satisfy \eqref{mu}--\eqref{dt} for all $t\in(-\infty,t_0]$, that the points $\xi_j=\xi_j(t)$ ($j=1,\dots,k$) are defined by \eqref{defpoints}, and that \eqref{mu2} holds.
Let $R>0$ be fixed and sufficiently large. After taking $|t_0|$ sufficiently large, we may assume that $40R\mu(t)\le\delta_0$ for all $t\le t_0$. Then
\begin{multline}\label{projection erro}
\mu^{\frac{n+2}{2}} \int_{B(0,R)}E(\mu y+\xi_1,t)Z_{n+1}(y)\,dy = pa_1\mu^{-1}\dot\mu\big(1+O(R^{-n})\big) \\
-\frac{n+2}{4}pa_2A_k\mu^{n-2}\big(1+O(R^{-2})\big) +\mu^{n-1}G_{n+1}[\mu,\dot\mu](t)
\end{multline}
uniformly for $t\in(-\infty,t_0]$, where
\be \label{eq:a1a2}
a_1:=\int_{\R^n}U^{p-1}Z_{n+1}^2\,dy>0, \qquad a_2:=-\int_{\R^n}U^{p-1}Z_{n+1}\,dy>0.
\ee
For this fixed $R$, we define
\[
G_{n+1}[\mu,\dot\mu](t) := \mu^{1-n}\int_{B(0,R)} \mu^{\frac{n+2}{2}}\Big[ \wte^1[\mu,\dot\mu](y,t)+\wte^2[\mu](y,t) \Big] Z_{n+1}(y)\,dy.
\]
Then $\|G_{n+1}[\mu,\dot\mu]\|_{0;\sigma}\le C_R$, where $C_R$ may also depend on the fixed parameter bounds.
\end{lemma}
\begin{proof}
The proof relies on the estimate \eqref{est-error-in-0}.

Since $U^{p-1}(y)=O(|y|^{-4})$ and $Z_{n+1}(y)=O(|y|^{2-n})$ as $|y|\to\infty$, we have
\begin{multline*}
\int_{B(0,R)} \left[p\mu^{-1}\dot\mu\,U^{p-1}Z_{n+1} +\frac{n+2}{4}pA_k\mu^{n-2}U^{p-1}\right]Z_{n+1}\,dy \\
= pa_1\mu^{-1}\dot\mu \big(1+O(R^{-n})\big) -\frac{n+2}{4}pa_2A_k\mu^{n-2} \big(1+O(R^{-2})\big).
\end{multline*}

For fixed $R$, the estimates \eqref{boe1}--\eqref{boe2} give
\[
\mu^{\frac{n+2}{2}} \left|\int_{B(0,R)}\big(\wte^1+\wte^2\big) Z_{n+1}\,dy\right| \le C_R\mu^{n-1}.
\]
Hence, $|G_{n+1}[\mu,\dot\mu](t)|\le C_R$.

Taking local time differences in \eqref{tildeE1} and \eqref{tildeE2_def}, and applying \eqref{mu} and \eqref{mu2}, we obtain $\|G_{n+1}[\mu,\dot\mu]\|_{0;\sigma}\le C_R$. This proves \eqref{projection erro}.
\end{proof}

\subsection{The inner--outer gluing scheme}\label{scheme}
To prove Theorem \ref{thm:main-an} with $h=1$, we seek a solution to \eqref{pb} in the form
\[
u(x,t) = u_*[\mu](x,t) + \Xi[\mu, \dot{\mu}](x,t),
\]
where $\Xi$ denotes the correction term. Substituting this expression into \eqref{pb}, we find the equation for $\Xi$:
\be \label{nl-p}
-p u_*^{p-1} \pp_t \Xi + \Delta \Xi + p \, \frac{n+2}{4} \, u_*^{p-1} \Xi + E + N(\Xi) = 0 \quad \text{in } \R^n \times (-\infty,t_0),
\ee
where $E$ is the error term introduced in \eqref{err} and $N(\Xi)$ is the nonlinear remainder term given as
\be \label{N}
\begin{aligned}
N(\Xi) &:= -p \left[ (u_*+\Xi)^{p-1} - u_*^{p-1} \right] \pp_t u_* - p \left[ (u_*+\Xi)^{p-1} - u_*^{p-1} \right] \pp_t \Xi \\
&\ + \frac{n+2}{4} \left[ (u_*+\Xi)^p - u_*^p - p u_*^{p-1} \Xi \right]. \end{aligned}
\ee

\medskip
Let $\eta:[0,\infty) \to [0,1]$ be a smooth cutoff function such that $\eta(r) = 1$ for $r < 1$ and $\eta(r) = 0$ for $r > 2$. For $j=1,\dots,k$, set
\[
D_j(x,t) := \left[\frac{2\big(\mu(t)^2+|x-\xi_j(t)|^2\big)}{1+|x|^2}\right]^{\frac{1}{2}}.
\]
We define
\be \label{def-cutoff-j}
\eta_j(x,t) := \eta\(\frac{D_j(x,t)}{R(t)\mu(t)}\), \quad \text{and} \quad
\eta_{j,2}(x,t) := \eta\(\frac{D_j(x,t)}{2R(t)\mu(t)}\),
\ee
where
\be \label{eq:Rt}
R(t) := |t|^{c_0},\qquad 0<c_0<\frac1{n-2}.
\ee
Then $R(t)\to\infty$ and $R(t)\mu(t)\to0$ as $t\to-\infty$. We take $|t_0|$ large enough that $R(t)\ge4$ and $40R(t)\mu(t)\le\delta_0$ for all $t \in (-\infty,t_0]$. Since $\mu^2+|\xi_j|^2=1$,
\[
D_j^2(x,t)=2-\frac{4\xi_j(t)\cdot x}{1+|x|^2}.
\]
Consequently, $D_j$ and both cutoff functions in \eqref{def-cutoff-j} are smooth in $x$, and $\eta_j(x,t)=\eta_j(x/|x|^2,t)$ for $x\ne0$, with the same identity for $\eta_{j,2}$. Also, $\eta_j\eta_{j,2}=\eta_j$.
If $\mscR$ denotes rotation through $2\pi/k$ in the $(x_1,x_2)$-plane, then, with indices understood modulo $k$,
\[
\eta_{j+1}(\mscR x,t)=\eta_j(x,t) \quad \text{and} \quad \eta_{j+1,2}(\mscR x,t)=\eta_{j,2}(x,t).
\]
Reflection across the hyperplane ${x_2=0}$ permutes the families $\{\eta_j\}_{j=1}^k$ and $\{\eta_{j,2}\}_{j=1}^k$, while each $\eta_j$ and $\eta_{j,2}$ is even in $x_3,\dots,x_n$.
Thus their sums $\sum_{j=1}^k \eta_j$ and $\sum_{j=1}^k \eta_{j,2}$ have all the polygonal symmetries in \eqref{pp1}--\eqref{pp2}.
The proof of the following lemma is postponed to Appendix \ref{app1}.
\begin{lemma}[Properties of the cutoff functions]\label{lemma:eta_j}
Take $t_0<0$ sufficiently large in magnitude, if necessary. Then, for each $j=1,\dots,k$ and $t\in(-\infty,t_0]$,
\be \label{eq:etaj}
\eta_j(\mu(t)y+\xi_j(t),t) = \begin{cases}
1 &\text{for } |y|\le \tfrac14R(t),\\
0 &\text{for } |y|\ge 5R(t).
\end{cases}
\ee
In addition,
\begin{align*}
|\nabla_x^{\ell}\eta_j(x,t)| &\le C(R(t)\mu(t))^{-\ell} \quad \text{for } \ell=1,2, \\
|\pp_t\eta_j(x,t)| &\le C\(\frac{|\dot R(t)|}{R(t)} + \frac{|\dot\mu(t)|}{\mu(t)} + \frac{|\dot\xi_j(t)|}{R(t)\mu(t)}\) \le C|t|^{-1}.
\end{align*}
\end{lemma}
The construction of $\Xi$ relies on an inner--outer gluing procedure: We search for the correction term $\Xi$ in the form
\be \label{eq:Xi}
\Xi(x,t) = \sum_{j=1}^k\eta_j(x,t)\,\tph_j(x,t)+\psi(x,t).
\ee
The system for the $k$-tuple of functions $\tph := (\tph_1,\dots,\tph_k)$ and the function $\psi$ is formulated as
\be \label{inner0}
\begin{aligned}
pU_j^{p-1}\pp_t\tph_j &=\Delta\tph_j+p\frac{n+2}{4}U_j^{p-1}\tph_j +\eta_{j,2}\(E+N(\Xi)+p\frac{n+2}{4}u_*^{p-1}\psi\) \\
&\ -p\eta_{j,2}(u_*^{p-1}-U_j^{p-1}) \(\pp_t\tph_j-\frac{n+2}{4}\tph_j\) \quad \text{in }\R^n\times(-\infty,t_0)
\end{aligned}
\ee
for each $j=1,\dots,k$, and 
\be \label{outer}
p u_*^{p-1} \pp_t \psi = \Delta \psi + W \psi + u_*^{p-1}\mch_{\mathrm{outer}} \quad \text{in } \R^n \times (-\infty,t_0),
\ee
where
\be \label{defW}
W(x,t) = W[\mu] (x,t) := p \frac{n+2}{4} \bigg(1-\sum_{j=1}^k\eta_j(x,t)\bigg) u_*^{p-1}(x,t)
\ee
and
\be \label{eq:Houter}
\begin{aligned}
\mch_{\mathrm{outer}} &= \mch_{\mathrm{outer}}[\psi,(\mu,\dot\mu,\phi)]\\
&:=u_*^{1-p}\bigg[\bigg(1-\sum_{j=1}^k\eta_j\bigg) (E+N(\Xi)) +\sum_{j=1}^k\big(\Delta\eta_j\tph_j+2\nabla\tph_j\cdot\nabla\eta_j-pu_*^{p-1}\tph_j\pp_t\eta_j\big)\bigg].
\end{aligned}
\ee
If $(\tph_j,\psi)$ solves \eqref{inner0}--\eqref{outer}, then $\Xi$ is a solution to \eqref{nl-p}.

\medskip
In the following definition, we specify the assumptions on the auxiliary parameters throughout the proof, including $a$, $b$, and $\sigma$.
\begin{definition}[Range of the auxiliary parameters] \label{def:num}
Let $\tnu := \frac{1}{n-2}$ and $c_0 \in (0,\tnu)$ be the number in \eqref{eq:Rt}. Also let
\begin{align*}
a \in \left.\begin{cases}
\{2\} &\text{if } n \ge 5 \\
(1,2) &\text{if } n=4 \\ 
(\frac{1}{2},1) &\text{if } n=3
\end{cases}\right\},
\quad b=1,
\quad \sigma\in(0,\min\{1,a\}),\\
\alpha_0 := \left.\begin{cases}
\frac12 &\text{if } n=3 \\
0 &\text{if } n \ge 4
\end{cases}\right\},
\quad \alpha \in(\alpha_0,a).
\end{align*}
Finally, we choose $c'\in(0,c_0/2)$ and set
\[
\beta := \frac12+(c_0-c')(a-\alpha).
\]
We fix $\vep_*>0$ such that
\[
2\vep_* < \min\big\{c'(a-\alpha),\, 1-c_0(n-2)\big\}.
\]
\end{definition}

\subsection{The outer problem}
We solve the outer problem \eqref{outer} under the assumptions that the parameters $\mu$, $\dot\mu$, $d$, and $(\xi_1,\dots,\xi_k)$ satisfy \eqref{defpoints}--\eqref{dt} for all $t\in(-\infty,t_0]$ and that \eqref{mu2} holds.
We also assume that the $k$-tuple of functions $\tph$ satisfies the following symmetry and norm conditions. More precisely, we assume that
\be \label{tildephi-ass-1}
\tph_j (\bx, x', t) = \tph_1 \big(e^{-\frac{2\pi(j-1)}{k}i} \bx, x', t\big) \quad \text{for } j=2, \dots, k,
\ee
and that $\tph_1$ satisfies the Kelvin invariance \eqref{fke} and is even with respect to $x_\ell$ for $\ell=2, \dots, n$, namely,
\[ 
\tph_1 (x_1, \dots, x_\ell, \dots, x_n, t) = \tph_1 (x_1, \dots, -x_\ell, \dots, x_n, t) \quad \text{for } \ell=2, \dots, n.
\]
Moreover, we assume that $\tph_1$ is obtained by rescaling a function $\phi$ as follows:
\be \label{tildephi-ass-3}
\tph_1(x,t) = \mu(t)^{-\frac{n-2}{2}}\phi\(\frac{x-\xi_1(t)}{\mu(t)},t\),
\ee
where $\phi$ satisfies $\|\phi\|_{\sharp\sharp,a,b;\sigma}\le\msM_1$ for a fixed $\msM_1>0$. This radius $\msM_1$ is chosen in Proposition \ref{prop:nonlin}, after the parameter bound $\msM_0$ in \eqref{mu} and \eqref{mu2} has been fixed.

\medskip
The existence of $\psi$ and its key properties are summarized in the following proposition. Its proof is given in Section \ref{sec:outer}.
\begin{prop}[Outer problem]\label{prop:outer}
Suppose that the parameter and symmetry assumptions preceding the proposition hold and that the auxiliary exponents satisfy Definition \ref{def:num}.
There exists $\msM_2>0$ such that, for sufficiently negative $t_0$, equation \eqref{outer} has a unique solution $\psi=\psi[\mu,\dot\mu,\phi]$ satisfying
\be \label{psi-norm-bound}
\|\psi\|_{\widetilde{**},\alpha,\beta;\sigma} \le \msM_2|t_0|^{c'(\alpha-a)}.
\ee
The solution satisfies the symmetries \eqref{pp1}--\eqref{fke} and is even in $x_2$.

Let $(\mu^{(m)},\phi^{(m)})$, $m\ge1$, and $(\mu,\phi)$ satisfy the preceding assumptions with the same fixed bounds. If
\be \label{psi-lip-1}
\mu^{(m)}\to\mu \quad \text{in } C^1_{\loc}((-\infty,t_0]) \quad \text{and} \quad
\phi^{(m)}\to\phi \quad \text{in } C^{2,1}_{\loc}\big(\R^n\times(-\infty,t_0]\big)
\ee
as $m \to \infty$, then
\be \label{psi-lip-2}
\psi[\mu^{(m)},\dot\mu^{(m)},\phi^{(m)}]\to\psi[\mu,\dot\mu,\phi] \quad \text{in } C^{2,1}_{\loc}\big(\R^n\times(-\infty,t_0]\big)
\ee
as $m \to \infty$. If $\hps^{(m)}$ and $\hps$ denote the corresponding conformal lifts to $\Ss^n$, then
\[
\hps^{(m)}\to\hps \quad \text{in } C^{2,1}_{\loc}\big(\Ss^n\times(-\infty,t_0]\big) \quad \text{as } m \to \infty. 
\]
Here, $C^{2,1}_{\loc}$ means uniform convergence on compact subsets of the function, its spatial derivatives through order two, and its first time derivative.

For fixed $\mu$ satisfying the preceding assumptions and two inner profiles $\phi^{(1)}$ and $\phi^{(2)}$ satisfying the same symmetry conditions and norm bound, one has
\be \label{psi-lip-3}
\left\|\psi[\mu,\dot\mu,\phi^{(1)}]-\psi[\mu,\dot\mu,\phi^{(2)}]\right\|_{\widetilde{**},\alpha,\beta;\sigma} \le C|t_0|^{c'(\alpha-a)} \big\|\phi^{(1)}-\phi^{(2)}\big\|_{\sharp\sharp,a,1;\sigma}.
\ee
The constants $\msM_2$ and $C$ depend only on $n$, $k$, the auxiliary parameters fixed in Definition \ref{def:num}, $\msM_0$, and $\msM_1$.
\end{prop}

\subsection{The inner problem}\label{subsec:inner}
We continue to assume that $\mu$, $\dot\mu$, and $d$ satisfy \eqref{mu}--\eqref{dt} for all $t\in(-\infty,t_0]$, that
\eqref{mu2} holds, and that the points $\xi_j$ ($j=1,\dots,k$) are defined by \eqref{defpoints}.
Then the symmetry condition \eqref{tildephi-ass-1} allows us to reduce the inner analysis to the case $j=1$.
A straightforward computation shows that the evolution equation for $\phi$, defined in \eqref{tildephi-ass-3}, takes the following form in the stretched variable $y=\mu^{-1}(x-\xi_1)$:
\be \label{inner}
\begin{aligned}
{\bf L}\phi &= {\bf L}[\mu,\dot\mu]\phi \\
&:= pU^{p-1}\pp_t\phi-\Delta\phi -pU^{p-1}\bigg[\frac{n+2}{4}\phi +\frac{n-2}{2}\mu^{-1}\dot\mu\phi +\mu^{-1}\nabla\phi\cdot(\dot\xi_1+\dot\mu y)\bigg] \\
&= U^{p-1}\mch_{\mathrm{inner}} \quad \text{in } \R^n\times(-\infty,t_0),
\end{aligned}
\ee
where
\be \label{mh1}
\begin{aligned}
\mch_{\mathrm{inner}}(y,t) &= \mch_{\mathrm{inner}}[\phi,(\mu,\dot{\mu})](y,t)
:= U^{1-p}(y)\mu^{\frac{n+2}{2}} \mch_1(\mu y+\xi_1,t),\\
\mch_1(x,t) &:= \eta_{1,2}\bigg[E+N(\Xi)+p\frac{n+2}{4}u_*^{p-1}\psi\\
&\hspace{1.2cm} -p(u_*^{p-1}-U_1^{p-1}) \(\pp_t\tph_1-\frac{n+2}{4}\tph_1\)\bigg](x,t).
\end{aligned}
\ee
The rest of this subsection is devoted to the analysis of the inner problem \eqref{inner}.

\subsubsection{Solvability of the nonlinear inner problem}
Let $Z_{n+1}$ be the function defined in \eqref{eq:Zi}, and let $\mcz_{n+1}$ be the modified function defined by
\be \label{eq:mczn1}
\mcz_{n+1}(y,t) := d(t)^2Z_{n+1}(y)- \mu(t)d(t)Z_1(y), \quad d(t) = \sqrt{1-\mu(t)^2}
\ee
for $(y,t)\in\R^n\times(-\infty,t_0]$, chosen so that $\mu^{-(n-2)/2} \mcz_{n+1}(\mu^{-1}(x-\xi_1),t)$ satisfies the Kelvin invariance \eqref{fke}.

We solve
\be \label{eq:nonlin}
{\bf L}[\mu,\dot\mu]\phi = U^{p-1}(\mch_{\mathrm{inner}} +c_{n+1}(t)\mcz_{n+1}) \quad \text{in } \R^n \times (-\infty,t_0),
\ee
where $\psi$ in \eqref{mh1} is the solution of Proposition \ref{prop:outer}.
The coefficient $c_{n+1}:(-\infty,t_0]\to\R$ is determined together with $\phi$. We impose the orthogonality conditions:
\be \label{eq:inhom-i21}
\int_{\R^n} \phi(y,t)U^{p-1}(y)Z_j(y) \, dy=0 \quad \text{for each } j=1,\dots,n+1 \text{ and } t \in (-\infty,t_0],
\ee
and
\be \label{eq:inhom-i22}
\int_{\R^n}\phi(y,t_0)U^p(y)\,dy=0.
\ee
In light of \eqref{eq:mczn1}, these conditions imply
\[
\int_{\R^n} \phi(y,t)U^{p-1}(y)\mcz_{n+1}(y,t) \, dy = 0 \quad \text{for each } t \in (-\infty,t_0].
\]

\medskip
The proof of the following proposition is given in Subsection \ref{subsec:innernonlin}.
\begin{prop}[Inner problem]
\label{prop:nonlin}
Suppose that the parameters $\mu$, $\dot\mu$, $d$, and $(\xi_1,\dots,\xi_k)$ satisfy \eqref{defpoints}--\eqref{dt} and \eqref{mu2} with fixed $\msM_0$.
There exists a constant $\msM_1>0$, depending only on $n$, $k$, these parameter bounds, and the auxiliary parameters in Definition \ref{def:num}, such that, for all sufficiently negative $t_0$, equation \eqref{eq:nonlin} has a unique solution pair $(\phi,c_{n+1})$ satisfying the orthogonality conditions \eqref{eq:inhom-i21}--\eqref{eq:inhom-i22}, the bound
\be \label{est-phi-inner}
\|\phi\|_{\sharp\sharp,a,1;\sigma}+\|c_{n+1}\|_{1;\sigma} \le \msM_1,
\ee
and the following symmetry conditions: for every $t\le t_0$, $\phi(\cdot,t)$ is even in $y_2,\dots,y_n$, and $x\mapsto \mu(t)^{-(n-2)/2} \phi(\mu(t)^{-1}(x-\xi_1(t)),t)$ is Kelvin invariant in the sense of \eqref{fke}.

Let $\mu^{(m)}$, $m\ge1$, and $\mu$ satisfy \eqref{mu} and \eqref{mu2}, where the bound $\msM_0$ is the same for all $m$. If $\mu^{(m)}\to\mu$ in $C^1_{\loc}((-\infty,t_0])$ as $m \to \infty$, then
\begin{align}
\phi[\mu^{(m)},\dot\mu^{(m)}] \to\phi[\mu,\dot\mu] &\quad \text{in } C^{2,1}_{\loc}\big(\R^n\times(-\infty,t_0]\big), \label{plim} \\
c_{n+1}[\mu^{(m)},\dot\mu^{(m)}] \to c_{n+1}[\mu,\dot\mu] &\quad \text{in }C^0_{\loc}((-\infty,t_0]) \label{plimd}
\end{align}
as $m \to \infty$.
\end{prop}

\subsubsection{Choice of the parameter $\mu$}
\begin{lemma}[Modulation ODE]\label{lem:modulation-single}
For the solution pair $(\phi,c_{n+1})$ of Proposition \ref{prop:nonlin}, $c_{n+1}=0$ if and only if
\be \label{chomu}
pa_1\mu^{-1}\dot\mu-\frac{n+2}{4}pa_2A_k\mu^{n-2}=\mcr_{n+1}[\mu](t).
\ee
Moreover,
\be \label{eq:mcrest}
\|\mcr_{n+1}[\mu]\|_{1+2\vep_*;\sigma}\le C,
\ee
where $C>0$ depends only on $n$, $k$, the auxiliary parameters fixed in Definition \ref{def:num}, and $\msM_0$.

Let $\mu^{(m)}$, $m\ge1$, and $\mu$ satisfy \eqref{mu} and \eqref{mu2}, where the bound $\msM_0$ is the same for all $m$.
If $\mu^{(m)}\to\mu$ in $C^1_{\loc}((-\infty,t_0])$ as $m \to \infty$, then
\[
\mcr_{n+1}[\mu^{(m)}] \to \mcr_{n+1}[\mu] \quad \text{in } C^0_{\loc}((-\infty,t_0]) \quad \text{as } m \to \infty.
\]
\end{lemma}
\begin{proof}
By \eqref{eq:mczn1} and \eqref{eq:a1a2}, together with the oddness of $U^{p-1}Z_1Z_{n+1}$ in $y_1$, it holds that $\int_{\R^n}U^{p-1}\mcz_{n+1}Z_{n+1}\,dy=d^2a_1>0$. Projecting \eqref{eq:nonlin} against $Z_{n+1}$, we observe
\be \label{eq:cn10}
\int_{\R^n} (-{\bf L}\phi+U^{p-1}\mch_{\mathrm{inner}}) Z_{n+1}\,dy =-a_1d^2c_{n+1}.
\ee
Thus $c_{n+1}=0$ if and only if this integral vanishes. By \eqref{eq:inhom-i21} and \eqref{eq:Zieq}, integration by parts reduces the contribution of $-{\bf L}\phi$ to the integral in \eqref{eq:Gn+13} below.

Multiplying \eqref{est-error-in-0} by $\eta_{1,2}(\xi_1+\mu y,t)Z_{n+1}(y)$ and integrating over $\R^n$, we apply \eqref{mu}, \eqref{boe1}--\eqref{boe2}, and \eqref{eq:etaj} to obtain
\begin{multline} \label{eq:Gn+10}
\left|\mu^{\frac{n+2}{2}} \int_{\R^n} (\eta_{1,2}E)(\xi_1+\mu y,t)Z_{n+1}(y)\,dy -pa_1\mu^{-1}\dot\mu +\frac{n+2}{4}pa_2A_k\mu^{n-2}\right| \\
\le C|t|^{-1}\left[\mu+R^{-2}+|t|^{-1}R^{n-2}+|t|^{-(p-1)}R^2\right].
\end{multline}
Taking local time differences in the explicit expressions \eqref{tildeE1} and \eqref{tildeE2_def}, and using \eqref{mu2} and \eqref{def-cutoff-j}, we obtain the same local $C_t^{\sigma/2}$ bound for the expression inside the absolute value.

For the remaining terms, the pointwise estimates used in \eqref{eq:H2}, \eqref{eq:H4}, and \eqref{eq:H3}, together with \eqref{psi-norm-bound}, \eqref{mu2}, and \eqref{est-phi-inner}, yield
\begin{align}
&\left|\frac{p(n+2)}4\mu^{\frac{n+2}{2}} \int_{\R^n} \eta_{1,2}u_*^{p-1}\psi Z_{n+1}\,dy\right| \le C|t|^{-1-(\beta-\frac12)}, \label{eq:Gn+11} \\
&\left|p\mu^{\frac{n+2}{2}}\int_{\R^n} \eta_{1,2}(u_*^{p-1}-U_1^{p-1}) \(\pp_t\tph_1-\frac{n+2}{4}\tph_1\) Z_{n+1}\,dy\right| \le C|t|^{-1-[1-c_0(n-2)]}, \label{eq:Gn+12}\\
&\left|p\int_{\R^n}\mu^{-1}\nabla\phi\cdot (\dot\xi_1+\dot\mu y)U^{p-1}Z_{n+1}\,dy\right| \le C|t|^{-2}, \label{eq:Gn+13} \\
&\left|\mu^{\frac{n+2}{2}}\int_{\R^n} \eta_{1,2}N(\Xi)Z_{n+1}\,dy\right| \le C|t|^{-1-2\vep_*}. \label{eq:Gn+14}
\end{align}
Here, the $x$-dependent factors are evaluated at $(\xi_1+\mu y,t)$. The local $C_t^{\sigma/2}$ seminorms of the four integrals satisfy the respective bounds in \eqref{eq:Gn+11}--\eqref{eq:Gn+14}.

\medskip
We define
\[
\mcr_{n+1}[\mu] := pa_1\mu^{-1}\dot\mu - \frac{n+2}{4}pa_2A_k\mu^{n-2} + a_1d^2c_{n+1}[\mu,\dot\mu].
\]
Since $a_1d^2>0$, \eqref{chomu} is equivalent to $c_{n+1}=0$.
Combining \eqref{eq:cn10} with \eqref{eq:Gn+10}--\eqref{eq:Gn+14} and Definition \ref{def:num}, we obtain \eqref{eq:mcrest}.
Finally, the definition of $\mcr_{n+1}[\mu]$ and \eqref{plimd} imply the asserted continuity under local $C^1$ convergence of $\mu$.
\end{proof}

\begin{prop}[Solvability of the parameter $\mu$] \label{prop:inner1}
For fixed $\msM_0$ sufficiently large in terms of $n$ and $k$, and sufficiently negative $t_0$, equation \eqref{chomu} has a positive solution satisfying \eqref{mu} and \eqref{mu2} of the form
\be \label{eq:muest}
\mu(t) = \left[\frac{(n^2-4)a_2A_k}{4a_1}|t|\right]^{-\frac1{n-2}}+\rho_1(t),
\ee
where $\rho_1$ is a remainder term. Moreover, for some $\ga>0$ with $0<\ga/(n-2)<\vep_*$, it holds that
\be \label{eq:rho1est}
\|\rho_1\|_{\tnu+\ga/(n-2);\sigma}+\|\dot\rho_1\|_{1+\tnu+\ga/(n-2);\sigma} \le C,
\ee
where $C>0$ depends only on $n$, $k$, the auxiliary parameters fixed in Definition \ref{def:num}, $\ga$, and $\msM_0$.
\end{prop}
\begin{proof}
We apply the scalar case of the fixed-point argument in the proof of Proposition \ref{prop:inner1-mult}, with $B=(A_k)$, $H=(1)$, and no height variable.
Choose $\delta_*\in(0,\vep_*)$ and set $\ga=(n-2)\delta_*$. The resulting fixed point gives \eqref{eq:muest} and \eqref{eq:rho1est}.
For sufficiently negative $t_0$, these estimates also imply positivity of $\mu$ and the bounds \eqref{mu} and \eqref{mu2}.
\end{proof}

\subsection{Completion of the proof of Theorem \ref{thm:main-an} with
\texorpdfstring{$h=1$}{h=1}}\label{subsec:comp}
Choose $\mu$ as in Proposition \ref{prop:inner1} and let $(\phi,c_{n+1})$ be the associated pair from Proposition \ref{prop:nonlin}. By Lemma \ref{lem:modulation-single}, $c_{n+1}=0$.
Rescaling $\phi$ by \eqref{tildephi-ass-3} and rotating as in \eqref{tildephi-ass-1} therefore gives solutions $\tph_j$ of \eqref{inner0}, with $\psi$ supplied by Proposition \ref{prop:outer}.
Multiplying the $j$-th equation by $\eta_j$, summing over $j$, and adding \eqref{outer}, using $\eta_j\eta_{j,2}=\eta_j$, shows that $\Xi=\sum_{j=1}^k\eta_j\tph_j+\psi$ solves \eqref{nl-p}.
Thus $u=u_*+\Xi$ satisfies the equation in \eqref{pb}.

Lemma \ref{lemma:eta_j} and the bounds \eqref{psi-norm-bound} and \eqref{est-phi-inner} give
\[
\sup_{x\in\R^n}\frac{|\Xi(x,t)|}{u_*(x,t)} \le C\(|t|^{-[1-c_0(n-2-a)]} +|t|^{-(\beta-\frac12+\frac{\alpha}{n-2})}\) \le \frac12
\]
for all $t\le t_0$, provided that $t_0$ is sufficiently negative. Both decay exponents in the preceding estimate are positive by Definition \ref{def:num}. In particular, $u_*/2\le u\le3u_*/2$, so $u>0$.

The parameter asymptotics \eqref{eq:parasym0} follow from \eqref{eq:muest}, \eqref{defpoints}, and \eqref{dt}.
Since $\mu(t)\to0$, we have $d(t)\to1$ and $\xi_j(t)\to{\bf q}_j$, proving Property \textup{(i)}.

For Property \textup{(ii)}, set $x_\sharp(t) := d(t)(\cos\frac\pi k,\sin\frac\pi k,0_{\R^{n-2}})$.
Since $\min_j|x_\sharp(t)-\xi_j(t)|=2d(t)\sin\frac{\pi}{2k}$ is bounded below by a positive constant, \eqref{eq:bubble2}, \eqref{eq:delta01}, and $u_*/2\le u\le3u_*/2$ imply
\[
u(\xi_1(t),t)\asymp\mu(t)^{-\frac{n-2}{2}}, \qquad u(x_\sharp(t),t)\le C\mu(t)^{\frac{n-2}{2}}.
\]
For sufficiently negative $t_0$, the first value is strictly larger than the second for every $t\le t_0$. Since $|\xi_1(t)|=|x_\sharp(t)|=d(t)$, $u(\cdot,t)$ is not radial about the origin.

For Property \textup{(iii)}, each $U_j$ is Kelvin invariant by \eqref{defpoints} and \eqref{dt}, and the same holds for $\psi$ and $\tph_j$ by Propositions \ref{prop:outer} and \ref{prop:nonlin}.
Since $\eta_j(x,t)=\eta_j(x/|x|^2,t)$ by \eqref{def-cutoff-j}, their sum $u=u_*+\sum_{j=1}^k\eta_j\tph_j+\psi$ satisfies \eqref{fke}.

For Property \textup{(iv)}, Definition \ref{def:norm-refined} and \eqref{psi-norm-bound} give a $C^{2+\sigma,1+\sigma/2}$ extension of $\hps$ across $\msN_\infty$ on each compact time interval.
By Lemma \ref{lemma:eta_j}, all $\eta_j$ vanish for $|x|\ge2$.
The bubble lifts extend positively across $\msN_\infty$ by \eqref{eq:bubble2} and \eqref{stereo-proj-1};
hence $\hu$ also extends with this regularity, and $u\ge u_*/2$ implies that its extension is positive. Conformal covariance transforms \eqref{pb} into
\[
p\hu^{p-1}\pp_t\hu =\Delta_{\Ss^n}\hu-\frac{n(n-2)}4\hu+\frac{n+2}{4}\hu^p.
\]
All terms extend continuously across $\msN_\infty$, so the equation holds on the whole sphere $\Ss^n$.
On every compact time interval contained in $(-\infty,t_0)$, $\hu$ has a positive minimum and a finite maximum.
After division by $p\hu^{p-1}$, the equation is therefore uniformly parabolic. Interior parabolic regularity gives $\hu\in C^\infty(\Ss^n\times(-\infty,t_0))$.

Finally, we replace $u(x,t)$ by $u(x,t+t_0)$ and translate $\mu$, $d$, $\xi_j$, and $\Xi$ by the same amount.
This gives \eqref{so1.2} on $\R^n\times(-\infty,0)$ and preserves \eqref{eq:parasym0} and Properties \textup{(i)}--\textup{(iv)}.
Reversing \eqref{ubarvrel} and setting $s=\tau/[p(n-1)]$ then gives a compact ancient solution of \eqref{YF-ori} on $\Ss^n$ for $s<0$, completing the single-layer case of Theorem \ref{thm:main-an}.

\section{The Outer Problem: Proof of Proposition \ref{prop:outer}}\label{sec:outer}
\subsection{The linear outer problem}
Throughout this section, we assume that $\mu$, $\dot\mu$, and $d$ satisfy \eqref{mu}--\eqref{dt} for all $t\in(-\infty,t_0]$, that \eqref{mu2} holds, and that the points $\xi_j$ ($j=1,\dots,k$) are defined by \eqref{defpoints}.

In this subsection, we solve
\be \label{outer-linear}
-p u_*^{p-1} \pp_t \psi + \Delta \psi + W\psi + u_*^{p-1} f = 0 \quad \text{in } \R^n \times (-\infty,t_0)
\ee
for a function $f$ satisfying the symmetry properties in \eqref{pp1}--\eqref{fke}, where $u_*$ is the approximate solution in \eqref{approx*} and $W$ is the function defined in \eqref{defW}.
The main result of this subsection is the following proposition.
\begin{prop}[Linear outer problem] \label{outer-thm}
Assume that $f$ satisfies \eqref{pp1}--\eqref{fke} and that $\|f\|_{L^2_{\nu;t_0}}+\|f\|_{\tilde{*},\alpha,\beta;\sigma}<\infty$ for some $\nu\in[0,1)$, $\alpha\in(0,n-2)$, $\beta>0$, and $\sigma\in(0,\min\{1,a\})$.
For sufficiently negative $t_0$, there exists an ancient solution $\psi$ to \eqref{outer-linear} satisfying \eqref{pp1}--\eqref{fke} and
\be \label{outer-bound}
\|\psi\|_{\widetilde{**},\alpha,\beta;\sigma} \le C_0 \|f\|_{\tilde{*},\alpha,\beta;\sigma},
\ee
where $C_0>0$ depends only on $n$, $k$, $\nu$, $\alpha$, $\beta$, $\sigma$, and the fixed parameter bound $\msM_0$.
If we set $\psi := \mct_{\mathrm{outer}}(f)$, then the map $\mct_{\mathrm{outer}}$ is linear.
\end{prop}

In Subsections \ref{subsec:linout1}--\ref{subsec:linout3}, we prove Proposition \ref{outer-thm} in three stages. In the first stage, we construct an ancient solution in a suitable energy space.
More precisely, we solve the equation uniquely on a finite time interval $(t_-,t_0)$, derive estimates that are uniform as $t_- \to -\infty$, and then pass to the limit to obtain an ancient solution on $(-\infty,t_0]$.
Next, under suitable pointwise assumptions on the inhomogeneous term, we show that the solution satisfies corresponding pointwise bounds.
Finally, we apply Schauder estimates to sharpen these bounds and complete the proof.

\subsection{\texorpdfstring{Solving \eqref{outer-linear} in energy spaces}{Solving the linear outer equation in energy spaces}}\label{subsec:linout1}
In this subsection, we construct an ancient solution to \eqref{outer-linear} in a weighted $H^2$ space. 

\begin{prop}[Solvability in energy spaces] \label{outer-linear-energy-global}
Assume that $\|f\|_{L^2_{\nu;t_0}} <\infty$ for some $\nu \in [0,1)$. Take $t_0<0$ sufficiently large in magnitude, if necessary. Then there exists an ancient solution $\psi$ to \eqref{outer-linear} such that
\be \label{est-energy-0}
\|\psi\|_{H^2_{\nu;t_0} } \le C\|f\|_{L^2_{\nu;t_0} },
\ee
where $C > 0$ is a constant depending only on $n$, $k$, and $\nu$. If $f$ satisfies \eqref{pp1}--\eqref{fke}, then so does $\psi$.
In addition, the map $\psi=\mct_{\mathrm{outer}}(f)$ is linear.
\end{prop}

To prove Proposition \ref{outer-linear-energy-global}, we fix a finite initial time $t_- < t_0$ and consider the initial value problem
\be \label{outer-linear-1}
\begin{cases}
-p u_*^{p-1} \pp_t \psi + \Delta \psi + W\psi + u_*^{p-1} f =0 &\text{in } \R^n \times (t_-,t_0),\\
\psi(\cdot,t_-) =0 &\text{in } \R^n.
\end{cases}
\ee
Its solution $\psi = \psi^{t_-}$ will be constructed using Lemmas \ref{outer-linear-energy}--\ref{outer-linear-energy2}.
We then pass to the limit as $t_-\to -\infty$ to produce an ancient solution of \eqref{outer-linear}. 
\begin{lemma}\label{outer-linear-energy}
Assume that $\|f\|_{L^2_{\nu;t_-,t_0}} <\infty$ for some $\nu \in [0,1)$. Take $t_0<0$ sufficiently large in magnitude, if necessary, and set $t_- < \min\{t_0-1, \frac32t_0\}$.
Then a solution $\psi$ to \eqref{outer-linear-1} satisfies the a priori estimate
\be \label{est-energy-01}
\|\psi\|_{H^2_{\nu;t_-,t_0}} \le C \|f\|_{L^2_{\nu;t_-,t_0}},
\ee
where $C > 0$ is a constant depending only on $n$, $k$, and $\nu$.
\end{lemma}
\begin{proof}
The following computations are justified by smooth approximation and spatial cutoffs. They also apply to the Dirichlet problem on balls, with constants independent of the radius.

\medskip
By \eqref{approx*}, H\"older's inequality, and the Sobolev inequality on $\dot W^{1,2}(\R^n)$, we have
\be \label{est-energy-21}
\|u_*^{p-1}(\cdot,t)\|_{L^{n/2}(\R^n)}\le C \quad \text{and} \quad \int_{\R^n}u_*^{p-1}\psi^2 \le C\int_{\R^n}|\nabla\psi|^2.
\ee
On the support of $W$, we have $|x-\xi_j|\ge R\mu/4$ for all $j=1,\dots,k$. Hence
\[
\|W(\cdot,t)\|_{L^{n/2}(\R^n)} \le C\(\sum_{j=1}^k \int_{\{|x-\xi_j|\ge R\mu/4\}}U_j^{p+1}\,dx\)^{2/n} \le CR^{-2}.
\]
Together with \eqref{partial-t-ustar} and \eqref{mu}--\eqref{dt}, this implies
\be \label{mass}
\left|\int_{\R^n}W\psi^2\right| \le CR^{-2}\int_{\R^n}|\nabla\psi|^2 \quad \text{and} \quad
|\pp_t(u_*^{p-1})| \le C|t|^{-1}u_*^{p-1}.
\ee
Multiplying \eqref{outer-linear-1} by $\psi$, integrating by parts, and applying the Cauchy--Schwarz and Young inequalities together with \eqref{est-energy-21} and \eqref{mass}, we obtain
\begin{align*}
\frac p2\frac d{dt}\int_{\R^n}u_*^{p-1}\psi^2 +\int_{\R^n}|\nabla\psi|^2 &=\frac p2\int_{\R^n}\pp_t(u_*^{p-1})\psi^2 +\int_{\R^n}W\psi^2 +\int_{\R^n}u_*^{p-1}f\psi\\
&\le \(C|t|^{-1}+CR^{-2}+\frac14\) \int_{\R^n}|\nabla\psi|^2 +C\int_{\R^n}u_*^{p-1}f^2.
\end{align*}
Taking $t_0$ sufficiently negative and using \eqref{est-energy-21} once more, we deduce
\be \label{eq:tpsell}
\frac d{dt}\int_{\R^n}u_*^{p-1}\psi^2 +c\int_{\R^n}\big(|\nabla\psi|^2+u_*^{p-1}\psi^2\big) \le C\int_{\R^n}u_*^{p-1}f^2.
\ee
Since $\psi(\cdot,t_-)=0$, it follows that
\be \label{eq:limit}
\int_{\R^n}(u_*^{p-1}\psi^2)(x,t)\,dx \le C\int_{t_-}^{t}e^{-c(t-r)} \int_{\R^n}(u_*^{p-1}f^2)(x,r)\,dx\,dr \le C|t|^{-2\nu}\|f\|_{L^2_{\nu;t_-,t_0}}^2.
\ee
The last inequality follows by splitting the time integral into intervals of length at most one and using the definition of $L^2_{\nu;t_-,t_0}$.
For $\tau\in[t_-,t_0-1]$, integration of \eqref{eq:tpsell} over $[\tau,\tau+1]$ and \eqref{eq:limit} yield
\be \label{eq:lemma33-low-window}
\int_\tau^{\tau+1}\!\!\int_{\R^n} \big(u_*^{p-1}\psi^2+|\nabla\psi|^2\big)\,dx\,dt \le C|\tau|^{-2\nu}\|f\|_{L^2_{\nu;t_-,t_0}}^2.
\ee

Next, multiplying \eqref{outer-linear-1} by $\pp_t\psi$, integrating over $\R^n$, and using $|W|\le Cu_*^{p-1}$, we obtain
\begin{multline}\label{eq:lemma33-time-test}
\frac12\frac d{dt}\int_{\R^n}|\nabla\psi|^2 +p\int_{\R^n}u_*^{p-1}(\pp_t\psi)^2 =\int_{\R^n}(W\psi+u_*^{p-1}f)\pp_t\psi\\
\le \frac p2\int_{\R^n}u_*^{p-1}(\pp_t\psi)^2 +C\int_{\R^n}u_*^{p-1}(\psi^2+f^2).
\end{multline}
Fix $\tau\in[t_-,t_0-1]$. If $\tau\ge t_-+1$, we choose $r_\tau\in[\tau-1,\tau]$ such that
\[
\int_{\R^n}|\nabla\psi(x,r_\tau)|^2\,dx \le\int_{\tau-1}^{\tau}\!\!\int_{\R^n}|\nabla\psi|^2\,dx\,dt \le C|\tau|^{-2\nu}\|f\|_{L^2_{\nu;t_-,t_0}}^2.
\]
Integrating \eqref{eq:lemma33-time-test} from $r_\tau$ to $\tau+1$ and using \eqref{eq:lemma33-low-window} together with the definition of $L^2_{\nu;t_-,t_0}$, we find
\be \label{eq:lemma33-time-window}
\int_\tau^{\tau+1}\!\!\int_{\R^n} u_*^{p-1}(\pp_t\psi)^2\,dx\,dt \le C|\tau|^{-2\nu}\|f\|_{L^2_{\nu;t_-,t_0}}^2.
\ee
For $\tau<t_-+1$, we instead integrate from $t_-$ and use the zero initial condition satisfied by the smooth approximation.
The integration interval has length at most two, and $|t_-|\asymp|\tau|$, so the same bound follows. Finally, by \eqref{outer-linear-1},
\[
u_*^{1-p}\Delta\psi=p\pp_t\psi-u_*^{1-p}W\psi-f \quad \text{in } \R^n\times(t_-,t_0).
\]
Combining this identity with \eqref{eq:lemma33-low-window}, \eqref{eq:lemma33-time-window}, and the bound $u_*^{1-p}|W|\le C$, and then taking the weighted supremum over $\tau\in[t_-,t_0-1]$, yields \eqref{est-energy-01}.
The estimates pass to the energy solution by approximation.
\end{proof}

\begin{lemma}\label{outer-linear-energy2}
Under the assumptions of Lemma \ref{outer-linear-energy}, there exists a unique solution $\psi$ to equation \eqref{outer-linear-1} satisfying the estimate \eqref{est-energy-01}.
If $f$ satisfies \eqref{pp1}--\eqref{fke}, then so does $\psi$.
\end{lemma}
\begin{proof}
Fix $t_-<t_0-1$. For $\mcr>0$, we solve
\be \label{outer-ball}
\begin{cases}
pu_*^{p-1}\pp_t\psi_{\mcr} =\Delta\psi_{\mcr}+W\psi_{\mcr}+u_*^{p-1}f &\text{in }B(0,\mcr)\times(t_-,t_0),\\
\psi_{\mcr}=0 &\text{on }\pp B(0,\mcr)\times(t_-,t_0),\\
\psi_{\mcr}(\cdot,t_-)=0 &\text{in }B(0,\mcr).
\end{cases}
\ee
On this bounded cylinder, $u_*^{p-1}$ is bounded above and below by positive constants, so standard linear parabolic theory provides an energy solution on the whole interval.
As in the proof of Lemma \ref{outer-linear-energy}, testing \eqref{outer-ball} with $\psi_{\mcr}$ and $\pp_t\psi_{\mcr}$ yields estimates uniform in $\mcr$. More precisely,
\be \label{uuno2}
\int_{t_-}^{t_0}\int_{B(0,\mcr)}|\nabla\psi_{\mcr}|^2\,dx\,dt + \sup_{t_-\le t\le t_0}\int_{B(0,\mcr)}u_*^{p-1}\psi_{\mcr}^2\,dx \le C_{t_-,t_0}\|f\|_{L^2_{\nu;t_-,t_0}}^2
\ee
and
\be \label{est-out-1}
\int_{t_-}^{t_0}\int_{B(0,\mcr)} u_*^{p-1}|\pp_t\psi_{\mcr}|^2\,dx\,dt \le C_{t_-,t_0}\|f\|_{L^2_{\nu;t_-,t_0}}^2.
\ee
Estimates \eqref{uuno2} and \eqref{est-out-1} provide the local bounds required for the standard compactness argument.
Hence, after passing to a subsequence, $\psi_{\mcr}$ converges locally to a solution $\psi$ of \eqref{outer-linear-1} on $\R^n\times(t_-,t_0)$. Moreover, by \eqref{eq:limit},
\[
\int_{\R^n}u_*^{p-1}(x,t)\psi_{\mcr}(x,t)^2\,dx\to0 \quad \text{uniformly in } \mcr \text{ as }t\downarrow t_-.
\]
Passing to the limit, we conclude that $\psi(\cdot,t_-)=0$. Lemma \ref{outer-linear-energy} then yields \eqref{est-energy-01}. Applying the same estimate with $f=0$ to the difference of two solutions proves uniqueness.
Finally, the transformations in \eqref{pp1}--\eqref{fke} preserve \eqref{outer-linear-1} and the energy class. Uniqueness therefore implies the corresponding symmetries \eqref{pp1}--\eqref{fke} of $\psi$.
\end{proof}

\begin{proof}[Proof of Proposition \ref{outer-linear-energy-global}]
We take a decreasing sequence $\{t_{-,\ell}\}_{\ell\in\N} \subset(-\infty,\min\{t_0-1,\frac32t_0\})$ such that $t_{-,\ell}\to-\infty$ as $\ell\to\infty$.
By Lemmas \ref{outer-linear-energy}--\ref{outer-linear-energy2}, for each $\ell \in \N$, there exists a unique solution $\psi^{t_{-,\ell}}$ of \eqref{outer-linear-1} with $t_-=t_{-,\ell}$, and \eqref{est-energy-01} holds with a constant $C>0$ independent of $\ell$.

Fix $T<t_0$ and $R>0$. For all sufficiently large $\ell$, the uniform estimate \eqref{est-energy-01} implies the local bounds required for compactness on $B(0,R)\times[T,t_0]$.
After passing to a subsequence, $\psi^{t_{-,\ell}}$ converges strongly in $L^2$, while its derivatives converge weakly in the corresponding local spaces.
A diagonal argument as $R\to\infty$ and $T\to-\infty$ yields an ancient weak solution $\psi$ of equation \eqref{outer-linear}.
On each unit-time interval, the terms defining the $H^2$-norm are lower semicontinuous under these convergences.
Passing to the limit in \eqref{est-energy-01} on each such interval and then taking the supremum over all unit-time intervals leads us to \eqref{est-energy-0}.

To prove uniqueness in $H^2_{\nu;t_0}$, let $\psi$ be the difference of two ancient solutions with the same source. Integrating \eqref{eq:tpsell} with $f=0$, we obtain
\be \label{eq:psiintest}
\int_{\R^n}(u_*^{p-1}\psi^2)(x,t)\,dx \le e^{-c(t-T)} \int_{\R^n}(u_*^{p-1}\psi^2)(x,T)\,dx
\ee
for $T<t\le t_0$. Since $\|\psi\|_{L^2_{\nu;t_0}}<\infty$, the integral on the right-hand side of \eqref{eq:psiintest} remains bounded along a sequence $T_\ell\to-\infty$.
Taking $T=T_\ell$ and letting $\ell\to\infty$ proves that $\psi=0$.
Linearity of $\mct_{\mathrm{outer}}$ and the symmetry properties \eqref{pp1}--\eqref{fke} follow from uniqueness. This completes the proof.
\end{proof}

\subsection{\texorpdfstring{Maximum principle and pointwise bounds for solutions to \eqref{outer-linear}}{Maximum principle and pointwise bounds for solutions to the linear outer equation}}
We derive pointwise estimates for solutions to \eqref{outer-linear} under suitable pointwise assumptions on the inhomogeneous term $f$. A key ingredient is the following result.
\begin{lemma}[Weak parabolic maximum principle] \label{maximum-principle}
We recall that $R(t)$ satisfies \eqref{eq:Rt}, and that $W$ is the function defined in \eqref{defW}.
Suppose $\psi$ satisfies $\|\psi\|_{H^1_{0;t_0} } < \infty$ and the inequality
\be \label{eq:mp0}
pu_*^{p-1} \pp_t \psi \ge \Delta \psi + W \psi \quad \text{in } \R^n \times (-\infty,t_0)
\ee
in the weak sense. Take $t_0<0$ sufficiently large in magnitude, if necessary. Then $\psi \ge 0$ a.e. in $\R^n \times (-\infty,t_0]$.
\end{lemma}
\begin{proof}
Let $\psi_-:=\max\{-\psi,0\}$. Testing \eqref{eq:mp0} with $-\psi_-$, justified by temporal averages and truncation, yields the same energy inequality as \eqref{eq:tpsell}, with $f=0$ and $\psi$ replaced by $\psi_-$.
Here, \eqref{est-energy-21} and \eqref{mass} absorb the potential and the time derivative of the weight. Consequently,
\be \label{eq:mp4}
\int_{\R^n}(u_*^{p-1}\psi_-^2)(x,t)\,dx \le e^{-c(t-T)} \int_{\R^n}(u_*^{p-1}\psi_-^2)(x,T)\,dx
\ee
for $T<t\le t_0$. Since $\|\psi\|_{H^1_{0;t_0}}<\infty$, there exists a sequence $T_\ell\to-\infty$ along which the integral on the right-hand side of \eqref{eq:mp4} remains bounded. Taking $T=T_\ell$ and letting $\ell\to\infty$ in \eqref{eq:mp4} yields $\psi_-=0$ a.e. in $\R^n \times (-\infty,t_0]$.
\end{proof}	
\begin{lemma}[Pointwise estimates] \label{outer-linear-point}
Assume that $\alpha \in (0,n-2)$ and $\beta >0$. Take $t_0<0$ sufficiently large in magnitude, if necessary. Then the solution $\psi$ to \eqref{outer-linear}, which we found in Proposition \ref{outer-linear-energy-global}, satisfies
\be \label{psif}
\|\psi\|_{\widetilde{**},\alpha,\beta} \le C\|f\|_{\tilde{*},\alpha,\beta}
\ee
provided that the right-hand side is finite. Here, $C>0$ depends only on $n$, $k$, $\alpha$, $\beta$, and the fixed parameter bound $\msM_0$.
\end{lemma}
\begin{proof}
We consider the radial functions
\[
\mcq(x) := \frac1{1+r^{n+2}}, \qquad \mcp(x) := \int_r^\infty\frac1{(r')^{n-1}} \int_0^{r'}\mcq(\rho)\rho^{n-1}\,d\rho dr'
\]
for $x \in \R^n$ and $r=|x|\ge0$. Then, $\mcp(r) > 0$ for $r \ge 0$ and $-\Delta \mcp=\mcq$ in $\R^n$. Moreover,
\[
\mcp(x)\asymp \frac{1}{1+|x|^{n-2}} \quad \text{for } x \in \R^n.
\]

Let $\delta_0 > 0$ be the number in \eqref{eq:delta0}. Given $\alpha\in(0,n-2)$, we also define
\be \label{eq:qpalpha}
\begin{aligned}
\mcq_\alpha(r,t) &:= \frac{\mone_{\{r\le\delta_0\}}}{(\mu(t)+r)^{\alpha+2}},\\
\mcp_\alpha(r,t) &:= \int_r^\infty\frac1{(r')^{n-1}} \(\int_0^{\min\{r',\delta_0\}} \frac{\rho^{n-1}\,d\rho}{(\mu(t)+\rho)^{\alpha+2}}\)dr'>0
\end{aligned}
\ee
for $(r,t) \in [0,\infty) \times (-\infty,t_0]$. Then, for each fixed $t$, $\mcp_\alpha(\cdot,t)\in C^1_{\loc}(\R^n)$ satisfies $-\Delta \mcp_\alpha=\mcq_\alpha$ weakly in $\R^n$. Furthermore, we infer from \eqref{eq:qpalpha} that
\[
|\mcp_\alpha(r,t)| \le C\left[(\mu(t)+r)^{-\alpha}\mone_{\{r\le \delta_0\}}+r^{2-n}\mone_{\{r\ge \delta_0\}}\right].
\]

\medskip
Let $C_1>0$ be a large constant to be determined below. We now define the barrier function
\[
\bps(x,t) := C_1|t|^{-(\beta+\frac{\alpha}{n-2})} \left[\sum_{j=1}^k\mcp_\alpha(|x-\xi_j|,t)+\mcp(x)\right]
\]
for $(x,t)\in\R^n\times(-\infty,t_0)$, which satisfies
\[
-\Delta\bps=C_1|t|^{-(\beta+\frac{\alpha}{n-2})} \left[\sum_{j=1}^k\mcq_\alpha(|x-\xi_j|,t)+\mcq(x)\right].
\]
Using the definition of $W$ in \eqref{defW}, we obtain
\begin{align*}
(W\bps)(x,t) &\le CC_1|t|^{-(\beta+\frac{\alpha}{n-2})} \bigg[R^{-2}\sum_{j=1}^k\mcp_\alpha(|x-\xi_j|,t) \frac{\mone_{\{|x-\xi_j|\le\delta_0\}}}{(\mu+|x-\xi_j|)^2} +\frac{\mu^2}{(1+|x|)^4}\mcp(x)\bigg]\\
&= o_{|t_0|}(1) C_1|t|^{-(\beta+\frac{\alpha}{n-2})} \left[\sum_{j=1}^k \mcq_\alpha(|x-\xi_j|,t) + \mcq(x)\right]
\end{align*}
provided that $|t_0|$ is sufficiently large. In addition,
\[
p u_*^{p-1}\pp_t \(|t|^{-(\beta+\frac{\alpha}{n-2})}\mcp\) \ge0
\]
and, since
\begin{align*}
&\ \pp_t \left[|t|^{-(\beta+\frac{\alpha}{n-2})}\mcp_\alpha(|x-\xi_j|,t)\right] \\
&= |t|^{-(\beta+\frac{\alpha}{n-2})} O\bigg(\bigg[\frac{|\dot\mu|+|\dot\xi_j|}{(\mu+|x-\xi_j|)^{\alpha+1}} +\frac{|t|^{-1}}{(\mu+|x-\xi_j|)^\alpha}\bigg] \mone_{\{|x-\xi_j|\le\delta_0\}} +\frac{|t|^{-1}}{1+|x|^{n-2}} \mone_{\{|x-\xi_j|\ge\delta_0\}}\bigg)
\end{align*}
and $\mu^{-1}(|\dot\mu|+|\dot\xi_j|)=O(|t|^{-1})$, we have
\[
pu_*^{p-1}\pp_t\left[\sum_{j=1}^k|t|^{-(\beta+\frac{\alpha}{n-2})} \mcp_\alpha(|x-\xi_j|,t)\right] =o(1)|t|^{-(\beta+\frac{\alpha}{n-2})} \left[\sum_{j=1}^k\mcq_\alpha(|x-\xi_j|,t)+\mcq(x)\right].
\]
Hence,
\[
pu_*^{p-1}\pp_t\bps-\Delta\bps-W\bps \ge\frac12C_1|t|^{-(\beta+\frac{\alpha}{n-2})} \left[\sum_{j=1}^k\mcq_\alpha(|x-\xi_j|,t)+\mcq(x)\right]
\]
provided that $|t_0|$ is sufficiently large.

On the other hand, by the definition of the $\|\cdot\|_{*,\alpha,\beta}$-norm,
\[
[u_*^{p-1}|f|](x,t) \le C\|f\|_{*,\alpha,\beta}|t|^{-(\beta+\frac{\alpha}{n-2})} \left[\sum_{j=1}^k\mcq_\alpha(|x-\xi_j|,t)+\mcq(x)\right].
\]
Therefore, choosing the constant $C_1 > 0$ sufficiently large, we deduce
\be \label{eq:bps}
\|f\|_{*,\alpha,\beta} (pu_*^{p-1}\pp_t\bps-\Delta\bps-W\bps) \ge u_*^{p-1}|f| \quad \text{in } \R^n\times(-\infty,t_0).
\ee

\medskip
By combining \eqref{eq:bps} and \eqref{outer-linear}, we find
\[
pu_*^{p-1}\pp_t(\|f\|_{*,\alpha,\beta}\bps\pm\psi) \ge \Delta(\|f\|_{*,\alpha,\beta}\bps\pm\psi) + W(\|f\|_{*,\alpha,\beta}\bps\pm\psi)
\]
in $\R^n\times(-\infty,t_0)$ in the weak sense. Since $\|\bps\|_{H^1_{0;t_0}}<\infty$, Lemma \ref{maximum-principle} yields
\[
|\psi| \le \|f\|_{*,\alpha,\beta}\bps \le C\|f\|_{*,\alpha,\beta} \(\omega^1_{\alpha,\beta,0} +\omega^2_{n-2,\beta+\frac{\alpha}{n-2}}\)
\]
a.e. in $\R^n\times(-\infty,t_0)$. In other words,
\[
\|\psi\|_{**,\alpha,\beta} \le C\|f\|_{*,\alpha,\beta},
\]
which is equivalent to \eqref{psif}.
\end{proof}

\subsection{\texorpdfstring{Schauder estimates for solutions to \eqref{outer-linear}}{Schauder estimates for solutions to the linear outer equation}}\label{subsec:linout3}
In this subsection, we complete the proof of Proposition \ref{outer-thm}. In view of Proposition \ref{outer-linear-energy-global} and Lemma \ref{outer-linear-point}, it remains to establish the Schauder estimate \eqref{outer-bound} for solutions to \eqref{outer-linear}.

\begin{proof}[Proof of Proposition \ref{outer-thm}]
The proof is partly motivated by \cite{km}. We treat the bubbling region and the exterior region separately.

\medskip \noindent \textbf{Step 1.}
We first consider the region $\cup_{j=1}^k\{|x-\xi_j|\le\delta_0/6\} \times (-\infty,t_0)$, where $\delta_0$ is defined in \eqref{eq:delta0}. We give the estimates for $j=1$; the same argument applies at each center.

\medskip
Set
\[
\bps(y,t') := \psi(\mu(t')y+\xi_1(t'),t') \quad \text{and} \quad \baf(y,t') := f(\mu(t')y+\xi_1(t'),t').
\]
Then \eqref{outer-linear} becomes
\be \label{eq:co1}
p\pp_{t'}\bps = \mu^{-2}u_*^{1-p}\Delta_y\bps +p\mu^{-1}(\dot\mu y+\dot\xi_1)\cdot\nabla_y\bps +u_*^{1-p}W\bps+\baf,
\ee
where the coefficients are evaluated at $(x,t')=(\mu(t')y+\xi_1(t'),t')$. We fix $t<t_0$ and define
\[
\wtq := \{|y|\le4\}\times[t-1,t] \quad \text{and} \quad Q := \{|y|\le6\}\times[t-2,t].
\]
By \eqref{eq:bubble}--\eqref{eq:bubble2}, \eqref{approx*}, \eqref{mu}--\eqref{dt}, \eqref{mu2}, and Lemma \ref{lemma:eta_j}, the diffusion coefficient $\mu^{-2}u_*^{1-p}$ is bounded above and below by positive constants on $Q$,
while all the coefficients in \eqref{eq:co1} have uniformly bounded $C^{\sigma,\sigma/2}$ norms. Hence \eqref{eq:co1} is uniformly parabolic on $Q$. Interior Schauder estimates therefore yield
\be \label{eq:co2}
\sum_{\ell=0}^2\|\nabla_y^\ell\bps\|_{C^{\sigma,\sigma/2}(\wtq)} +\|\pp_{t'}\bps\|_{C^{\sigma,\sigma/2}(\wtq)} \lesssim\|\bps\|_{L^\infty(Q)}+\|\baf\|_{C^{\sigma,\sigma/2}(Q)},
\ee
where
\[
\|v\|_{C^{\sigma,\sigma/2}(Q)} := \|v\|_{L^\infty(Q)}+ \sup_{\substack{(y_1,t_1),\,(y_2,t_2)\in Q \\ (y_1,t_1)\ne(y_2,t_2)}} \frac{|v(y_1,t_1)-v(y_2,t_2)|}{|y_1-y_2|^\sigma+|t_1-t_2|^{\sigma/2}}.
\]
For any $(y,t') \in Q$, with $x=\mu(t')y+\xi_1(t')$, it holds that $\min_j \langle \mu(t')^{-1}(x-\xi_j(t')) \rangle\asymp1$ and $|t'|\asymp|t|$. Thus \eqref{psif} and Definition \ref{def:norm-refined} imply
\[
\|\bps\|_{L^\infty(Q)} \lesssim |t|^{-\beta}\|f\|_{\tilde *,\alpha,\beta}.
\]
Furthermore, we observe that
\[
\|\baf\|_{C^{\sigma,\sigma/2}(Q)} \lesssim |t|^{-\beta}\|f\|_{\tilde *,\alpha,\beta;\sigma}.
\]
Returning to the $x$-variable at $y=\mu(t)^{-1}(x-\xi_1(t))$, we note that
\be \label{eq:psix}
\begin{aligned}
\nabla_x^\ell\psi(x,t) &=\mu(t)^{-\ell}\nabla_y^\ell\bps(y,t),\\
\pp_t\psi(x,t) &=\pp_t\bps(y,t)-\mu(t)^{-1}(\dot\mu(t)y+\dot\xi_1(t))\cdot\nabla_y\bps(y,t).
\end{aligned}
\ee
By substituting \eqref{eq:psix} into \eqref{eq:co2} and applying \eqref{mu}--\eqref{dt} as well as \eqref{mu2}, we deduce that for $|y|\le3$ and $\ell\in\{0,1,2\}$,
\be \label{eq:co4}
|\nabla_x^\ell\psi(x,t)|+\mu(t)^\sigma[\nabla_x^\ell\psi]_{C_x^\sigma}(x,t)+[\nabla_x^\ell\psi]_{C_t^{\sigma/2}}(x,t) \\
\lesssim |t|^{-\beta}\mu(t)^{-\ell} \|f\|_{\tilde *,\alpha,\beta;\sigma}
\ee
and
\be \label{eq:co5}
|\pp_t\psi(x,t)|+\mu(t)^\sigma[\pp_t\psi]_{C_x^\sigma}(x,t)+[\pp_t\psi]_{C_t^{\sigma/2}}(x,t) \\
\lesssim |t|^{-\beta} \|f\|_{\tilde *,\alpha,\beta;\sigma}.
\ee
Here, we use the temporal H\"older seminorm $[\cdot]_{C_t^{\sigma/2}}$ in \eqref{eq:Holdert} and set $[\psi]_{C_x^\sigma} := [\psi]_{C^\sigma_{B(0,\msR_0)}}$ with the radius $r_0$ from \eqref{eq:r0}.

\medskip
We next consider the region $2\mu(t)<|x-\xi_1(t)|\le\delta_0/6$. We choose $m_0\in\N$ such that $2^{m_0}\le\delta_0/(6\mu(t))<2^{m_0+1}$. For $m\in\{1,\dots,m_0\}$, we set
\[
\wtq_{m,1}:=\{1<|Y|<2\}\times[-1,0],\qquad \wtq_{m,2}:=\{\tfrac12<|Y|<\tfrac52\}\times[-2,0].
\]
We define
\[
\Psi(Y,\tau):=\bps(2^mY,t+2^{-2m}\tau).
\]
By \eqref{eq:co1},
\be \label{eq:co6}
\begin{aligned}
p\pp_\tau\Psi &= \mu^{-2}2^{-4m}u_*^{1-p}\Delta_Y\Psi +2^{-2m}p\mu^{-1}(\dot\mu Y+2^{-m}\dot\xi_1)\cdot\nabla_Y\Psi\\
&\ +2^{-2m}u_*^{1-p}W\Psi +2^{-2m}\baf(2^mY,t+2^{-2m}\tau),
\end{aligned}
\ee
where $t'=t+2^{-2m}\tau$ and $x=\mu(t')2^mY+\xi_1(t')$ in the coefficients. On $\wtq_{m,2}$, the bubble formula implies $u_*^{p-1}(x,t')\asymp\mu(t')^{-2}2^{-4m}$.
Together with \eqref{mu}--\eqref{dt}, \eqref{mu2}, and Lemma \ref{lemma:eta_j}, this ensures uniform ellipticity and uniform $C^{\sigma,\sigma/2}$ bounds for the coefficients of \eqref{eq:co6}, independently of $t$ and $m$. Interior Schauder estimates therefore yield
\begin{multline}\label{eq:Psiest}
\sum_{\ell=0}^2\|\nabla_Y^\ell\Psi\|_{C^{\sigma,\sigma/2}(\wtq_{m,1})} +\|\pp_\tau\Psi\|_{C^{\sigma,\sigma/2}(\wtq_{m,1})}\\
\lesssim\|\Psi\|_{L^\infty(\wtq_{m,2})} +\big\|2^{-2m}\baf(2^m\,\cdot,t+2^{-2m}\,\cdot)\big\|_{C^{\sigma,\sigma/2}(\wtq_{m,2})}.
\end{multline}
The first term on the right-hand side of \eqref{eq:Psiest} is governed by \eqref{psif}, while Definitions \ref{defpr} and \ref{def:norm-refined} control the second term:
\be \label{eq:Psiest2}
\begin{aligned}
\|\Psi\|_{L^\infty(\wtq_{m,2})} &\lesssim |t|^{-\beta}2^{-m\alpha}\|f\|_{\tilde *,\alpha,\beta}, \\
\big\|2^{-2m}\baf(2^m\,\cdot,t+2^{-2m}\,\cdot)\big\|_{C^{\sigma,\sigma/2}(\wtq_{m,2})} &\lesssim |t|^{-\beta}2^{-m\alpha} \|f\|_{\tilde *,\alpha,\beta;\sigma}.
\end{aligned}
\end{equation}
Employing \eqref{eq:Psiest2} and
\begin{align*}
\nabla_x^\ell\psi(x,t) &= (\mu(t)2^m)^{-\ell}\nabla_Y^\ell\Psi(Y,0), \\
\pp_t\psi(x,t) &= 2^{2m}\pp_\tau\Psi(Y,0)-\mu(t)^{-1}(\dot\mu(t)y+\dot\xi_1(t))\cdot\nabla_y\bps(y,t),
\end{align*}
where $y=\mu(t)^{-1}(x-\xi_1(t))=2^mY$, we obtain
\be \label{eq:co8}
|\nabla_x^\ell\psi(x,t)|+(\mu(t)|y|)^\sigma[\nabla_x^\ell\psi]_{C_x^\sigma}(x,t) +|y|^{-\sigma}[\nabla_x^\ell\psi]_{C_t^{\sigma/2}}(x,t) \\
\lesssim\frac{|t|^{-\beta}\mu(t)^{-\ell}}{|y|^{\alpha+\ell}} \|f\|_{\tilde *,\alpha,\beta;\sigma}
\ee
for $\ell\in\{0,1,2\}$, and
\be \label{eq:co9}
|\pp_t\psi(x,t)|+(\mu(t)|y|)^\sigma[\pp_t\psi]_{C_x^\sigma}(x,t) +|y|^{-\sigma}[\pp_t\psi]_{C_t^{\sigma/2}}(x,t) \\
\lesssim\frac{|t|^{-\beta}}{|y|^{\alpha-2}} \|f\|_{\tilde *,\alpha,\beta;\sigma}.
\ee
The estimates extend to the boundary spheres of the dyadic annuli by continuity. Thus they hold throughout $2\mu(t)<|x-\xi_1(t)|\le\delta_0/6$ and $t<t_0$.

\medskip \noindent \textbf{Step 2.}
We examine the region $[\Ss^n\setminus\pi(\cup_{j=1}^kB(\xi_j(t),\delta_0/12))] \times (-\infty,t_0)$, where $\pi: \R^n \to \Ss^n \setminus \{\msN_{\infty}\}$ is the map defined in \eqref{stereo-proj}.

\medskip
We introduce the time coordinate $\mss=-|t|^{n/(n-2)}$ and write
\be \label{stps}
\psi(x,t) = \(\frac2{1+|x|^2}\)^{\frac{n-2}{2}}\tps(z,\mss), \qquad z=\pi(x).
\ee
From \eqref{stps}, we see that
\[
\int_{\Ss^n} \(|\nabla_{\Ss^n}\tps(z,\mss)|^2+\frac{n(n-2)}4|\tps(z,\mss)|^2\)\,dv_{\Ss^n}(z) = \int_{\R^n}|\nabla_x\psi(x,t)|^2\,dx.
\]
Together with \eqref{est-energy-0}, this implies $\tps\in L^2_{\loc}((-\infty,-|t_0|^{n/(n-2)}); H^1(\Ss^n))$.
This bound and a standard cutoff argument extend the equation for $\tps$ weakly across $\msN_\infty$, yielding
\[
\pp_{\mss}\tps =\mca\(\Delta_{\Ss^n}\tps-\frac{n(n-2)}4\tps\) +\wtw\tps+\tf \quad \text{in } \Ss^n\times(-\infty,-|t_0|^{n/(n-2)}),
\]
where
\be \label{stf}
\begin{aligned}
\mca(z,\mss) &:= \frac{n-2}{np}|t|^{-\frac2{n-2}} \(\frac2{1+|x|^2}\)^2u_*^{1-p}(x,t),\\
\wtw(z,\mss) &:= \frac{n^2-4}{4n}|t|^{-\frac2{n-2}} \bigg(1-\sum_{j=1}^k\eta_j(x,t)\bigg),
\end{aligned}
\ee
and
\be \label{stf2}
\tf(z,\mss) := \frac{n-2}{np}|t|^{-\frac2{n-2}} \(\frac{1+|x|^2}{2}\)^{\frac{n-2}{2}}f(x,t) = \frac{n-2}{np}|t|^{-\frac2{n-2}}\hf(z,t).
\ee
Here, $\hf$ is the conformal lift of $f$ defined by \eqref{stereo-proj-1}.

Since $\xi_j(t)\to{\bf q}_j$ as $t\to-\infty$, we can choose finitely many fixed pairs of coordinate patches $\wtom\Subset\Omega\Subset \Ss^n\setminus\{\pi({\bf q}_j):j=1,\dots,k\}$ such that the patches $\wtom$ cover $\Ss^n\setminus\pi(\cup_{j=1}^kB(\xi_j(t),\delta_0/12))$ for all $t\le t_0$.
We take one $\wtom$ to contain $\msN_\infty$ and decrease $t_0$ so that every $\Omega$ remains a fixed positive spherical distance from each $\pi(\xi_j(t))$.
On each $\Omega$, it follows from \eqref{stf}, \eqref{eq:bubble2}, \eqref{approx*}, \eqref{mu}--\eqref{dt}, \eqref{mu2}, and Lemma \ref{lemma:eta_j} that $C^{-1}\le\mca\le C$ and that $\mca$ and $\wtw$ have uniformly bounded $C^{\sigma,\sigma/2}$-norms in $(z,\mss)$ throughout $t\le t_0$.

For $t<t_0$, we set $J_t=[\mss-2,\mss]$ and $J'_t=[\mss-1,\mss]$. If $\mss'\in J_t$ and $t'<0$ is defined by $\mss'=-|t'|^{n/(n-2)}$, then $|t'|\asymp|t|$ uniformly for $\mss'\in J_t$. Hence, by the mean value theorem,
\[
|t_1-t_2|\asymp |t|^{-\frac2{n-2}}|\mss_1-\mss_2| \quad \text{for any } \mss_1,\mss_2\in J_t.
\]
Consequently, \eqref{psif}, \eqref{stf2}, and Definitions \ref{defpr} and \ref{def:norm-refined} imply
\begin{align*}
\|\tps\|_{L^\infty(\Omega\times J_t)} &\lesssim |t|^{-(\beta+\alpha/(n-2))}\|f\|_{\tilde *,\alpha,\beta},\\
\|\tf\|_{C^{\sigma,\sigma/2}(\Omega\times J_t)} &\lesssim |t|^{-(\beta+\alpha/(n-2))}\|f\|_{\tilde *,\alpha,\beta;\sigma}.
\end{align*}
Interior Schauder estimates therefore yield
\begin{multline*}
\sum_{\ell=0}^2\|\nabla_{\Ss^n}^\ell\tps\|_{C^{\sigma,\sigma/2}(\wtom\times J'_t)} +\|\pp_{\mss}\tps\|_{C^{\sigma,\sigma/2}(\wtom\times J'_t)}\\
\lesssim \|\tps\|_{L^\infty(\Omega\times J_t)} +\|\tf\|_{C^{\sigma,\sigma/2}(\Omega\times J_t)} \lesssim |t|^{-(\beta+\alpha/(n-2))}\|f\|_{\tilde *,\alpha,\beta;\sigma}.
\end{multline*}

To recover the unit-time seminorms in the $t$-variable, we use the $\mss$-H\"older estimates when $|t_1-t_2|\le C^{-1}|t|^{-2/(n-2)}$ and the pointwise bounds otherwise. Since
\[
\hps(z,t)=\tps(z,\mss),\qquad \pp_t\hps(z,t)=\frac{n}{n-2}|t|^{\frac2{n-2}}\pp_{\mss}\tps(z,\mss),
\]
and
\[
|\nabla_x^\ell\psi(x,t)| \le C\(\frac2{1+|x|^2}\)^{\frac{n-2+\ell}{2}} \sum_{i=0}^\ell|\nabla_{\Ss^n}^i\tps(z,\mss)| \quad \text{for } \ell\in\{0,1,2\},
\]
we obtain, whenever $x\in B(0,\msR_0)$ and $\min_j|x-\xi_j(t)|\ge\delta_0/12$,
\begin{multline}\label{eso1}
(1+|x|)^\ell\Big[|\nabla_x^\ell\psi(x,t)| +(1+|x|)^\sigma[\nabla_x^\ell\psi]_{C_x^\sigma}(x,t) \\
+|t|^{-\frac\sigma{n-2}}[\nabla_x^\ell\psi]_{C_t^{\sigma/2}}(x,t)\Big] \le C\frac{|t|^{-(\beta+\frac\alpha{n-2})}}{1+|x|^{n-2}} \|f\|_{\tilde *,\alpha,\beta;\sigma}
\end{multline}
for $\ell\in\{0,1,2\}$, and
\be \label{eso2}
|\pp_t\psi(x,t)|+(1+|x|)^\sigma[\pp_t\psi]_{C_x^\sigma}(x,t) +|t|^{-\frac\sigma{n-2}}[\pp_t\psi]_{C_t^{\sigma/2}}(x,t)
\le C\frac{|t|^{-(\beta+\frac{\alpha-2}{n-2})}}{1+|x|^{n-2}} \|f\|_{\tilde *,\alpha,\beta;\sigma}.
\ee
On $\mcb^{\ep_0}\times(-\infty,t_0]$, the same estimates imply
\be \label{eso3}
|\nabla_{\Ss^n}^\ell\hps|(z,t)
+[\nabla_{\Ss^n}^\ell\hps]_{C_z^\sigma}(z,t)
+|t|^{-\frac\sigma{n-2}}[\nabla_{\Ss^n}^\ell\hps]_{C_t^{\sigma/2}}(z,t)
\le C|t|^{-(\beta+\frac\alpha{n-2})}\|f\|_{\tilde *,\alpha,\beta;\sigma}
\ee
for $\ell\in\{0,1,2\}$, and
\be \label{eso4}
|\pp_t\hps(z,t)|+[\pp_t\hps]_{C_z^\sigma}(z,t)
+|t|^{-\frac\sigma{n-2}}[\pp_t\hps]_{C_t^{\sigma/2}}(z,t) \le C|t|^{-(\beta+\frac{\alpha-2}{n-2})}\|f\|_{\tilde *,\alpha,\beta;\sigma},
\ee
where $[\hps]_{C_z^\sigma} := [\hps]_{C^\sigma_{\mcb^{\ep_0}}}$.

\medskip
On $|x-\xi_1(t)|\le\delta_0/6$, we have $\min_j \langle y_j\rangle\asymp\langle y_1\rangle$, whereas $\min_j \langle y_j\rangle\asymp\mu(t)^{-1}$ in the bounded exterior region.
By these comparisons, \eqref{mu}, and Definition \ref{defpr}, the estimates \eqref{eq:co4}--\eqref{eq:co5}, \eqref{eq:co8}--\eqref{eq:co9}, and \eqref{eso1}--\eqref{eso2} imply the required bound for
$\|\psi\|_{**,\alpha,\beta;\sigma(B(0,\msR_0)_{t_0})}$.
Combining this bound with \eqref{eso3}--\eqref{eso4} proves \eqref{outer-bound}.
\end{proof}
\begin{remark}\label{rmk:psitdecay}
The spherical component of $\|f\|_{\tilde *,\alpha,\beta;\sigma}$ controls the H\"older regularity of the conformal lift $\hf$ near the north pole $\msN_\infty$.
This allows the exterior Schauder estimates in the proof of Proposition \ref{outer-thm} to be applied across $\msN_\infty$.
The resulting bounds for $\hps$, together with \eqref{stereo-proj-1}, imply that, for $t\le t_0$ and $|x|\ge\msR_0$,
\begin{align*}
|\nabla_x^\ell\psi(x,t)|
&\le C|t|^{-\beta-\alpha/(n-2)}|x|^{2-n-\ell}
\|f\|_{\tilde *,\alpha,\beta;\sigma}
\quad \text{for } 0\le\ell\le2,\\
|\pp_t\psi(x,t)|
&\le C|t|^{-\beta-(\alpha-2)/(n-2)}|x|^{2-n}
\|f\|_{\tilde *,\alpha,\beta;\sigma}.
\end{align*}
By contrast, the Euclidean estimates based only on $\|f\|_{*,\alpha,\beta;\sigma}$ yield the weaker bound
\[
|\pp_t\psi(x,t)| \le C|t|^{-\beta-(\alpha-2)/(n-2)}|x|^{4-n} \|f\|_{*,\alpha,\beta;\sigma}.
\]
This loss of two powers of $|x|$ is insufficient for the fixed-point argument, described in the next subsection. In short, the spherical component ensures the required $|x|^{2-n}$ decay of $\pp_t\psi$.
\end{remark}

\subsection{Solving the outer problem}
The auxiliary exponents are fixed as in Definition \ref{def:num}. We recall the function $\mch_{\mathrm{outer}}$ in \eqref{eq:Houter} and write
\begin{align}
\mch_{\mathrm{outer},1} &:= E_0+u_*^{1-p}\sum_{j=1}^k \big(\Delta\eta_j\tph_j+2\nabla\eta_j\cdot\nabla\tph_j-pu_*^{p-1}\tph_j\pp_t\eta_j\big),\nonumber\\
\mch_{\mathrm{outer},2} &:= u_*^{1-p}\bigg(1-\sum_{j=1}^k\eta_j\bigg)N(\Xi), \qquad E_0 := u_*^{1-p}\bigg(1-\sum_{j=1}^k\eta_j\bigg)E.\label{hout2}
\end{align}
Then $\mch_{\mathrm{outer}}=\mch_{\mathrm{outer},1}+\mch_{\mathrm{outer},2}$.

\begin{lemma}[Estimate of $\mch_{\mathrm{outer}}$] \label{lemma:Houterest2}
Suppose that
\[
\|\phi\|_{\sharp\sharp,a,1;\sigma}\le \msM_1 \quad \text{and} \quad \|\psi\|_{\widetilde{**},\alpha,\beta;\sigma}\le \msM_2|t_0|^{c'(\alpha-a)}
\]
for fixed $\msM_1,\msM_2>0$. For sufficiently negative $t_0$,
\be\label{hou1}
\|\mch_{\mathrm{outer}}\|_{\tilde *,\alpha,\beta;\sigma} \le C_{\msM_0,\msM_1}|t_0|^{c'(\alpha-a)}+C_{\msM_0,\msM_1,\msM_2}|t_0|^{-2\vep_*} \|\psi\|_{\widetilde{**},\alpha,\beta;\sigma}.
\ee
In particular, $\|\mch_{\mathrm{outer}}\|_{L^2_{0;t_0}}<\infty$.
\end{lemma}
\begin{proof}
We first estimate $\mch_{\mathrm{outer},1}$. For $x\in B(0,\msR_0)$, we choose $j \in \{1,\dots,k\}$ such that $|x-\xi_j|=\min_{1\le i\le k}|x-\xi_i|$ and set $y_j=\mu^{-1}(x-\xi_j)$.
On the support of the function $1-\sum_{i=1}^k\eta_i$, we have $|y_j|\ge cR$. Lemma \ref{size error} and \eqref{est-error-out-0} imply
\begin{multline*} 
|t|^\beta\mu^2(1+|y_j|^{\alpha+2}) \left| \(1-\sum_{i=1}^k\eta_i\)E \right| \le C|t|^{\beta-\frac12}|y_j|^{\alpha-2}\\
\le C|t|^{(c_0-c')(a-\alpha)}R^{\alpha-2} \le C|t|^{c_0(a-2)+c'(\alpha-a)} \le C|t|^{c'(\alpha-a)}.
\end{multline*}
Thus $\|E_0\|_{*,\alpha,\beta(B(0,\msR_0)_{t_0})} \le C|t_0|^{c'(\alpha-a)}$.

The estimates in Lemma \ref{lemma:eta_j} imply, with $y_j=\mu^{-1}(x-\xi_j)$,
\be \label{esnec}
|\Delta\eta_j\tph_j+2\nabla\eta_j\cdot\nabla\tph_j-pu_*^{p-1}\tph_j\pp_t\eta_j| \le C\mu^{-(n+2)/2}|t|^{-1}R^{-a-2} \mone_{\{R/4\le|y_j|\le5R\}} \|\phi\|_{\sharp\sharp,a,1;\sigma}.
\ee
Using \eqref{esnec} and $\beta+\frac12=b+(c_0-c')(a-\alpha)$, we obtain
\begin{multline*}
\sum_{j=1}^k |t|^\beta\mu^2(1+|y_j|^{\alpha+2}) \left| \Delta\eta_j\tph_j +2\nabla\eta_j\cdot\nabla\tph_j -pu_*^{p-1}\tph_j\pp_t\eta_j \right| \\
\le C|t|^{(c_0-c')(a-\alpha)}R^{\alpha-a} \|\phi\|_{\sharp\sharp,a,1;\sigma} \le C|t|^{c'(\alpha-a)} \|\phi\|_{\sharp\sharp,a,1;\sigma}.
\end{multline*}
The corresponding H\"older bounds follow from similar weighted estimates, the product inequalities, \eqref{mu2}, and Lemma \ref{lemma:eta_j}. Hence
\[
\|\mch_{\mathrm{outer},1}\|_{*,\alpha,\beta;\sigma(B(0,\msR_0)_{t_0})} \le C|t_0|^{c'(\alpha-a)}.
\]

We next consider the set $\mcb^{\ep_0}_{t_0} := B_{\Ss^n}(\msN_{\infty},\ep_0) \times (-\infty,t_0]$. 
Let $\whe_0$ be the conformal lift of $E_0$, defined through \eqref{stereo-proj-1}.
For sufficiently negative $t_0$, all cutoff terms vanish on this set, so $\widehat{\mch}_{\mathrm{outer},1}=\whe_0$ there. Writing $z=(\bar z,z_{n+1})$, let
\[
\widehat U_j(z,t) = \msc_n\(\frac{\mu}{2}\)^{\frac{n-2}{2}} \(1-d\,\bar z\cdot {\bf q}_j\)^{-\frac{n-2}{2}}, \qquad \hu_*=\sum_{j=1}^k\widehat U_j,
\]
where ${\bf q}_j \in \Ss^{n-1}$ is defined in \eqref{eq:bfqs}. Then
\[
\whe_0 = -p\pp_t\hu_* -\frac{n+2}{4}\hu_*^{1-p} \sum_{j=1}^k\widehat U_j^p +\frac{n+2}{4}\hu_*.
\]
On $\mcb^{\ep_0}_{t_0}$, it holds that $1-d(t)\,\bar z\cdot{\bf q}_j\ge c>0$ uniformly for all sufficiently negative $t_0$. Together with \eqref{mu}-\eqref{dt} and \eqref{mu2}, this yields
\[
\|\whe_0\|_{*',\alpha,\beta;\sigma(\mcb^{\ep_0}_{t_0})} \le C|t_0|^{c'(\alpha-a)}.
\]
Therefore,
\[
\|\mch_{\mathrm{outer},1}\|_{\tilde *,\alpha,\beta;\sigma} \le C|t_0|^{c'(\alpha-a)}.
\]

Finally, the nonlinear estimate established in Appendix \ref{app4} implies
\be \label{poho2}
\|\mch_{\mathrm{outer},2}\|_{\tilde *,\alpha,\beta;\sigma} \le C|t_0|^{-2\vep_*} \left[|t_0|^{c'(\alpha-a)}\|\phi\|_{\sharp\sharp,a,1;\sigma}+\|\psi\|_{\widetilde{**},\alpha,\beta;\sigma}\right].
\ee
Combining the estimates for $\mch_{\mathrm{outer},1}$ and $\mch_{\mathrm{outer},2}$ proves \eqref{hou1}.

\medskip
Using the definition of $\|\mch_{\mathrm{outer}}\|_{\tilde *,\alpha,\beta}$, we have
\begin{align*}
\int_{\R^n}u_*^{p-1}|\mch_{\mathrm{outer}}|^2\,dx &\le C|t|^{-2\beta}\|\mch_{\mathrm{outer}}\|_{\tilde *,\alpha,\beta}^2 \left[\mu^{n-2}\int_0^{C/\mu}r^{n-1}(1+r^2)^{-\alpha}\,dr+\mu^{2\alpha-2}\right]\\
&\le C|t|^{-2\beta+\frac{2(1-\alpha)}{n-2}}\|\mch_{\mathrm{outer}}\|_{\tilde *,\alpha,\beta}^2.
\end{align*}
Here, we also used $n>2\alpha$ and $\mu\asymp|t|^{-1/(n-2)}$. Since $\beta-(1-\alpha)/(n-2)>0$ by Definition \ref{def:num}, integration on unit intervals yields
\[
\|\mch_{\mathrm{outer}}\|_{L^2_{0;t_0}} \le C|t_0|^{-\beta+\frac{1-\alpha}{n-2}}\|\mch_{\mathrm{outer}}\|_{\tilde *,\alpha,\beta} < \infty. \qedhere
\]
\end{proof}

\begin{proof}[Proof of Proposition \ref{prop:outer}]
We fix $\mu$ and $\phi$ as in the proposition and let $\mct_{\mathrm{outer}}$ be the linear solution operator from Proposition \ref{outer-thm}.
We define
\[
\mct(\psi) := \mct_{\mathrm{outer}} \big(\mch_{\mathrm{outer}}[\psi,(\mu,\dot\mu,\phi)]\big).
\]
By Lemma \ref{lemma:Houterest2} and \eqref{outer-bound}, $\mct$ maps the ball \eqref{psi-norm-bound}, restricted to functions satisfying \eqref{pp1}--\eqref{fke}, into itself when $\msM_2$ is sufficiently large and $t_0$ is sufficiently negative.
Since $\phi$ is fixed, \eqref{nonh2} and \eqref{outer-bound} yield
\[
\left\|\mct(\psi^{(1)})-\mct(\psi^{(2)})\right\|_{\widetilde{**},\alpha,\beta;\sigma} \le C|t_0|^{-2\vep_*} \big\|\psi^{(1)}-\psi^{(2)}\big\|_{\widetilde{**},\alpha,\beta;\sigma}.
\]
Thus $\mct$ is a contraction and has a unique fixed point in the ball \eqref{psi-norm-bound} subject to \eqref{pp1}--\eqref{fke}.
The linear estimate \eqref{outer-bound} remains valid without symmetry assumptions, as follows from the energy argument, \eqref{psif}, and the Schauder estimates in the proof of Proposition \ref{outer-thm}.
Together with \eqref{nonh2}, this proves uniqueness among all solutions in the ball \eqref{psi-norm-bound}.
Reflection invariance of the equation then implies that $\psi$ is even in $x_2$.

The continuity statement \eqref{psi-lip-2}, including the corresponding convergence of the conformal lifts, is proved in Appendix \ref{app5}.

It remains to prove \eqref{psi-lip-3}. Let $\psi^{(1)}=\psi[\mu,\dot\mu,\phi^{(1)}]$ and $\psi^{(2)}=\psi[\mu,\dot\mu,\phi^{(2)}]$ with $\mu$ fixed.
For fixed $\mu$, the cutoff terms in \eqref{eq:Houter} are linear in $\phi$, while $E_0$ is independent of $\phi$.
Applying \eqref{esnec} to their difference and using the weighted estimate in the proof of Lemma \ref{lemma:Houterest2}, we obtain
\be \label{nonh1}
\left\|\mch_{\mathrm{outer},1}[\phi^{(1)}]-\mch_{\mathrm{outer},1}[\phi^{(2)}]\right\|_{\tilde *,\alpha,\beta;\sigma} \le C|t_0|^{c'(\alpha-a)} \big\|\phi^{(1)}-\phi^{(2)}\big\|_{\sharp\sharp,a,1;\sigma}.
\ee
Subtracting the outer equations satisfied by $\psi^{(1)}$ and $\psi^{(2)}$, and applying Proposition \ref{outer-thm}, \eqref{nonh1}, and \eqref{nonh2}, we obtain
\begin{align*}
\big\|\psi^{(1)}-\psi^{(2)}\big\|_{\widetilde{**},\alpha,\beta;\sigma}
&\le C|t_0|^{c'(\alpha-a)} \big\|\phi^{(1)}-\phi^{(2)}\big\|_{\sharp\sharp,a,1;\sigma} +C|t_0|^{-2\vep_*} \big\|\psi^{(1)}-\psi^{(2)}\big\|_{\widetilde{**},\alpha,\beta;\sigma}.
\end{align*}
Absorbing the last term for $|t_0|$ large gives \eqref{psi-lip-3}.
\end{proof}

\section{The Inner Problem: Proof of Proposition \ref{prop:nonlin}}\label{sec:inner}
Throughout this section, we assume that $\mu$, $\dot\mu$, and $d$ satisfy \eqref{mu}--\eqref{dt} for all $t\in(-\infty,t_0]$, that
\eqref{mu2} holds, and that the points $\xi_j$ ($j=1,\dots,k$) are defined by \eqref{defpoints}.

We first state and prove the linear inner solvability result, then apply its solution operator to the nonlinear inner problem \eqref{eq:nonlin}.

\subsection{The linear inner problem}
With $\mcz_{n+1}$ defined by \eqref{eq:mczn1}, we consider the inhomogeneous linear problem
\be \label{eq:inhom-i}
{\bf L}[\mu,\dot\mu]\phi = U^{p-1}(\mch+c_{n+1}(t)\mcz_{n+1}) \quad \text{in } \R^n\times(-\infty,t_0)
\ee
subject to the orthogonality conditions \eqref{eq:inhom-i21}--\eqref{eq:inhom-i22}. Here, $c_{n+1}:(-\infty,t_0]\to\R$ is a time-dependent coefficient function determined together with $\phi$. We assume that the source $\mch$ satisfies
\be \label{mchin}
\text{\it the map } y=(y_1,\dots,y_n) \mapsto \mch(y,t) \text{\it is even with respect to } y_j \text{ \it for } j=2, \dots, n,
\ee
and
\be \label{mchin2}
\text{\it the map } x \mapsto \mu^{-\frac{n-2}{2}} \mch\big(\mu^{-1}(x-\xi_1),t\big) \text{ \it satisfies the Kelvin invariance } \eqref{fke}.
\ee

\begin{prop}[Linear inner problem] \label{prop:lin}
Suppose that the parameter assumptions in the opening paragraph of Subsection \ref{subsec:inner} hold.
Assume that $\tnu=\frac1{n-2}$, $a\in(0,n-2)$, $b>0$, $0<\sigma<\min\{1,a\}$, and that $\mch$ satisfies \eqref{mchin}--\eqref{mchin2} and $\|\mch\|_{\sharp,a+2,b;\sigma}<\infty$.
Take $t_0<0$ sufficiently large in magnitude, if necessary. Then there exist a function $\phi=\phi[\mch]$ and a scalar function $c_{n+1}=c_{n+1}[\mch]$ satisfying \eqref{eq:inhom-i}, the orthogonality conditions \eqref{eq:inhom-i21}--\eqref{eq:inhom-i22}, and
\be \label{phmh}
\|\phi\|_{\sharp\sharp,a,b;\sigma}+\|c_{n+1}\|_{b;\sigma} \le C\|\mch\|_{\sharp,a+2,b;\sigma},
\ee
where $C>0$ depends only on $n$, $a$, $b$, $\sigma$, and $\msM_0$. Furthermore, \eqref{mchin}--\eqref{mchin2} hold with $\mch$ replaced by $\phi$, and the map $\mch \mapsto (\phi[\mch],c_{n+1}[\mch])$ is linear.
The pair $(\phi,c_{n+1})$ is unique among the solutions satisfying \eqref{eq:inhom-i21}--\eqref{eq:inhom-i22} and $\|\phi\|_{\sharp\sharp,a,b;\sigma}<\infty$.
\end{prop}
\begin{proof}
The proof is divided into five steps. We use weighted pointwise estimates on $\Ss^n$ to construct the solution. The H\"older estimate in \eqref{phmh} is obtained only after returning to $\R^n$.

\medskip
\noindent \textbf{Step 1: Conformal formulation and the spectral conditions.}
Let $\hph$, $\whmch$, and $\whmcz_{n+1}$ denote the conformal lifts of $\phi$, $\mch$, and $\mcz_{n+1}$ under \eqref{stereo-proj-1}. Set
\be \label{eq:bmcl0}
\bmcl_0q := \pp_tq-\frac1n(\Delta_{\Ss^n}q+nq), \qquad \mcb_tq := A(z,t)q-V(z,t)\cdot\nabla_{\Ss^n}q,
\ee
where
\[
A(z,t) := \frac{n-2}{2}\mu(t)^{-1} \big(\dot\xi_1(t)\cdot\bz+\dot\mu(t)z_{n+1}\big)
\]
and
\be \label{eq:V}
\begin{aligned}
V(z,t) &:= \mu(t)^{-1}d\pi_y\big(\dot\xi_1(t)+\dot\mu(t)y\big) \\
&= \mu(t)^{-1}\((1-z_{n+1})\dot\xi_1(t)-[\dot\xi_1(t)\cdot\bz+\dot\mu(t)z_{n+1}]\bz, \right. \\
&\hspace{64pt} \left. [\dot\xi_1(t)\cdot\bz](1-z_{n+1})+\dot\mu(t)(1-z_{n+1}^2)\),
\end{aligned}
\ee
with $z=\pi(y) \in \Ss^n\setminus\{\msN_\infty\}$. In \eqref{eq:V}, the first entry inside the large parentheses is the $\R^n$-component, while the second is the $(n+1)$-st component. In particular, $V(\cdot,t)$ is tangent to $\Ss^n$.
A direct calculation shows that \eqref{eq:inhom-i} is equivalent to
\be \label{eq:lin-sph-eq}
\bmcl_0\hph+\mcb_t\hph = \frac1p\whmch+\frac1pc_{n+1}(t)\whmcz_{n+1} \quad \text{in } (\Ss^n\setminus\{\msN_\infty\})\times(-\infty,t_0).
\ee
We next translate \eqref{eq:inhom-i21}--\eqref{eq:inhom-i22}. By \eqref{eq:Zi},
\[
\frac{Z_j(y)}{U(y)}=-\frac{n-2}{2}z_j,\qquad j=1,\dots,n+1.
\]
Thus the conformal change of variables \eqref{stereo-proj-1} gives
\[
\int_{\Ss^n}\hph(z,t)z_j\,dv_{\Ss^n}(z)=0 \quad\Longleftrightarrow\quad \int_{\R^n}\phi(y,t)U^{p-1}(y)Z_j(y)\,dy=0
\]
for $j=1,\dots,n+1$, and
\[
\int_{\Ss^n}\hph(z,t_0)\,dv_{\Ss^n}(z)=0 \quad\Longleftrightarrow\quad \int_{\R^n}\phi(y,t_0)U^p(y)\,dy=0.
\]
Thus \eqref{eq:inhom-i21} requires $\hph(\cdot,t)$ to be orthogonal to the first spherical harmonics for every $t$, whereas \eqref{eq:inhom-i22} fixes the constant spherical component at the finite endpoint $t_0$.

\medskip
Let $\msP_0$ and $\msP_1$ be the orthogonal projections onto constants and first spherical harmonics. Their formulas extend to $F\in L^1(\Ss^n)$:
\be \label{eq:msP}
\begin{cases}
\displaystyle \msP_0F := \frac1{|\Ss^n|}\int_{\Ss^n}F\,dv_{\Ss^n},\\
\displaystyle \msP_1F := \frac{n+1}{|\Ss^n|}\sum_{j=1}^{n+1} \(\int_{\Ss^n}F(z)z_j\,dv_{\Ss^n}(z)\)z_j, \\
\displaystyle \msP_\perp F := F-\msP_0F-\msP_1F.
\end{cases}
\ee
Rather than solving \eqref{eq:lin-sph-eq} directly, we first construct a solution to the more general equation
\be \label{eq:lin-sph-eq-gen}
\bmcl_0\hph+\mcb_t\hph = \frac1p\whmch+\frac1p\sum_{j=1}^{n+1}d_j(t)z_j,
\ee
subject to the conditions
\be \label{eq:lin-sph-norm}
\msP_1\hph(\cdot,t)=0 \quad \text{for } t \in (-\infty,t_0] \quad \text{and} \quad \msP_0\hph(\cdot,t_0)=0,
\ee
which are equivalent to \eqref{eq:inhom-i21}--\eqref{eq:inhom-i22}. The scalar coefficients $d_j$ form the vector ${\bm d}(t) := (d_1(t),\dots,d_{n+1}(t))^T\in\R^{n+1}$ and will be determined in Step 3.
In Step 4, we use the reflection and Kelvin symmetries to show that the term $\sum_{j=1}^{n+1}d_j(t)z_j$ is a scalar multiple of $\whmcz_{n+1}$.

\medskip
Since $\dot\xi_1=-d^{-1}\mu\dot\mu\, {\bf e}_1$ and $d=(1-\mu^2)^{1/2}$, the assumptions \eqref{mu} and \eqref{mu2} imply, after increasing $|t_0|$ if necessary, that $d\ge\frac12$ and
\be \label{eq:lin-coeff-bound}
\|A(\cdot,t)\|_{C^1(\Ss^n)} + \|V(\cdot,t)\|_{C^1(\Ss^n)} \le C|t|^{-1}
\ee
for some $C > 0$. They also imply, whenever $t<t+\delta\le\min\{t_0,t+1\}$,
\be \label{eq:lin-coeff-tdiff}
\|A(\cdot,t+\delta)-A(\cdot,t)\|_{C^1(\Ss^n)} + \|V(\cdot,t+\delta)-V(\cdot,t)\|_{C^1(\Ss^n)} \le C|t|^{-1}\delta^{\sigma/2}.
\ee

To study \eqref{eq:lin-sph-eq-gen}--\eqref{eq:lin-sph-norm}, we use the weighted pointwise norms
\[
\begin{cases}
\displaystyle \|F\|_{\wtsh,a,b} := \sup_{z \in \Ss^n \setminus \{\msN_{\infty}\},\, t \in (-\infty,t_0]} |t|^b d_{\Ss^n}(z)^{n-a}|F(z,t)|,\\
\displaystyle \|q\|_{\wtshs,a+2,b} := \sum_{\ell=0}^1 \sup_{z \in \Ss^n \setminus \{\msN_{\infty}\},\, t \in (-\infty,t_0]} |t|^b d_{\Ss^n}(z)^{n-2-a+\ell} \left|\nabla_{\Ss^n}^\ell q(z,t)\right|,
\end{cases}
\]
where $d_{\Ss^n}(z)=d_{\Ss^n}(z,\msN_\infty)$. If $a \in (0,n]$, then $d_{\Ss^n}(z)^{a-n}$ is integrable over $\Ss^n$ and $d_{\Ss^n}(z)^{n-a}$ is bounded. It follows from \eqref{eq:msP} and the triangle inequality that
\be \label{eq:lin-proj-Linf}
\sum_{i\in\{0,1,\perp\}} \|\msP_iF\|_{\wtsh,a,b} \le C\|F\|_{\wtsh,a,b}.
\ee

On the other hand, it holds that
\be \label{eq:lin-ptcomp}
\|\whmch\|_{\wtsh,a,b} \le C\|\mch\|_{\sharp,a+2,b} \le C\|\mch\|_{\sharp,a+2,b;\sigma}
\ee
for any $\sigma \in (0,1)$. Also, if $\hph$ has finite $\wtshs$-norm, then its inverse conformal lift satisfies 
\be \label{eq:lin-Eucl-C1-bound}
|\phi(y,t)|+\la y\ra|\nabla_y\phi(y,t)| \le C|t|^{-b}\la y\ra^{-a} \|\hph\|_{\wtshs,a+2,b}.
\ee
Conversely, if $\phi$ has finite $\sharp\sharp$-H\"older norm, then its conformal lift satisfies
\be \label{eq:lin-rev-conf}
\|\hph\|_{\wtshs,a+2,b}\le C\|\phi\|_{\sharp\sharp,a,b;\sigma}.
\ee
Estimates \eqref{eq:lin-ptcomp}--\eqref{eq:lin-rev-conf} are consequences of \eqref{stereo-proj-1}, the uniform comparison 
\be \label{eq:dSsn}
d_{\Ss^n}(\pi(y))\asymp\la y\ra^{-1} \quad \text{for } y \in \R^n,
\ee
and the identity $|\nabla_y(\hph\circ\pi)| = \frac{2}{1+|y|^2}|(\nabla_{\Ss^n}\hph)\circ\pi|$.

\medskip
\noindent \textbf{Step 2: The constant and higher spherical modes.}
We write
\[
L := \frac1n(\Delta_{\Ss^n}+n) \quad \text{so that} \quad \bmcl_0=\pp_t-L.
\]
The constant functions have eigenvalue $1$ under $L$, the first spherical harmonics have eigenvalue $0$, and all spherical harmonics of degree at least two have negative eigenvalues.
The constant mode is therefore selected from the finite endpoint $t_0$, the first harmonics are removed by \eqref{eq:lin-sph-norm}, and the higher modes are selected from the ancient endpoint $-\infty$.

\medskip
For a scalar function $f_0$, we set
\[
(\mct_0f_0)(t) := -\int_t^{t_0}e^{t-\tit}f_0(\tit)\,d\tit.
\]
Then $e=\mct_0f_0$ solves $\bmcl_0e=e'-e=f_0$, $e(t_0)=0$, and
\be \label{eq:lin-T0-Linf}
\sup_{t\le t_0}|t|^b\big(|e(t)|+|e'(t)|\big) \le C\sup_{t\le t_0}|t|^b|f_0(t)|.
\ee

\medskip
We next construct the inverse of $\bmcl_0$ on the range of $\msP_\perp$. Let $\{Y_{\ell,m}\}_{m=1}^{\msd_\ell}$ be a real orthonormal basis of the spherical harmonics of degree $\ell$. For $\ell\ge2$,
\[
LY_{\ell,m}=-\ga_\ell Y_{\ell,m}, \qquad
\ga_\ell:=\frac{\ell(\ell+n-1)-n}{n} \ge \frac{n+2}{n} = \ga_2.
\]
For $\tau>0$, we define
\[
K_\perp(\tau,z,w) := \sum_{\ell=2}^\infty e^{-\ga_\ell\tau} \sum_{m=1}^{\msd_\ell}Y_{\ell,m}(z)Y_{\ell,m}(w).
\]
Given any $\tau_*>0$, the series and all its derivatives converge uniformly on $[\tau_*,\infty)\times\Ss^n\times\Ss^n$. For every $f\in L^2(\Ss^n)$, it holds that
\be \label{eq:hker}
\begin{aligned}
(e^{\tau L}\msP_\perp f)(z) &= \sum_{\ell=2}^\infty e^{-\ga_\ell\tau} \sum_{m=1}^{\msd_\ell} Y_{\ell,m}(z) \int_{\Ss^n}Y_{\ell,m}(w)f(w)\,dv_{\Ss^n}(w) \\
&= \int_{\Ss^n}K_\perp(\tau,z,w)f(w)\,dv_{\Ss^n}(w),
\end{aligned}
\ee
and hence $K_\perp$ is the smooth kernel of $e^{\tau L}\msP_\perp$ for every $\tau>0$. If $H(\tau,z,w)$ denotes the heat kernel of $e^{\tau\Delta_{\Ss^n}}$ for $\tau>0$, then
\[
K_\perp(\tau,z,w) = e^\tau H\(\frac{\tau}{n},z,w\) - \frac{e^\tau}{|\Ss^n|} - \frac{n+1}{|\Ss^n|}z\cdot w.
\]
The Gaussian derivative bounds for the heat kernel $H$ on $\Ss^n$, together with the spectral gap $\ga_\ell \ge \ga_2=\frac{n+2}{n}$ for $\ell\ge2$ yield that
\be \label{eq:kerest}
|\nabla_z^mK_\perp(\tau,z,w)| \le \begin{cases}
C\left[\tau^{-\frac{n+m}{2}} e^{-d_{\Ss^n}(z,w)^2/(C\tau)}+1\right] &\text{for } 0<\tau\le1,\\
Ce^{-\ga_2\tau} &\text{for } \tau\ge1,
\end{cases}
\ee
provided $m=0,1$. Integrating in $\tau$ yields
\be \label{eq:lin-int-kernel}
\int_0^\infty|\nabla_z^mK_\perp(\tau,z,w)|\,d\tau \le C\big(1+d_{\Ss^n}(z,w)^{2-n-m}\big), \qquad m=0,1.
\ee
For $0<a<n-2$, the elementary weighted potential estimate
\be \label{eq:lin-wpot}
\int_{\Ss^n} \big(1+d_{\Ss^n}(z,w)^{2-n-m}\big) d_{\Ss^n}(w)^{a-n}\,dv_{\Ss^n}(w) \le C d_{\Ss^n}(z)^{a+2-n-m}
\ee
holds for $m=0,1$. To verify this when $z$ is near $\msN_\infty$, we decompose $\Ss^n$ into the three regions:
$\{w:\ d_{\Ss^n}(z,w)<\frac12{d_{\Ss^n}(z)}\}$, $\{w:\ d_{\Ss^n}(w)\ge2d_{\Ss^n}(z)\}$, and $\{w:\ d_{\Ss^n}(w)<2d_{\Ss^n}(z),\ d_{\Ss^n}(z,w)\ge\frac12d_{\Ss^n}(z)\}$.
The integral over each region is bounded by the right-hand side of \eqref{eq:lin-wpot}. When $z$ stays away from $\msN_\infty$, the estimate follows directly from compactness.

For $F$ with finite $\wtsh$-norm, we define
\be \label{eq:hker2}
(\mct_\perp F)(z,t) := \int_0^\infty\int_{\Ss^n} K_\perp(\tau,z,w)F(w,t-\tau) \,dv_{\Ss^n}(w)\,d\tau.
\ee
Since $|t-\tau|\ge|t|$, estimates \eqref{eq:lin-int-kernel}--\eqref{eq:lin-wpot} imply that the integral is absolutely convergent for every $z\in \Ss^n \setminus \{\msN_\infty\}$ and satisfies
\be \label{eq:lin-Tperp-Linf}
\|\mct_\perp F\|_{\wtshs,a+2,b} \le C\|F\|_{\wtsh,a,b}.
\ee
Let $\vph=\mct_\perp F$. Since $F(\cdot,\tit)\in L^1(\Ss^n)$, changing variables $\tit=t-\tau$ in \eqref{eq:hker}, integrating the result over $\tit \in (-\infty,t)$, and using \eqref{eq:hker2} give
\[
\vph(\cdot,t) = \int_{-\infty}^te^{(t-\tit)L}\msP_\perp F(\cdot,\tit)\,d\tit \\
= e^{(t-T)L}\vph(\cdot,T)+\int_T^te^{(t-\tit)L}\msP_\perp F(\cdot,\tit)\,d\tit
\]
in $L^1(\Ss^n)$ for every $T<t\le t_0$. Equivalently,
\[
\bmcl_0\vph=\msP_\perp F, \qquad \msP_0\vph=\msP_1\vph=0,
\]
in the distributional sense. If $F$ is locally H\"older continuous away from $\msN_\infty$, standard interior parabolic regularity makes $\vph$ a classical solution there.
Furthermore, the solution is unique within this weighted class. Indeed, the difference between any two solutions contains only spherical-harmonic modes of degree $\ell \ge 2$, with coefficients proportional to $e^{-\ga_\ell t}$.
These coefficients grow exponentially as $t \to -\infty$, whereas the $\wtshs$-norm gives an $O(|t|^{-b})$ bound. Therefore, all such coefficients must be zero.

\medskip
\noindent \textbf{Step 3: Existence and uniqueness of the spherical solution.}
We will construct a solution $\hph$ to \eqref{eq:lin-sph-eq-gen}--\eqref{eq:lin-sph-norm} of the form $\hph=e+\vph$, where $e$ and $\vph$ satisfy
\be \label{eq:evarphi}
e = \mct_0\left[\frac1p\msP_0\whmch-\msP_0\mcb_t(e+\vph)\right],\qquad
\vph = \mct_\perp\left[\frac1p\whmch-\mcb_t(e+\vph)\right].
\ee
The second formula includes the projection $\msP_\perp$ through the definition of $\mct_\perp$. We equip pairs $(e,\vph)$ with the norm $\sup_{t \in (-\infty,t_0]}|t|^b|e(t)| +\|\vph\|_{\wtshs,a+2,b}$. By \eqref{eq:lin-coeff-bound},
\be \label{eq:lin-B-small}
\|\mcb_t(e+\vph)\|_{\wtsh,a,b} \le C|t_0|^{-1} \(\sup_{t \in (-\infty,t_0]}|t|^b|e(t)|+\|\vph\|_{\wtshs,a+2,b}\).
\ee

Equations \eqref{eq:lin-proj-Linf}, \eqref{eq:lin-T0-Linf}, \eqref{eq:lin-Tperp-Linf}, and
\eqref{eq:lin-B-small} show that the map defined by the right-hand sides of the two equations in \eqref{eq:evarphi} is a contraction when $|t_0|$ is sufficiently large.
Let $(e,\vph)$ be its unique fixed point and set $\hph=e+\vph$. Then \eqref{eq:lin-sph-norm} holds and
\be \label{eq:lin-sph-Linf-est}
\|\hph\|_{\wtshs,a+2,b} \le C\|\whmch\|_{\wtsh,a,b}.
\ee

Equation \eqref{eq:evarphi} implies that the $\msP_0$- and $\msP_\perp$-components of $\bmcl_0\hph+\mcb_t\hph-\frac1p\whmch$ vanish. On the other hand, since $\msP_1\hph=0$, we have $\langle\hph,z_j\rangle=0$, and hence, in the distributional sense,
\[
\big\langle \bmcl_0\hph,z_j \big\rangle =\frac d{dt} \big\langle \hph,z_j \big\rangle -\frac1n \big\langle \hph,(\Delta_{\Ss^n}+n)z_j \big\rangle=0,
\]
where $(\Delta_{\Ss^n}+n)z_j=0$. Therefore the first-harmonic coefficients of $p(\bmcl_0\hph+\mcb_t\hph)-\whmch$ are given by
\be \label{eq:lin-dj}
d_j(t) := \frac{n+1}{|\Ss^n|} \left[p\int_{\Ss^n}(\mcb_t\hph)(z,t)z_j\,dv_{\Ss^n}(z) - \int_{\Ss^n}\whmch(z,t)z_j\,dv_{\Ss^n}(z)\right]
\ee
for $j=1,\dots,n+1$. This proves that \eqref{eq:lin-sph-eq-gen} holds distributionally. From the assumption $a>0$, \eqref{eq:lin-coeff-bound}, and \eqref{eq:lin-sph-Linf-est}, we obtain
\be \label{eq:lin-dj-Linf}
\sup_{t \in (-\infty,t_0]} |t|^b \sum_{j=1}^{n+1}|d_j(t)| \le C\|\whmch\|_{\wtsh,a,b}.
\ee

The contraction proves uniqueness of $(e,\vph)$, and \eqref{eq:lin-dj} then gives uniqueness of the coefficient vector ${\bm d}(t)$.

\medskip
\noindent \textbf{Step 4: Reflection and Kelvin symmetries.}
For $j=2,\dots,n$, let $\mscS_j$ denote reflection in the $z_j$-coordinate. By \eqref{mchin} and \eqref{eq:bmcl0}--\eqref{eq:V}, the source $\whmch$ is invariant and $\bmcl_0+\mcb_t$ is covariant under $\mscS_j$, i.e.,
\[
(\bmcl_0+\mcb_t)(q\circ\mscS_j)=((\bmcl_0+\mcb_t)q)\circ\mscS_j
\]
for every smooth function $q$ on $\Ss^n\times(-\infty,t_0]$, and hence distributionally. Hence $\hph\circ\mscS_j$, together with the coefficient vector obtained by replacing $d_j$ by $-d_j$, solves \eqref{eq:lin-sph-eq-gen}--\eqref{eq:lin-sph-norm}.
Since $\mscS_j$ fixes $\msN_\infty$, uniqueness from Step 3 gives $\hph\circ\mscS_j=\hph$ and $d_j=0$ for all $j=2,\dots,n$.

For each fixed $t \in (-\infty,t_0]$, we put
\be \label{eq:lin-Kelvin-map}
\mathbf n(t) := d(t){\bf e}_1+\mu(t){\bf e}_{n+1} \quad \text{and} \quad
\mck_tz := z-2\langle z,\mathbf n(t)\rangle\mathbf n(t) \quad \text{for } z\in\Ss^n.
\ee
Since $d^2+\mu^2=1$, $\mck_t$ is the orthogonal reflection across $\mathbf n(t)^\perp$. Equivalently, it fixes $z_2,\dots,z_n$ and acts on $(z_1,z_{n+1})$ by
\be \label{eq:Mt}
M_t:= \begin{pmatrix}
2\mu(t)^2-1 &-2\mu(t)d(t)\\
-2\mu(t)d(t) &1-2\mu(t)^2
\end{pmatrix} \in O(2).
\ee
The map $\mck_t$ is induced by Kelvin inversion in the original Euclidean variable $x$, as verified below.
Since $\mck_t$ does not in general fix $\msN_\infty$, the weighted class used in Step 3, whose weight is centered at $\msN_\infty$, is not invariant under composition with $\mck_t$.
To prove the Kelvin invariance of $\hph$ and identify the transformed first-harmonic coefficients, we therefore invoke a uniqueness result in an $L^1$ class that is invariant under composition with $\mck_t$.

\medskip
Suppose that $q(\cdot,t)\in L^1(\Ss^n)$ for a.e. $t \in (-\infty,t_0]$ and $q$ satisfies
\be \label{eq:lin-sph-eq-gen2}
\bmcl_0q+\mcb_tq = \frac1p\sum_{j=1}^{n+1}d_j(t)z_j, \qquad
\msP_1q=0, \quad \msP_0q(\cdot,t_0)=0
\ee
in the distributional sense on $\Ss^n\times(-\infty,t_0)$, where $d_j\in L^1_{\rm loc}((-\infty,t_0])$. The terminal condition is imposed on the absolutely continuous representative of $\msP_0q$. We also assume that
\be \label{eq:N1}
N_1 := \esssup_{t \in (-\infty,t_0]} |t|^b\|q(\cdot,t)\|_{L^1(\Ss^n)}<\infty.
\ee
Here, the first-order term is defined by
\be \label{eq:divform}
\mcb_tq=(A+\operatorname{div}_{\Ss^n}V)q-\operatorname{div}_{\Ss^n}(Vq).
\ee
Since $\sum_{j=1}^{n+1}d_j(t)z_j$ belongs to the first spherical-harmonic subspace, it is annihilated by both $\msP_0$ and $\msP_\perp$.
Testing \eqref{eq:lin-sph-eq-gen2} against the constant function $1$ gives $(\msP_0q)'-\msP_0q=-\msP_0(\mcb_tq)$, and it follows from $\msP_0q(t_0)=0$ that
\be \label{eq:msP0q}
\msP_0q(t)=\int_t^{t_0}e^{t-\tit}\msP_0(\mcb_{\tit}q(\tit))\,d\tit \quad \text{and} \quad |\msP_0q(t)|\le C|t_0|^{-1}N_1|t|^{-b}.
\ee
For the higher modes, testing \eqref{eq:lin-sph-eq-gen2} against $e^{(t-\tit)L}\msP_\perp\zeta$ for $\zeta\in C^\infty(\Ss^n)$ and integrating
from $T$ to $t$ yield the weak Duhamel identity:
\be \label{eq:weakDu}
\msP_\perp q(t) = e^{(t-T)L}\msP_\perp q(T)-\int_T^t e^{(t-\tit)L}\msP_\perp(\mcb_{\tit}q(\tit))\,d\tit
\ee
for almost every $T<t\le t_0$. The divergence formulation \eqref{eq:divform} makes \eqref{eq:weakDu} well defined in the distributional sense without assuming $\nabla q(\cdot,t) \in L^1(\Ss^n)$. For $f\in L^1(\Ss^n)$ and an integrable tangent vector field $X$ on $\Ss^n$, \eqref{eq:hker} and \eqref{eq:kerest} imply
\begin{align}
\|e^{\tau L}\msP_\perp f\|_{L^1(\Ss^n)} &\le Ce^{-\ga_2\tau}\|f\|_{L^1(\Ss^n)}, \label{eq:hkerL1} \\
\|e^{\tau L}\msP_\perp\operatorname{div}_{\Ss^n}X\|_{L^1(\Ss^n)} &\le C(1+\tau^{-1/2})e^{-\ga_2\tau}\|X\|_{L^1(\Ss^n)}. \label{eq:hkerL12}
\end{align}
Inserting $f=q(\cdot,T)$ and $\tau=t-T$ into \eqref{eq:hkerL1}, we obtain $e^{(t-T)L}\msP_\perp q(T) \to 0$ in $L^1(\Ss^n)$ as $T\to-\infty$, which combined with \eqref{eq:weakDu} implies
\be \label{eq:weakDu2}
\msP_\perp q(t) = -\int_{-\infty}^t e^{(t-\tit)L}\msP_\perp(\mcb_{\tit}q(\tit))\,d\tit.
\ee
Using \eqref{eq:weakDu2}, \eqref{eq:divform}, \eqref{eq:lin-coeff-bound}, \eqref{eq:hkerL1}--\eqref{eq:hkerL12}, and the elementary bound
\[
\int_{-\infty}^t(1+(t-\tit)^{-1/2})e^{-c(t-\tit)}|\tit|^{-b-1}\,d\tit \le C|t_0|^{-1}|t|^{-b},
\]
we obtain
\be \label{eq:msPperqest}
\|\msP_\perp q(\cdot,t)\|_{L^1} \le C|t_0|^{-1}N_1|t|^{-b}.
\ee
Since $q=\msP_0q+\msP_\perp q$, we infer from \eqref{eq:msP0q} and \eqref{eq:msPperqest} that $N_1\le C|t_0|^{-1}N_1$.
Hence, for $|t_0|$ sufficiently large, $N_1=0$, and therefore $q=0$. Substituting this into \eqref{eq:lin-sph-eq-gen2} then yields $d_j=0$ for $j=1,\dots,n+1$. This proves the $L^1$-uniqueness result.

The solution $\hph$ constructed in Step 3 satisfies \eqref{eq:N1} with $q=\hph$. Indeed, by \eqref{eq:lin-sph-Linf-est} and the assumption $a>0$, we have that $N_1\le C\|\whmch\|_{\wtsh,a,b}$.

\medskip
We now relate \eqref{eq:lin-Kelvin-map} to Kelvin inversion in the $x$-variable.
Let $\phi$ be the inverse conformal lift of $\hph$ and define $\tph_1$ by \eqref{tildephi-ass-3}.
Let $\tph_1^*$ be the Kelvin transform of $\tph_1$ in the $x$-variable, as given by \eqref{fke}. Writing $x=\mu y+\xi_1=\mu y+d{\bf e}_1$, we define 
\[
\begin{gathered}
\wtmch(x,t) := \mu(t)^{-(n-2)/2}\mch(y,t), \qquad \wtz_j(x,t) := \mu(t)^{-(n-2)/2}Z_j(y),\\
\mscL_1w := pU_1^{p-1}\pp_tw-\Delta w-p\frac{n+2}{4}U_1^{p-1}w.
\end{gathered}
\]
Then
\[
(\mscL_1\tph_1)(x,t) = \mu(t)^{-\frac{n+2}{2}} ({\bf L}[\mu,\dot\mu]\phi)(y,t) = U_1^{p-1} \(\wtmch+\sum_{j=1}^{n+1}\beta_j\wtz_j\)(x,t),
\]
where ${\bf L}$ is the operator in \eqref{inner} and the coefficients $\beta_j$ are fixed dimensional multiples of $d_j$, hence locally bounded by \eqref{eq:lin-dj-Linf}. 
Using the Kelvin invariance of $U_1$ and assumption \eqref{mchin2}, and denoting by $\wtz_j^*$ the Kelvin transform of $\wtz_j$, we obtain
\be \label{eq:lin-Kelvin-cov}
\mscL_1\tph_1^* = U_1^{p-1}\(\wtmch+\sum_{j=1}^{n+1}\beta_j\wtz_j^*\) \quad \text{on } (\R^n\setminus\{0\})\times(-\infty,t_0).
\ee
For every compact time interval $I \subset (-\infty,t_0]$, \eqref{eq:lin-Eucl-C1-bound} yields
\[
|\tph_1^*(x,t)|\le C_I|x|^{a+2-n} \quad \text{and} \quad |\nabla\tph_1^*(x,t)|\le C_I|x|^{a+1-n},
\]
where $C_I > 0$ is determined by the interval $I$. Moreover, the right-hand side of \eqref{eq:lin-Kelvin-cov} is locally bounded on $\R^n \times I$.
Testing with a cutoff that vanishes on $B(0,\epsilon)$ and equals one outside $B(0,2\epsilon)$ produces a spatial error bounded by $C_I\epsilon^a$. Since $a>0$, no singular distribution occurs at $x=0$.
As a result, \eqref{eq:lin-Kelvin-cov} holds on $\R^n\times(-\infty,t_0)$.

We introduce the $y$-variable corresponding to the inverted point $x/|x|^2$:
\[
y^\sharp := \frac{x/|x|^2-d{\bf e}_1}{\mu}.
\]
A direct calculation gives
\be \label{eq:ysharp}
1+|y^\sharp|^2=\frac{1+|y|^2}{|x|^2} \quad \text{and} \quad \pi(y^\sharp)=\mck_t(\pi(y)).
\ee
We define
\[
\phi^\sharp(y,t) := \mu(t)^{\frac{n-2}{2}}\tph_1^*(\mu(t)y+\xi_1(t),t) = |x|^{2-n}\phi(y^\sharp,t).
\]
By the definition \eqref{stereo-proj-1} of the conformal lift and \eqref{eq:ysharp},
\[
\hph^\sharp=\hph\circ\mck_t.
\]
After returning from the $x$-variable to the $y$-variable and taking the conformal lift, \eqref{eq:lin-Kelvin-cov} reads
\[
\bmcl_0\hph^\sharp+\mcb_t\hph^\sharp = \frac1p\whmch\circ\mck_t+\frac1p\sum_{j=1}^{n+1}d_j(t)(z_j\circ\mck_t) = \frac1p\whmch+\frac1p\sum_{j=1}^{n+1}d_j^\sharp(t) z_j.
\]
Here, \eqref{mchin2} ensures that $\whmch\circ\mck_t=\whmch$. Besides, the action of $\mck_t$ on the first spherical harmonics is
\[
\binom{z_1\circ\mck_t}{z_{n+1}\circ\mck_t} = M_t\binom{z_1}{z_{n+1}}
\quad \text{and} \quad z_j\circ\mck_t=z_j \quad \text{for } j=2,\dots,n,
\]
where the matrix $M_t$ is given in \eqref{eq:Mt}. Then,
\[
\sum_{j=1}^{n+1}d_j(t)(z_j\circ\mck_t) = \sum_{j=1}^{n+1}d_j^\sharp(t) z_j,
\]
where
\be \label{eq:bmdeq}
\binom{d_1^\sharp}{d_{n+1}^\sharp}=M_t\binom{d_1}{d_{n+1}}
\quad \text{and} \quad d_j^\sharp=d_j \quad \text{for } j=2,\dots,n.
\ee
The isometry $\mck_t$ preserves spherical measure $dv_{\Ss^n}$, as well as the constant and first-harmonic subspaces.
Thus $\hph^\sharp$ satisfies the same terminal zero-mean condition and the same first-harmonic orthogonality conditions as $\hph$. 
Therefore, with ${\bm d}^\sharp := (d_1^\sharp,\dots,d_{n+1}^\sharp)^T$, the pair $(\hph^\sharp,{\bm d}^\sharp)$ is another solution of \eqref{eq:lin-sph-eq-gen}--\eqref{eq:lin-sph-norm}. The isometry $\mck_t$ also preserves the $L^1$-norm:
\[
\|\hph^\sharp(\cdot,t)\|_{L^1(\Ss^n)} = \|\hph(\cdot,t)\|_{L^1(\Ss^n)}.
\]

Applying the preceding $L^1$-uniqueness result to $\hph-\hph^\sharp$ gives
\[
\phi^\sharp=\phi \quad \text{and} \quad d_j^\sharp=d_j \quad \text{for } j=1,\dots,n+1.
\]
By \eqref{eq:bmdeq}, $d_1=-\mu d_{n+1}/d$. Since
\[
\whmcz_{n+1}(z,t)=-\frac{(n-2)\msc_n}{2^{n/2}} d\big(-\mu z_1+d z_{n+1}\big),
\]
we obtain
\be \label{eq:bmc}
\sum_{j=1}^{n+1}d_jz_j=c_{n+1}(t)\whmcz_{n+1},
\qquad c_{n+1}(t)=-\frac{2^{n/2}d_{n+1}(t)}{(n-2)\msc_n d(t)^2}.
\ee
Equation \eqref{eq:lin-dj-Linf} and $d\ge\frac12$ imply
\be \label{eq:lin-cn1-Linf}
\sup_{t \in (-\infty,t_0]}|t|^b|c_{n+1}(t)| \le C\|\whmch\|_{\wtsh,a,b}.
\ee

\medskip
\noindent\textbf{Step 5: Temporal regularity and the Euclidean Schauder estimate.}
By \eqref{eq:lin-ptcomp}, \eqref{eq:lin-sph-Linf-est}, and \eqref{eq:lin-Eucl-C1-bound},
\be \label{eq:lin-pre-Eucl-bound}
|\phi(y,t)|+\la y\ra|\nabla_y\phi(y,t)| \le C\|\mch\|_{\sharp,a+2,b;\sigma}|t|^{-b}\la y\ra^{-a},
\ee
and \eqref{eq:lin-cn1-Linf} gives $|c_{n+1}(t)|\le C\|\mch\|_{\sharp,a+2,b;\sigma}|t|^{-b}$.

For $0<\delta\le1$ and $t+\delta\le t_0$, we write $D_\delta q(z,t):=q(z,t+\delta)-q(z,t)$. The temporal part of the $\sharp$-H\"older-norm, \eqref{stereo-proj-1}, and \eqref{eq:dSsn} give
\be \label{eq:lin-source-tdiff}
|D_\delta \whmch(z,t)| \le C\|\mch\|_{\sharp,a+2,b;\sigma}|t|^{-b}\delta^{\sigma/2} d_{\Ss^n}(z)^{a-n-\sigma}.
\ee
Since $0<a-\sigma<n-2$, the estimate \eqref{eq:lin-Tperp-Linf} applies with $a$ replaced by $a-\sigma$.

\medskip
To improve the temporal regularity of $\phi$, we repeat the fixed-point argument used to solve \eqref{eq:evarphi} in Step 3 in the stronger space $X_\sigma$ with norm
\begin{multline*}
\|(e,\vph)\|_{X_\sigma} := \sup_{t\le t_0}|t|^b|e(t)|+\|\vph\|_{\wtshs,a+2,b} \\
+\sup_{\substack{0<\delta\le1,\ t+\delta\le t_0,\\ z\in\Ss^n\setminus\{\msN_\infty\}}} |t|^b\delta^{-\sigma/2}\Bigl(|D_\delta e(t)| +\sum_{\ell=0}^1d_{\Ss^n}(z)^{n-2-a+\sigma+\ell} |\nabla_{\Ss^n}^{\ell}D_\delta\vph(z,t)|\Bigr).
\end{multline*}
Let $\mct_{0,\perp}(e,\vph)$ denote the pair given by the right-hand sides of \eqref{eq:evarphi}, and write $(\te,\tvph)=\mct_{0,\perp}(e,\vph)$. Since $\mct_\perp$ in \eqref{eq:hker2} commutes with time translations, for $0<\delta\le1$ and $t+\delta\le t_0$, we have
\[
D_\delta \tvph = \mct_\perp\left[\frac1p D_\delta \whmch - \mcb_{t+\delta}\big(D_\delta (e+\vph)\big) - (\mcb_{t+\delta}-\mcb_t)(e+\vph)\right].
\]
Here, \eqref{eq:lin-Tperp-Linf} is applied on $(-\infty,t_0-\delta]$; its proof uses endpoint-independent kernel bounds and $|t-\tau|\ge|t|$ for $\tau\ge0$, so the constant is independent of $t_0-\delta$ and hence of $\delta$.

Together with the scalar estimate $|D_\delta \te(t)| \le \delta\sup_{\tit\in[t,t+\delta]}|\te'(\tit)|$, the bounds \eqref{eq:lin-coeff-bound}--\eqref{eq:lin-coeff-tdiff} yield
\begin{align*}
\|\mct_{0,\perp}(e,\vph)\|_{X_\sigma} &\le C\|\mch\|_{\sharp,a+2,b;\sigma}+C|t_0|^{-1}\|(e,\vph)\|_{X_\sigma}, \\
\left\|\mct_{0,\perp}(e^{(1)},\vph^{(1)})-\mct_{0,\perp}(e^{(2)},\vph^{(2)})\right\|_{X_\sigma} &\le C|t_0|^{-1} \big\|(e^{(1)}-e^{(2)},\vph^{(1)}-\vph^{(2)})\big\|_{X_\sigma}
\end{align*}
for any two pairs $(e^{(1)},\vph^{(1)})$ and $(e^{(2)},\vph^{(2)})$ in $X_\sigma$. Thus, after increasing $|t_0|$, $\mct_{0,\perp}$ is a contraction on $X_\sigma$.
Its fixed point is also a fixed point of \eqref{eq:evarphi} in the space used in Step 3 and therefore, by the uniqueness established there, coincides with the previously constructed solution. Therefore,
\be \label{eq:lin-sph-tdiff}
\sup_{\substack{0<\delta\le1,\ t+\delta\le t_0,\\z\in\Ss^n\setminus\{\msN_\infty\}}} |t|^b\delta^{-\sigma/2} \sum_{\ell=0}^1d_{\Ss^n}(z)^{n-2-a+\sigma+\ell} |\nabla_{\Ss^n}^\ell D_\delta \hph(z,t)| \le C\|\mch\|_{\sharp,a+2,b;\sigma}.
\ee
Taking time differences in \eqref{eq:lin-dj} and using \eqref{eq:lin-coeff-bound}--\eqref{eq:lin-coeff-tdiff}, \eqref{eq:lin-source-tdiff}, and \eqref{eq:lin-sph-tdiff}, we arrive at
\[
\|d_j\|_{b;\sigma}\le C\|\mch\|_{\sharp,a+2,b;\sigma} \quad \text{for } j=1,\dots,n+1.
\]
The formula \eqref{eq:bmc} and the parameter estimates for $\mu$ and $d$ then yield
\be \label{eq:lin-cn1-holder}
\|c_{n+1}\|_{b;\sigma}\le C\|\mch\|_{\sharp,a+2,b;\sigma}.
\ee

\medskip
Let us estimate $\phi$ in the full $\sharp\sharp$-H\"older-norm. Dividing \eqref{eq:inhom-i} by $pU^{p-1}$, we see that
\[
\pp_t\phi -\frac1pU^{1-p}\Delta\phi -\mu^{-1}(\dot\xi_1+\dot\mu y)\cdot\nabla\phi -\(\frac{n+2}{4}+\frac{n-2}{2}\mu^{-1}\dot\mu\)\phi =\frac1p\big(\mch+c_{n+1}\mcz_{n+1}\big).
\]
We fix $(y_0,t_0') \in \R^n \times (-\infty,t_0]$ and set $\lambda_0 := \la y_0\ra$. For $(y,t) \in B(y_0,3\lambda_0/4)\times[t_0'-2\lambda_0^{-2},t_0']$, we put
\[
Y := \lambda_0^{-1}(y-y_0), \qquad T := \lambda_0^2(t-t_0'), \qquad \Phi(Y,T) := |t_0'|^b\lambda_0^a\phi(y,t).
\]
On this space--time cylinder, $\la y\ra\asymp\lambda_0$, $|t|\asymp|t_0'|$, and $U^{1-p}(y)=\msc_n^{1-p}\la y\ra^4$. Hence the rescaled equation is uniformly parabolic, and its coefficients are uniformly bounded in $C^{\sigma,\sigma/2}$.
The rescaled contribution of $\mch$ has $C^{\sigma,\sigma/2}$-norm at most $C\|\mch\|_{\sharp,a+2,b;\sigma}$.
The explicit decay of $Z_1,Z_{n+1}$ and \eqref{eq:lin-cn1-holder} give $\|c_{n+1}\mcz_{n+1}\|_{\sharp,a+2,b;\sigma}\le C\|\mch\|_{\sharp,a+2,b;\sigma}$.
As a consequence, the interior parabolic Schauder estimate and \eqref{eq:lin-pre-Eucl-bound} yield
\[
\|\Phi\|_{C^{2+\sigma,1+\sigma/2}(B(0,1/2)\times[-1,0])} \le C\(\|\Phi\|_{L^\infty(B(0,3/4)\times[-2,0])}+\|\mch\|_{\sharp,a+2,b;\sigma}\) \le C\|\mch\|_{\sharp,a+2,b;\sigma}.
\]
With $r_1(y_0)=\lambda_0/2$ as in Definition \ref{def:inner-radius}, we scale back and use the pointwise derivative bounds for time differences larger than $\lambda_0^{-2}$ to obtain
\[
\|\phi\|_{\sharp\sharp,a,b;\sigma}\le C\|\mch\|_{\sharp,a+2,b;\sigma}.
\]
Together with \eqref{eq:lin-cn1-holder}, this proves \eqref{phmh}.

\medskip
Every operation in the construction is linear in $\mch$.

Let $(\phi^{(1)},c_{n+1}^{(1)})$ and $(\phi^{(2)},c_{n+1}^{(2)})$ be two solution pairs with the same source satisfying the stated uniqueness conditions.
By \eqref{eq:lin-rev-conf}, the conformal lift of $\phi^{(1)}-\phi^{(2)}$ satisfies the $L^1$ bound \eqref{eq:N1} required in Step 4.
Subtracting their equations \eqref{eq:inhom-i} and evaluating the resulting identity at $y=0$, where $\mcz_{n+1}(0,t)=d(t)^2Z_{n+1}(0)\ne0$, shows that $c_{n+1}^{(1)}-c_{n+1}^{(2)}$ is locally bounded in time.
The equation for this conformal lift extends distributionally across $\msN_\infty$ by the standard cutoff argument, and the first-harmonic orthogonality and terminal zero-mean conditions also hold.
The $L^1$-uniqueness result of Step 4 therefore gives $\phi^{(1)}=\phi^{(2)}$, and \eqref{eq:inhom-i} then yields $c_{n+1}^{(1)}=c_{n+1}^{(2)}$. This proves uniqueness in the class stated in Proposition \ref{prop:lin}.
\end{proof}

\subsection{The nonlinear inner problem}\label{subsec:innernonlin}
To solve \eqref{eq:nonlin} under the conditions \eqref{eq:inhom-i21}--\eqref{eq:inhom-i22}, we apply Proposition \ref{prop:lin} to $\mch_{\mathrm{inner}}$ in \eqref{mh1}, with the outer solution from Proposition \ref{prop:outer}.

\begin{lemma}[Estimate of $\mch_{\mathrm{inner}}$]
For each fixed $\msM_1>0$, suppose that $\|\phi\|_{\sharp\sharp,a,1;\sigma}\le \msM_1$. Then, for sufficiently negative $t_0$,
\be \label{le2.16}
\|\mch_{\mathrm{inner}}\|_{\sharp,a+2,1;\sigma}\le C_{\msM_1},
\ee
where $C_{\msM_1}>0$ depends only on $n$, $k$, the auxiliary parameters fixed in Definition \ref{def:num}, $\msM_0$, and $\msM_1$.
\end{lemma}
\begin{proof}
Put $x=\xi_1+\mu y$. On the support of $\eta_{1,2}$, it holds that $\la y\ra\le CR=C|t|^{c_0}$.
By Lemma \ref{size error}, $\mu^{\frac{n+2}{2}}|E(\xi_1+\mu y,t)| \le C|t|^{-1}\la y\ra^{-4}$. Since $a\le2$, it follows that
\[
|t|\la y\ra^{a+2} \mu^{\frac{n+2}{2}} |(\eta_{1,2}E)(\xi_1+\mu y,t)| \le C.
\]
The explicit formulas \eqref{est-error-in-0}, \eqref{tildeE1}, and \eqref{tildeE2_def}, together with \eqref{mu2}, give the corresponding spatial and temporal H\"older estimates. Therefore
\be \label{eq:H1}
\left\|U^{1-p}\mu^{\frac{n+2}{2}} (\eta_{1,2}E)(\xi_1+\mu y,t)\right\|_{\sharp,a+2,1;\sigma} \le C.
\ee

The term containing $\psi$ satisfies
\be \label{eq:H2}
\left\|U^{1-p}\mu^{\frac{n+2}{2}} [\eta_{1,2}u_*^{p-1}\psi](\xi_1+\mu y,t)\right\|_{\sharp,a+2,1;\sigma} \le C|t_0|^{-(\beta-\frac12)} \|\psi\|_{\widetilde{**},\alpha,\beta;\sigma},
\ee
where $\beta-\frac12=(c_0-c')(a-\alpha)$. The right-hand side is uniformly bounded by Proposition \ref{prop:outer}.

For the last term in \eqref{mh1}, on the support of $\eta_{1,2}$ we have
\[
\left|\frac{\mu^{\frac{n-2}{2}}u_*(\xi_1+\mu y,t)}{U(y)}-1\right| \le C|t|^{-1}\la y\ra^{n-2}=o(1).
\]
Hence the mean-value formula gives
\begin{multline}\label{eq:H41}
\mu^2\left|u_*^{p-1}(\xi_1+\mu y,t)-U_1^{p-1}(\xi_1+\mu y,t)\right| \\
= \left|\big(\mu^{\frac{n-2}{2}}u_*(\xi_1+\mu y,t)\big)^{p-1}-U^{p-1}(y)\right| \le CU^{p-1}(y)|t|^{-1}\la y\ra^{n-2}.
\end{multline}
Using \eqref{tildephi-ass-3}, we also have
\be \label{eq:H42}
\left|\mu^{\frac{n-2}{2}} \(\pp_t\tph_1-\frac{n+2}{4}\tph_1\)(\xi_1+\mu y,t)\right| \le C|t|^{-1}\la y\ra^{2-a} \|\phi\|_{\sharp\sharp,a,1;\sigma}.
\ee
Estimates \eqref{eq:H41}--\eqref{eq:H42} give
\begin{multline*}
|t|\la y\ra^{a+2}U^{p-1}(y) \left|U^{1-p}\mu^{\frac{n+2}{2}}\eta_{1,2}(u_*^{p-1}-U_1^{p-1}) \(\pp_t\tph_1-\frac{n+2}{4}\tph_1\)(\xi_1+\mu y,t)\right| \\
\le C|t|^{-1}\la y\ra^{n-2} \mone_{\{\la y\ra\le CR\}} \|\phi\|_{\sharp\sharp,a,1;\sigma} \le C|t|^{-[1-c_0(n-2)]} \|\phi\|_{\sharp\sharp,a,1;\sigma}.
\end{multline*}
The corresponding spatial $C^\sigma$ and temporal $C^{\sigma/2}$ estimates follow analogously. Therefore, if $\|\phi\|_{\sharp\sharp,a,1;\sigma}\le\msM_1$, then
\begin{multline} \label{eq:H4}
\left\|U^{1-p}\mu^{\frac{n+2}{2}}\eta_{1,2} (u_*^{p-1}-U_1^{p-1}) \(\pp_t\tph_1-\frac{n+2}{4}\tph_1\)(\xi_1+\mu y,t) \right\|_{\sharp,a+2,1;\sigma} \\
\le C_{\msM_1}|t_0|^{-[1-c_0(n-2)]}.
\end{multline}

Finally, Appendix \ref{app2} gives
\be \label{eq:H3}
\left\|U^{1-p}\mu^{\frac{n+2}{2}} (\eta_{1,2}N(\Xi))(\xi_1+\mu y,t) \right\|_{\sharp,a+2,1;\sigma} \le C_{\msM_1}|t_0|^{-2\vep_*}.
\ee

Combining estimates \eqref{eq:H1}, \eqref{eq:H2}, \eqref{eq:H4}, and \eqref{eq:H3} with \eqref{mh1} proves \eqref{le2.16}.
\end{proof}

\begin{proof}[Proof of Proposition \ref{prop:nonlin}]
Fixing $\mu$, we solve the inner problem by the contraction mapping principle.
To choose the radius of the fixed-point ball, we first set $\phi=0$.
Proposition \ref{prop:outer}, applied with the input bound
$\|\phi\|_{\sharp\sharp,a,1;\sigma}\le1$, yields the corresponding
outer solution $\psi[0]$.
We then apply Proposition \ref{prop:lin} with the source $\mch_{\mathrm{inner}}[0]$, evaluated using $\psi[0]$, and denote the resulting pair by $(\phi_0,c_{n+1,0})$. By \eqref{le2.16} and \eqref{phmh},
\[
\|\phi_0\|_{\sharp\sharp,a,1;\sigma}+\|c_{n+1,0}\|_{1;\sigma}\le C_*,
\]
where $C_*$ is independent of $t_0$. We choose $\msM_1\ge 2C_*$. With this $\msM_1$ fixed, choose $\msM_2$ as in Proposition \ref{prop:outer}, enlarging it to include the outer solution for $\phi=0$. By uniqueness, that solution is unchanged by this choice.

For each pair $(\phi,c_{n+1})$ in the ball \eqref{est-phi-inner} satisfying the stated orthogonality and symmetry conditions, let $\psi[\phi]$ denote the corresponding outer solution.
Using $\psi[\phi]$ in the source $\mch_{\mathrm{inner}}[\phi]$, we define $\mct(\phi,c_{n+1})$ as the solution pair furnished by Proposition \ref{prop:lin}.
By \eqref{mh1} and Proposition \ref{prop:outer}, $\mch_{\mathrm{inner}}[\phi]$ satisfies \eqref{mchin}--\eqref{mchin2}.
Consequently, Proposition \ref{prop:lin} ensures that $\mct(\phi,c_{n+1})$ satisfies the required orthogonality and symmetry conditions.
Using \eqref{psi-lip-3} and \eqref{phmh}, it is standard to show that $\mct$ maps the closed ball \eqref{est-phi-inner} into itself and is a contraction there.
Its unique fixed point is the required pair $(\phi,c_{n+1})$.

Finally, the convergences \eqref{plim}--\eqref{plimd} follow from the compactness and uniqueness argument in Appendix \ref{app5}.
\end{proof}

\section{Proof of Theorem \ref{thm:main-an}: Multiple-layer Case}\label{sec:mult}
In this section, we extend the construction developed in Sections \ref{sec:sing}--\ref{sec:inner} to the multiple-layer setting.
Since the argument follows closely the scheme established earlier, we only highlight the required modifications.

\subsection{Approximate solutions}\label{subse5.1}
Let $h\ge2$ be the number of layers and $k\ge2$ the number of bubbles in each layer. We fix $-1<\vth_1^*<\dots<\vth_h^*<1$. For $l=1,\dots,h$ and $j=1,\dots,k$, set
\[
\xi_{jl}(t)=d_l(t)\Big(\sqrt{1-\vth_l(t)^2}\,\omega_j,\vth_l(t),0_{\R^{n-3}}\Big),
\qquad d_l(t)=\sqrt{1-\mu_l(t)^2}.
\]
We assume that the parameters $\vth_l$ and $\mu_l$ belong to $C^{1,\sigma/2}_{\loc}$ in time and
\be \label{mdv}
\begin{gathered}
\msM_0^{-1}\le |t|^{\tnu}\mu_l(t)\le\msM_0,\qquad \msM_0^{-1}\le |t|^{1+\tnu}\dot\mu_l(t)\le\msM_0, \qquad \tnu=\frac1{n-2},\\
\|\mu_l\|_{\tnu;\sigma}+\|\dot\mu_l\|_{1+\tnu;\sigma}\le \msM_0, \qquad \|\vth_l-\vth_l^*\|_{\tnu+\delta_*;\sigma} +\|\dot\vth_l\|_{1+\tnu+\delta_*;\sigma}\le \msM_0
\end{gathered}
\ee
for a fixed $\delta_*>0$ and $\msM_0>1$. The number $\delta_*$ will be chosen in Proposition \ref{prop:inner1-mult}. These bounds imply $\mu_l^{-1}|\dot\xi_{1l}|\le C|t|^{-1}$. We seek a solution of \eqref{pb} of the form
\be \label{neso}
u=u_*+\Xi,\qquad u_* := \sum_{l=1}^h\sum_{j=1}^k U_{\mu_l,\xi_{jl}},\qquad \Xi=\sum_{l=1}^h\sum_{j=1}^k\eta_{jl}\tph_{jl}+\psi.
\ee
We define $\eta_{jl}$ and $\eta_{jl,2}$ by \eqref{def-cutoff-j} with $(\mu,\xi_j)$ replaced by $(\mu_l,\xi_{jl})$. Recalling \eqref{eq:bfq}, we choose $\delta_0>0$ such that
\be \label{eq:delta02}
\delta_0 = \frac{1}{4} \min_{(j,l) \ne (j',l')} |{\bf q}_{jl}^*-{\bf q}_{j'l'}^*|.
\ee
Let $\mscR_j(\bx,x') = \big(e^{2\pi i (j-1)/k} \bx,\,x'\big)$ for $(\bx,x') \in \R^2 \times \R^{n-2}$, and set
\[
\tph_{jl}(x,t)=\mu_l^{-\frac{n-2}{2}}\phi_l\!\(\frac{\mscR_j^{-1}(x)-\xi_{1l}}{\mu_l},t\).
\]
We require $\tph_{1l}$ and $\psi$ to be even in $x_2,x_4,\dots,x_n$, $\psi$ to satisfy the polygonal symmetry \eqref{pp2}, and each $\tph_{jl}$ and $\psi$ to satisfy \eqref{fke}. No symmetry in $x_3$ is imposed.

We collect layer-indexed quantities into column vectors, writing $\bm\vth := (\vth_1,\dots,\vth_h)^T$, $\bm\vth^* := (\vth_1^*,\dots,\vth_h^*)^T$, $\bm\mu := (\mu_1,\dots,\mu_h)^T$, and $\bm\phi := (\phi_1,\dots,\phi_h)^T$.
The same convention applies to the other layer-indexed scalar quantities below. For tuples of functions, we use the maximum of the component norms.

\medskip
To compute the error near the $l$-th layer, we write $x=\xi_{1l}+\mu_l y$ and set
\[
V_{j'l';l}(y,t):=\mu_l^{\frac{n-2}{2}}U_{\mu_{l'},\xi_{j'l'}}(\xi_{1l}+\mu_l y)=\frac{\msc_n(\mu_l\mu_{l'})^{\frac{n-2}{2}}}
{\big(\mu_{l'}^2+|\xi_{1l}-\xi_{j'l'}+\mu_l y|^2\big)^{\frac{n-2}{2}}}.
\]
Let
\[
\mcd_l(y,t) := \frac{\dot\mu_l}{\mu_l}Z_{n+1}(y)+\frac{\dot\xi_{1l}}{\mu_l}\cdot\nabla U(y).
\]
By \eqref{err}, we have
\be \label{error-expansion}
\begin{aligned}
&\ \mu_l^{\frac{n+2}{2}}E(\xi_{1l}+\mu_l y,t) \\
&= p\(U+\sum_{(j',l')\ne(1,l)}V_{j'l';l}\)^{p-1} \left[\mcd_l-\mu_l^{\frac{n-2}{2}}\sum_{(j',l')\ne(1,l)}	(\pp_tU_{\mu_{l'},\xi_{j'l'}})(x,t)\right] \\
&\ +\frac{n+2}{4}\left[\(U+\sum_{(j',l')\ne(1,l)}V_{j'l';l}\)^p -U^p-\sum_{(j',l')\ne(1,l)}V_{j'l';l}^{\,p}\right].
\end{aligned}
\ee
Here, $\pp_t$ acts at fixed $x$, and we then evaluate at $x=\xi_{1l}+\mu_l y$.
Since the bubble centers are separated, for $\la y\ra\le CR(t)$, with $R(t)$ defined in \eqref{eq:Rt}, we have
\[
0\le\sum_{(j',l')\ne(1,l)}V_{j'l';l} \le C|t|^{-1} \quad \text{and} \quad
\mu_l^{\frac{n-2}{2}}\sum_{(j',l')\ne(1,l)}|\pp_tU_{\mu_{l'},\xi_{j'l'}}| \le C|t|^{-2}.
\]

We define the interaction matrix $B=(B_{ll'})_{l,l'=1}^h\in\R^{h\times h}$ by
\be \label{eq:Bll'}
B_{ll'} := \sum_{\substack{1\le j'\le k,\\(j',l')\ne(1,l)}}
\frac{\msc_n}{|{\bf q}_{1l}^*-{\bf q}_{j'l'}^*|^{n-2}}.
\ee
By \eqref{eq:bfq}, $B_{ll'}=B_{l'l}$, and all entries are positive since $k\ge2$. We also define
\[
C_{ll'}^{3} := \sum_{\substack{1\le j'\le k,\\(j',l')\ne(1,l)}}
\frac{\msc_n({\bf q}_{1l}^*-{\bf q}_{j'l'}^*)_3} {|{\bf q}_{1l}^*-{\bf q}_{j'l'}^*|^n}.
\]
Then $C_{ll}^{3}=0$ and $C_{ll'}^{3}=-C_{l'l}^{3}$. For the modulation system \eqref{ode2} below, we set
\[
\msS_l(\bm\mu) := \sum_{l'=1}^hC_{ll'}^{3}(\mu_l\mu_{l'})^{\frac{n-2}{2}} \quad \text{and} \quad 
\msT_l(\bm\mu) := \sum_{l'=1}^hB_{ll'}(\mu_l\mu_{l'})^{\frac{n-2}{2}}.
\]

Let $a_1,a_2$ be as in \eqref{eq:a1a2}, and set
\[
a_3 := \int_{\R^n}U^{p-1}Z_3^2\,dy>0, \qquad a_4 := -\int_{\R^n}U^{p-1}y_3Z_3\,dy=\frac1p\int_{\R^n}U^p\,dy>0.
\]
Since the scales are uniformly comparable by \eqref{mdv}, Taylor's formula yields
\begin{multline*}
\bigg| \sum_{(j',l')\ne(1,l)}V_{j'l';l}(y,t)-\msT_l(\bm\mu) + (n-2)\mu_l\sum_{(j',l')\ne(1,l)} \frac{\msc_n(\mu_l\mu_{l'})^{\frac{n-2}{2}} ({\bf q}_{1l}^*-{\bf q}_{j'l'}^*)\cdot y} {|{\bf q}_{1l}^*-{\bf q}_{j'l'}^*|^n} \bigg| \\
\le C|t|^{-1}\Big[\mu_l^2(1+|y|^2) + |\bm\vth-\bm\vth^*|(1+\mu_l|y|)\Big].
\end{multline*}
Projection of \eqref{error-expansion} onto $Z_3$ and $Z_{n+1}$ yields, as in Lemma \ref{le2.3},
\begin{multline} \label{pre2}
\mu_l^{\frac{n+2}{2}}\int_{\R^n} \eta_{1l,2}(\xi_{1l}+\mu_l y,t)E(\xi_{1l}+\mu_l y,t)Z_3(y)\,dy\\
=pa_3\mu_l^{-1}(d_l\vth_l)'+\frac{(n+2)^2}{4}a_4\mu_l\msS_l(\bm\mu)+\mce_{3l}(t),
\end{multline}
and
\begin{multline} \label{pre0}
\mu_l^{\frac{n+2}{2}}\int_{\R^n} \eta_{1l,2}(\xi_{1l}+\mu_l y,t)E(\xi_{1l}+\mu_l y,t)Z_{n+1}(y)\,dy\\
=pa_1\mu_l^{-1}\dot\mu_l-\frac{n+2}{4}pa_2\msT_l(\bm\mu)+\mce_{(n+1)l}(t),
\end{multline}
where 
\[
|\mce_{il}(t)|+[\mce_{il}]_{C_t^{\sigma/2}}(t) \le C|t|^{-1} \Big[\mu_l+R^{-2} +|t|^{-1}R^{n-2} +|t|^{-(p-1)}R^2 +|\bm\vth-\bm\vth^*|\Big]
\]
for $i\in\{3,n+1\}$. Consequently, 
\[
\|\mce_{3l}\|_{1+2\vep_*;\sigma} + \|\mce_{(n+1)l}\|_{1+2\vep_*;\sigma} \le C,
\]
where $\vep_* > 0$ is the number fixed in Definition \ref{def:num}. The constant $C>0$ depends on $n$, $h$, $k$, $\bm\vth^*$, $\msM_0$, and the auxiliary parameters in Definition \ref{def:num}; it is independent of $t_0$.

\subsection{The inner and outer problems}
In the weight $\omega^1_{\alpha,\beta,\ell}$ introduced in \eqref{eq:omega}, we replace $\mu^{-\ell}$ by $\mu_1^{-\ell}$ and $\min_j\langle y_j\rangle$ by $\min_{j,l}\langle\mu_l^{-1}(x-\xi_{jl})\rangle$;
the exterior weight $\omega^2_{\vrh,\iota}$ remains unchanged.
We define $r_0(x,t)$ as in \eqref{eq:r0}, with the minimum taken over all $(j,l)$ and $(\mu,\xi_j)$ replaced by $(\mu_l,\xi_{jl})$.

Given a point $x \in \R^n$, we choose a bubble center $\xi_{j_0l_0}(t)$ closest to it. Since the scales $\mu_l(t)$ are uniformly comparable to one another, each minimum over the centers in the definitions of $\omega^1_{\alpha,\beta,\ell}(x,t)$
and $r_0(x,t)$ is comparable to the corresponding expression with $(j,l)=(j_0,l_0)$, up to uniform multiplicative constants.

\medskip
The proof of Proposition \ref{prop:outer} extends to the multi-layer setting by replacing sums over $j$ with sums over $(j,l)$.
The parameter bounds in \eqref{mdv} and the uniform separation of the centers ensure that the estimates remain valid with constants independent of $t_0$.
\begin{prop}[Outer problem]\label{prop:outer-mult}
Suppose that the parameters $(\bm\vth,\bm\mu)$ satisfy \eqref{mdv} with fixed $\msM_0$, that the auxiliary exponents satisfy Definition \ref{def:num}, and that $\|\bm\phi\|_{\sharp\sharp,a,1;\sigma}\le\msM_1$ for fixed $\msM_1>0$.
Assume also that each $\tph_{1l}$ satisfies the symmetry conditions in Subsection \ref{subse5.1}. There exists $\msM_2>0$ such that, for sufficiently negative $t_0$, the equation
\[
pu_*^{p-1}\pp_t\psi = \Delta\psi +\frac{p(n+2)}4\bigg(1-\sum_{l=1}^h\sum_{j=1}^k\eta_{jl}\bigg)u_*^{p-1}\psi +u_*^{p-1}\mch_{\mathrm{outer}}
\]
in $\R^n\times(-\infty,t_0)$ has a unique solution $\psi$ satisfying
\be \label{psi-norm-bound-mult}
\|\psi\|_{\widetilde{**},\alpha,\beta;\sigma} \le\msM_2|t_0|^{c'(\alpha-a)}.
\ee
Here, $u_*$ and $\Xi$ are given by \eqref{neso}, and $\mch_{\mathrm{outer}}$ is defined by \eqref{eq:Houter} with sums over all pairs $(j,l)$. The solution is even in $x_2,x_4,\dots,x_n$ and satisfies \eqref{pp2} and \eqref{fke}.

Let $(\bm\vth^{(m)},\bm\mu^{(m)})$, $m\ge1$, and $(\bm\vth,\bm\mu)$ satisfy \eqref{mdv} with the same $\msM_0$, $\delta_*$, and $\bm\vth^*$.
Suppose that $\bm\phi^{(m)}$ and $\bm\phi$ satisfy the preceding symmetry assumptions and the norm bound with the same $\msM_1$. If, for every $l=1,\dots,h$,
\begin{align*}
(\bm\vth^{(m)},\bm\mu^{(m)}) \to (\bm\vth,\bm\mu) &\quad \text{in }C^1_{\loc}((-\infty,t_0]),\\
\tph_{1l}^{(m)} \to \tph_{1l} &\quad \text{in }C^{2,1}_{\loc}\big(\R^n\times(-\infty,t_0]\big)
\end{align*}
as $m\to\infty$, then the corresponding outer solutions satisfy
\begin{align*}
\psi^{(m)} \to \psi &\quad \text{in } C^{2,1}_{\loc}\big(\R^n\times(-\infty,t_0]\big),\\
\hps^{(m)} \to \hps &\quad \text{in } C^{2,1}_{\loc}\big(\Ss^n\times(-\infty,t_0]\big)
\end{align*}
as $m\to\infty$. Here, $\psi$ is the unique outer solution associated with $(\bm\vth,\bm\mu,\bm\phi)$, and $\hps^{(m)}$ and $\hps$ are the conformal lifts of $\psi^{(m)}$ and $\psi$, respectively.

For fixed $(\bm\vth,\bm\mu)$ and two inner profiles $\bm\phi^{(1)}$ and $\bm\phi^{(2)}$ satisfying the preceding symmetry assumptions and norm bound, one has
\[
\left\|\psi[\bm\phi^{(1)}]-\psi[\bm\phi^{(2)}]\right\|_{\widetilde{**},\alpha,\beta;\sigma} \le C|t_0|^{c'(\alpha-a)}\big\|\bm\phi^{(1)}-\bm\phi^{(2)}\big\|_{\sharp\sharp,a,1;\sigma}.
\]
The constants $\msM_2$ and $C$ depend only on $n$, $k$, $h$, the auxiliary parameters fixed in Definition \ref{def:num}, $\msM_0$, $\msM_1$, and $\bm\vth^*$.
\end{prop}

The inner problem has two scalar projection coefficients in each layer, corresponding to height and scale. For $i\in\{3,n+1\}$, we write ${\bm c}_i(t) := (c_{i1}(t),\dots,c_{ih}(t))^T \in \R^h$. We define
\begin{align*}
\mcz_{3l} &:= d_l\(-\tfrac{\vth_l}{\sqrt{1-\vth_l^2}}Z_1+Z_3\),\\
\mcz_{(n+1)l} &:= Z_{n+1}-\frac{\mu_l}{d_l}\(\sqrt{1-\vth_l^2}Z_1+\vth_lZ_3\).
\end{align*}
The conformal lifts of these two functions span the Kelvin-invariant subspace of $\operatorname{span}\{z_1,z_3,z_{n+1}\}$.
Within this span, a first spherical harmonic $z\mapsto a\cdot z$ is Kelvin invariant precisely when its coefficient vector $a\in\R^{n+1}$ is orthogonal to 
\[
d_l\big(\sqrt{1-\vth_l^2}\,\mathbf e_1 +\vth_l\mathbf e_3\big)+\mu_l\mathbf e_{n+1}.
\]

Since $\mu_l^{-1}(|\dot\mu_l|+|\dot\xi_{1l}|) \le C|t|^{-1}$, the proof of Proposition \ref{prop:lin} extends to each layer with the two projection modes $\mcz_{3l}$ and $\mcz_{(n+1)l}$.
Indeed, the reflection symmetries eliminate the first-harmonic components in the directions $2,4,\dots,n$, while Kelvin invariance restricts the remaining components to the span of the conformal lifts of these two modes.
The contraction argument in the proof of Proposition \ref{prop:nonlin} then yields the following result.
\begin{prop}[Inner problem]\label{prop:nonlin-mult}
There exists a constant $\msM_1>0$, depending only on $n$, $k$, $h$, $\msM_0$, $\bm\vth^*$, and the auxiliary parameters in Definition \ref{def:num}, such that, for all sufficiently negative $t_0$, the system
\[
{\bf L}[\vth_l,\dot\vth_l,\mu_l,\dot\mu_l]\phi_l = U^{p-1} \big(\mch_{\mathrm{inner},l}+c_{3l}\mcz_{3l}+c_{(n+1)l}\mcz_{(n+1)l}\big)
\]
in $\R^n\times(-\infty,t_0)$, for $l=1,\dots,h$, has a unique solution triple $(\bm\phi,{\bm c}_3,{\bm c}_{n+1})$ satisfying the orthogonality conditions \eqref{eq:inhom-i21}--\eqref{eq:inhom-i22} with $\phi$ replaced by $\phi_l$ for each $l$, the bound
\be \label{est-phi-inner-mult}
\|\bm\phi\|_{\sharp\sharp,a,1;\sigma} +\|{\bm c}_{3}\|_{1;\sigma} +\|{\bm c}_{n+1}\|_{1;\sigma} \le \msM_1,
\ee
and the following symmetry conditions: for every $l=1,\dots,h$ and $t\le t_0$, $\phi_l(\cdot,t)$ is even in $y_2,y_4,\dots,y_n$, and $\tph_{1l}(\cdot,t)$ satisfies \eqref{fke}.
The operator ${\bf L}$ and the function $\mch_{\mathrm{inner},l}$ are defined by \eqref{inner} and \eqref{mh1}, respectively, with $(\mu,\xi_1,\phi)$ replaced by $(\mu_l,\xi_{1l},\phi_l)$, using $u_*$ and $\Xi$ from \eqref{neso} and the outer solution from Proposition \ref{prop:outer-mult}.

Let $(\bm\vth^{(m)},\bm\mu^{(m)})$, $m\ge1$, and $(\bm\vth,\bm\mu)$ satisfy \eqref{mdv} with the same $\msM_0$, $\delta_*$, and $\bm\vth^*$.
If $(\bm\vth^{(m)},\bm\mu^{(m)})\to(\bm\vth,\bm\mu)$ in $C^1_{\loc}((-\infty,t_0])$ as $m\to\infty$, then the corresponding solutions satisfy
\begin{align*}
\phi_l^{(m)} \to \phi_l &\quad \text{in }C^{2,1}_{\loc}\big(\R^n\times(-\infty,t_0]\big) \quad \text{for } l=1,\dots,h,\\
{\bm c}_i^{(m)} \to {\bm c}_i &\quad \text{in }C^0_{\loc}((-\infty,t_0]) \quad \text{for } i\in\{3,n+1\}
\end{align*}
as $m\to\infty$. Here, $(\bm\phi,{\bm c}_3,{\bm c}_{n+1})$ is the unique solution triple associated with $(\bm\vth,\bm\mu)$.
\end{prop}

Projecting the equation in Proposition \ref{prop:nonlin-mult} against $Z_3$ and $Z_{n+1}$, we obtain
\[
\int_{\R^n}\big({\bf L}\phi_l-U^{p-1}\mch_{\mathrm{inner},l}\big)
\begin{pmatrix}
Z_3\\
Z_{n+1}
\end{pmatrix}
\,dy =
\begin{pmatrix}
d_la_3&-\mu_l\vth_l a_3/d_l\\
0&a_1
\end{pmatrix}
\begin{pmatrix}
c_{3l}\\
c_{(n+1)l}
\end{pmatrix}.
\]
Since $d_la_1a_3>0$, the right-hand side vanishes if and only if $c_{3l}=c_{(n+1)l}=0$. We now choose the parameters $(\bm\vth,\bm\mu)$ so that both ${\bm c}_3$ and ${\bm c}_{n+1}$ vanish, giving
\be \label{inner-mult}
{\bf L}[\vth_l,\dot\vth_l,\mu_l,\dot\mu_l]\phi_l=U^{p-1}\mch_{\mathrm{inner},l} \quad \text{for } l=1,\dots,h.
\ee
The projection argument of Lemma \ref{lem:modulation-single}, using \eqref{pre2}--\eqref{pre0}, leads us to the following system.

\begin{lemma}[Modulation ODE system]\label{lem5.1}
For the solution triple $(\bm\phi,{\bm c}_3,{\bm c}_{n+1})$ of Proposition \ref{prop:nonlin-mult}, ${\bm c}_3={\bm c}_{n+1}=0$ if and only if
\be \label{ode2}
\begin{cases}
\displaystyle pa_3\mu_l^{-1}(d_l\vth_l)'+\frac{(n+2)^2}{4}a_4\mu_l\msS_l(\bm\mu)=\mcr_{3l}(t),\\
\displaystyle pa_1\mu_l^{-1}\dot\mu_l-\frac{n+2}{4}pa_2\msT_l(\bm\mu)=\mcr_{(n+1)l}(t),
\end{cases}
\text{for } l=1,\dots,h.
\ee
For $i\in\{3,n+1\}$, we write $\bm{\mcr}_i[\bm\vth,\bm\mu](t) := (\mcr_{i1}(t),\dots,\mcr_{ih}(t))^T$. Then,
\[
\|\bm{\mcr}_{3}\|_{1+2\vep_*;\sigma}+\|\bm{\mcr}_{n+1}\|_{1+2\vep_*;\sigma} \le C,
\]
where $C>0$ depends only on $n$, $k$, $h$, the auxiliary parameters fixed in Definition \ref{def:num}, $\msM_0$, and $\bm\vth^*$.

Let $(\bm\vth^{(m)},\bm\mu^{(m)})$, $m\ge1$, and $(\bm\vth,\bm\mu)$ satisfy \eqref{mdv} with the same $\msM_0$, $\delta_*$, and $\bm\vth^*$.
If $(\bm\vth^{(m)},\bm\mu^{(m)})\to(\bm\vth,\bm\mu)$ in $C^1_{\loc}((-\infty,t_0])$ as $m\to\infty$, then
\[
\bm{\mcr}_i[\bm\vth^{(m)},\bm\mu^{(m)}] \to \bm{\mcr}_i[\bm\vth,\bm\mu] \quad \text{in } C^0_{\loc}((-\infty,t_0]) \quad \text{for } i\in\{3,n+1\}
\]
as $m\to\infty$.
\end{lemma}

\begin{prop}[Solvability of the parameters $(\bm\vth,\bm\mu)$]\label{prop:inner1-mult}
For fixed $\msM_0$ sufficiently large in terms of $n$, $k$, $h$, and $\bm\vth^*$, there exists $\delta_*\in(0,\vep_*)$ such that, for all sufficiently negative $t_0$, system \eqref{ode2} has a solution $(\bm\vth,\bm\mu)$ satisfying \eqref{mdv} of the form
\[
\vth_l(t)=\vth_l^*+\zeta_l(t),\qquad \mu_l(t)=\kappa_l|t|^{-\tnu}+\rho_l(t) \qquad \text{for } l=1,\dots,h.
\]
Here, $\kappa_l := b_l^{2/(n-2)}>0$, where $\bm b$ is the unique positive solution of \eqref{bl} below. Moreover, it holds that
\be\label{stes}
\|\bm\zeta\|_{\tnu+\delta_*;\sigma} +\|\dot{\bm\zeta}\|_{1+\tnu+\delta_*;\sigma} +\|\bm\rho\|_{\tnu+\delta_*;\sigma} +\|\dot{\bm\rho}\|_{1+\tnu+\delta_*;\sigma}\le C,
\ee
where $C>0$ depends only on $n$, $k$, $h$, the auxiliary parameters fixed in Definition \ref{def:num}, $\msM_0$, $\bm\vth^*$, and $\delta_*$.
\end{prop}
\begin{proof}
We first determine the leading scale constants. We claim that there is a unique vector $\bm b=(b_1,\dots,b_h)^T\in(0,\infty)^h$ satisfying
\be\label{bl}
b_l(B\bm b)_l=\frac{4a_1}{(n^2-4)a_2} \quad \text{for } l=1,\dots,h,
\ee
where $B_{ll'} > 0$ is given in \eqref{eq:Bll'}. To prove this, we write $b_l=e^{z_l}$ and $\bm z=(z_1,\dots,z_h)^T \in \R^h$, and define
\[
\Psi(\bm z) := \frac12 \sum_{l=1}^h\sum_{l'=1}^h B_{ll'}e^{z_l+z_{l'}} -\frac{4a_1}{(n^2-4)a_2}\sum_{l=1}^hz_l.
\]
Its critical points correspond to positive solutions of \eqref{bl}. For every $\bm v=(v_1,\dots,v_h)^T\in\R^h\setminus\{0\}$,
\[
D^2\Psi(\bm z)[\bm v,\bm v] = \frac12 \sum_{l=1}^h\sum_{l'=1}^h B_{ll'}e^{z_l+z_{l'}}(v_l+v_{l'})^2>0.
\]
Besides,
\[
\Psi(\bm z) \ge \sum_{l=1}^h \left[\frac{1}{2}B_{ll}e^{2z_l}-\frac{4a_1}{(n^2-4)a_2}z_l\right] \to +\infty \quad \text{as } |\bm z|\to\infty.
\]
Thus $\Psi$ is strictly convex and coercive, and its unique minimizer determines $\bm b$. We set $\kappa_l=b_l^{2/(n-2)}$ and define the matrices $P=(P_{ll'})_{l,l'=1}^h$ and $H$ in $\R^{h\times h}$ by
\[
P_{ll'} := \frac{(n^2-4)a_2}{4a_1}B_{ll'}b_lb_{l'} \quad \text{and} \quad H := \frac12(\Id_h+P).
\]
The matrix $P$ is symmetric with strictly positive entries and satisfies $P\bm1=\bm1$, where $\bm1=(1,\dots,1)^T$.
At the minimizer, $D^2\Psi =\frac{4a_1}{(n^2-4)a_2}(\Id_h+P) =\frac{8a_1}{(n^2-4)a_2}H$, so $H$ is positive definite.
	
\medskip
Let us write
\be\label{ulxl}
d_l\vth_l=\vth_l^*+\tze_l \quad \text{and} \quad \mu_l=\kappa_l|t|^{-\tnu}(1+\trh_l) \quad \text{for } l=1,\dots,h,
\ee
where $\tze_l := d_l\zeta_l+(d_l-1)\vth_l^*$ and $\trh_l := \rho_l/(\kappa_l|t|^{-\tnu})$.
Substituting into \eqref{ode2}, we obtain the nonlinear system
\be \label{modep}
\begin{cases}
\displaystyle \dot{\tze}_l = -\frac{(n+2)^2a_4}{4pa_3}\mu_l^2\msS_l(\bm\mu) +\frac{\mu_l}{pa_3}\mcr_{3l} \quad \text{for } l=1,\dots, h,\\
\displaystyle \dot{\widetilde{\bm\rho}}-|t|^{-1}H\widetilde{\bm\rho} = |t|^{-1}\bm Q(\widetilde{\bm\rho}) +\(\frac{1+\trh_l}{pa_1}\mcr_{(n+1)l}\)_{l=1}^h,
\end{cases}
\ee
where the smooth vector function $\bm Q=(Q_1,\dots,Q_h)^T$ is defined by
\[
Q_l(\widetilde{\bm\rho}) := \tnu(1+\trh_l) \left[(1+\trh_l)^{\frac{n-2}{2}} \sum_{l'=1}^hP_{ll'}(1+\trh_{l'})^{\frac{n-2}{2}}-1\right] -(H\widetilde{\bm\rho})_l.
\]
In particular, $\bm Q(0)=D\bm Q(0)=0$.

First, we develop the linear theory needed to solve equation \eqref{modep}. Let $\lambda_{\min}(H)$ denote the least eigenvalue of $H$, and choose $0<\delta_*<\min\{\vep_*,\lambda_{\min}(H)\}$. For $\R^h$-valued functions $\bm F=\bm F(t)$ and $\bm G=\bm G(t)$ with $\|\bm F\|_{1+\tnu+\delta_*;\sigma}+\|\bm G\|_{1+\delta_*;\sigma}<\infty$, we consider the linear equations
\[
\dot{\widetilde{\bm\zeta}}=\bm F, \qquad \dot{\widetilde{\bm\rho}}-|t|^{-1}H\widetilde{\bm\rho}=\bm G,
\]
subject to $\widetilde{\bm\zeta}(t)\to0$ as $t\to-\infty$ and $\widetilde{\bm\rho}(t_0)=0$, respectively. Their solution operators are
\[
\mcs_3\bm F(t) := \int_{-\infty}^t\bm F(\tit)\,d\tit, \qquad
\mcs_H\bm G(t) := -\int_t^{t_0} \(\frac{|\tit|}{|t|}\)^H\bm G(\tit)\,d\tit.
\]
Here, because $H$ is real symmetric and positive definite, there exists $O\in O(h)$ such that $H=O\operatorname{diag}(\lambda_1,\dots,\lambda_h)O^T$, where $\lambda_l>\delta_*$ for every $l=1,\dots,h$.
The matrix power of a number $r>0$ is then defined by $r^H := O\operatorname{diag}(r^{\lambda_1},\dots,r^{\lambda_h})O^T$. We assert the estimates
\be\label{eq:mod-linear-estimates}
\begin{aligned}
\|\mcs_3\bm F\|_{\tnu+\delta_*;\sigma} +\|\pp_t\mcs_3\bm F\|_{1+\tnu+\delta_*;\sigma} &\le C\|\bm F\|_{1+\tnu+\delta_*;\sigma},\\
\|\mcs_H\bm G\|_{\delta_*;\sigma} +\|\pp_t\mcs_H\bm G\|_{1+\delta_*;\sigma} &\le C\|\bm G\|_{1+\delta_*;\sigma},
\end{aligned}
\ee
where $C>0$ depends on $h$, $H$, $\tnu$, $\delta_*$, and $\sigma$, but is independent of $t_0\le-2$.

The height estimate follows by direct integration:
\[
|\mcs_3\bm F(t)| \le C\|\bm F\|_{1+\tnu+\delta_*;\sigma} \int_{|t|}^{\infty}r^{-1-\tnu-\delta_*}\,dr = \frac{C}{\tnu+\delta_*}\|\bm F\|_{1+\tnu+\delta_*;\sigma} |t|^{-\tnu-\delta_*}.
\]
Also, $\pp_t\mcs_3\bm F=\bm F$. For the scale estimate, the $j$-th component in the eigenbasis of $H$ is
\[
\big(O^T\mcs_H\bm G(t)\big)_j = -|t|^{-\lambda_j}\int_t^{t_0}|\tit|^{\lambda_j} \big(O^T\bm G(\tit)\big)_j\,d\tit.
\]
Since $\lambda_j>\delta_*$, we have
\[
|t|^{-\lambda_j} \int_{|t_0|}^{|t|}r^{\lambda_j-1-\delta_*}\,dr = \frac{|t|^{-\delta_*}}{\lambda_j-\delta_*} \left[1-\(\frac{|t_0|}{|t|}\)^{\lambda_j-\delta_*}\right] \le \frac{|t|^{-\delta_*}}{\lambda_j-\delta_*}.
\]
It follows, after returning to the original coordinates, that
\[
|\mcs_H\bm G(t)| \le C\|\bm G\|_{1+\delta_*;\sigma}|t|^{-\delta_*}.
\]
Differentiating its integral formula yields $\pp_t\mcs_H\bm G=|t|^{-1}H\mcs_H\bm G+\bm G$, and hence
\[
|\pp_t\mcs_H\bm G(t)| \le C\|\bm G\|_{1+\delta_*;\sigma}|t|^{-1-\delta_*}.
\]

To bound the local time H\"older seminorms, we can use the mean value theorem and the estimate $[|\cdot|^{-1}]_{C_t^{\sigma/2}}(t)\le C|t|^{-2}$. We omit the details.

The desired estimate \eqref{eq:mod-linear-estimates} is now established.
	
\medskip
To solve the nonlinear system \eqref{modep}, we work in the Fr\'echet space $C^1_{\loc}((-\infty,t_0];\R^{2h})$ and consider the set
\[
\mcb_1 := \left\{(\widetilde{\bm\zeta},\widetilde{\bm\rho}):\ \widetilde{\bm\rho}(t_0)=0,\
\|\widetilde{\bm\zeta}\|_{\tnu+\delta_*;\sigma}+\|\dot{\widetilde{\bm\zeta}}\|_{1+\tnu+\delta_*;\sigma} +\|\widetilde{\bm\rho}\|_{\delta_*;\sigma}+\|\dot{\widetilde{\bm\rho}}\|_{1+\delta_*;\sigma} \le 1\right\},
\]
endowed with the induced topology. The set $\mcb_1$ is nonempty and convex. The lower semicontinuity of the weighted norms and continuity of evaluation at $t_0$ imply that $\mcb_1$ is closed in this topology.

We choose $\msM_0$ sufficiently large. For all sufficiently large $|t_0|$ and every $(\widetilde{\bm\zeta},\widetilde{\bm\rho})\in\mcb_1$, the parameters $(\bm\vth,\bm\mu)$ defined by \eqref{ulxl} satisfy \eqref{mdv}.
We fix $\msM_1$ and $\msM_2$ as in Propositions \ref{prop:nonlin-mult} and \ref{prop:outer-mult}, respectively, independently of $t_0$.
For each $(\widetilde{\bm\zeta},\widetilde{\bm\rho})\in\mcb_1$, we take the corresponding inner and outer solutions furnished by these propositions and set
\[
\mscA(\widetilde{\bm\zeta},\widetilde{\bm\rho}) := (\mcs_3\bm F,\mcs_H\bm G),
\]
where $\bm F=(F_1,\dots,F_h)^T$ and $\bm G$ are the right-hand sides of \eqref{modep}:
\begin{align*}
F_l &:= -\frac{(n+2)^2a_4}{4pa_3}\mu_l^2\msS_l(\bm\mu)+\frac{\mu_l}{pa_3}\mcr_{3l} \quad \text{for } l=1,\dots,h,\\
\bm G &:= |t|^{-1}\bm Q(\widetilde{\bm\rho})+\(\frac{1+\trh_l}{pa_1}\mcr_{(n+1)l}\)_{l=1}^h.
\end{align*}

We first verify continuity of $\mscA$ in $C^1_{\loc}$. Suppose that
\[
(\widetilde{\bm\zeta}_j,\widetilde{\bm\rho}_j) \to (\widetilde{\bm\zeta},\widetilde{\bm\rho}) \quad \text{in } C^1_{\loc} \quad \text{as } j\to\infty,
\]
with all pairs in $\mcb_1$, and let $\bm F_j,\bm G_j$ denote the corresponding right-hand sides. By Lemma \ref{lem5.1}, $\bm F_j\to\bm F$ and $\bm G_j\to\bm G$ in $C^0_{\loc}$ as $j\to\infty$. Moreover, $\mu_l^2\msS_l(\bm\mu)=O(|t|^{-1-2\tnu})$ and the bound for $\bm{\mcr}_3$ imply
\[
|\bm F_j(t)|\le C|t|^{-1-\tnu-\delta_*} \quad \text{for }t\le t_0,
\]
with $C$ independent of $j$. This common integrable bound and the finite integration intervals defining $\mcs_H$ imply
\[
\mcs_3\bm F_j\to\mcs_3\bm F, \qquad \mcs_H\bm G_j\to\mcs_H\bm G \quad \text{in } C^0_{\loc} \quad \text{as } j\to\infty.
\]
The identities $\pp_t\mcs_3\bm F_j=\bm F_j$ and $\pp_t\mcs_H\bm G_j =|t|^{-1}H\mcs_H\bm G_j+\bm G_j$ then imply convergence of the time derivatives in $C^0_{\loc}$. Hence $\mscA$ is continuous in $C^1_{\loc}$.

By Lemma \ref{lem5.1} and \eqref{eq:mod-linear-estimates}, the sum of the four weighted norms of $\mscA(\widetilde{\bm\zeta},\widetilde{\bm\rho})$ is at most
\[
C\left[|t_0|^{-(\tnu-\delta_*)}+|t_0|^{-(2\vep_*-\delta_*)}+|t_0|^{-\delta_*}\right],
\]
with $C>0$ uniform on $\mcb_1$ and independent of $t_0$. Since $\mcs_H\bm G(t_0)=0$, it follows that $\mscA(\mcb_1)\subset\mcb_1$ for sufficiently large $|t_0|$.
These weighted H\"older bounds also imply uniform local bounds and equicontinuity for the functions in $\mscA(\mcb_1)$ and their time derivatives.
Thus $\mscA(\mcb_1)$ is relatively compact in $C^1_{\loc}$ by Arzel\`a--Ascoli and a diagonal argument. Since $\mcb_1$ is closed, the closure of $\mscA(\mcb_1)$ is a compact subset of $\mcb_1$.
The Schauder--Tychonoff fixed-point theorem \cite[Theorem 2.6]{MN75} therefore yields a fixed point $(\widetilde{\bm\zeta},\widetilde{\bm\rho})\in\mcb_1$.
The corresponding parameters $(\bm\vth,\bm\mu)$ defined by \eqref{ulxl} solve \eqref{ode2}.

Finally, we write $\zeta_l=[(1-\mu_l^2)^{-1/2}-1]\vth_l^*+(1-\mu_l^2)^{-1/2}\tze_l$. Since $\delta_*<\vep_*<\tnu$ and $(1-\mu_l^2)^{-1/2}-1=O(\mu_l^2)$, one gets \eqref{stes}. This completes the proof.
\end{proof}

\subsection{Completion of the proof of Theorem \ref{thm:main-an} for \texorpdfstring{$h\ge2$}{h greater than or equal to 2}}
By Proposition \ref{prop:inner1-mult} and Lemma \ref{lem5.1}, ${\bm c}_3={\bm c}_{n+1}=0$. Combining Propositions \ref{prop:nonlin-mult} and \ref{prop:outer-mult}, we obtain a solution of \eqref{pb} of the form \eqref{neso}.
The parameters $(\bm\vth,\bm\mu)$ satisfy \eqref{eq:parasym1}--\eqref{eq:parasym2} with $\ga=(n-2)\delta_*$. The remaining properties follow as in Subsection \ref{subsec:comp}.

\subsection{Some representative configurations}
\leavevmode\par

\medskip \noindent \textbf{(1) Single-layer configuration.}
Let $u$ be the single-layer equatorial solution of Theorem \ref{thm:main-an}\textup{(1)} having the form in \eqref{so1.2}.

We fix $\vth\in(-1,1)\setminus\{0\}$ and set $r_\vth:=\sqrt{1-\vth^2}$. The preceding argument in this section also applies when $h=1$, provided that we retain the first equation in \eqref{ode2}, corresponding to projection onto $Z_3$.
Taking $\vth_1^*=\vth$, we obtain a Kelvin-invariant solution of \eqref{pb} satisfying
\[
\begin{gathered}
\tu \sim \sum_{j=1}^k U_{\tmu(t),\txi_j(t)},\qquad \tmu(t)\sim r_\vth\mu(t), \\
\txi_j(t) = \sqrt{1-\tmu(t)^2}\Big(\sqrt{1-\vth_1(t)^2}\,\omega_j,\vth_1(t),0_{\R^{n-3}}\Big), \quad \vth_1(t)\to\vth \text{ as } t\to -\infty.
\end{gathered}
\]
Here, $C_{11}^{3}=\msS_1=0$, $B_{11}=A_k/r_\vth^{n-2}$, and $\msT_1=B_{11}\tmu^{n-2}>0$. System \eqref{ode2} still has a solution, and $\vth_1(t)$ need not be constant.

For comparison, applying translation and scaling to $u$ yields
\[
\bu(x,t) = r_\vth^{-(n-2)/2}u\!\(\frac{x-\vth{\bf e}_3}{r_\vth},t\) = \sum_{j=1}^k U_{\bmu(t),\bar\xi_j(t)}(x) +r_\vth^{-(n-2)/2} \Xi\!\(\frac{x-\vth{\bf e}_3}{r_\vth},t\),
\]
where $\bmu=r_\vth\mu$ and $\bar\xi_j=r_\vth\xi_j+\vth{\bf e}_3$. Therefore, $\bmu^2+|\bar\xi_j|^2=1$ and $\bar\xi_j(t)\to r_\vth{\bf q}_j+\vth{\bf e}_3$.
The scales and centers of $\tu$ and $\bu$ agree to leading order; however, it remains unclear whether the two solutions are geometrically equivalent.

\medskip \noindent \textbf{(2) Multi-layer symmetric configuration.}
We choose $\vth_{h+1-l}^*=-\vth_l^*$ for $l=1,\dots,h$ and impose, for $t\le t_0$,
\[
\vth_{h+1-l}(t)=-\vth_l(t),\quad \mu_{h+1-l}(t)=\mu_l(t),\quad d_{h+1-l}(t)=d_l(t)=\sqrt{1-\mu_l(t)^2}.
\]
The argument in this section produces a solution that is even in $x_3$, as described below.

For $h=2m$, we choose $-1<\vth_1^*<\dots<\vth_m^*<0$. Then
\[
u\sim\sum_{l=1}^m\sum_{j=1}^k \big(U_{\mu_l,\xi_{jl}}+U_{\mu_l,\xi_{j(2m+1-l)}}\big),
\]
where $\vth_l(t)\to\vth_l^*$ as $t\to -\infty$ and 
\begin{align*}
\xi_{jl}(t) &= d_l(t)\Big(\sqrt{1-\vth_l(t)^2}\,\omega_j,\vth_l(t),0_{\R^{n-3}}\Big), \\
\xi_{j(h+1-l)}(t) &= d_l(t)\Big(\sqrt{1-\vth_l(t)^2}\,\omega_j,-\vth_l(t),0_{\R^{n-3}}\Big).
\end{align*}
For $h=2m+1\ge3$, we add one equatorial layer with
\[
\vth_{m+1}(t)=\vth_{m+1}^*=0, \qquad \xi_{j(m+1)}(t)=\sqrt{1-\mu_{m+1}(t)^2} \big(\omega_j,0,0_{\R^{n-3}}\big).
\]
In this case,
\[
u \sim \sum_{j=1}^k \left[U_{\mu_{m+1},\xi_{j(m+1)}} +\sum_{l=1}^m\big(U_{\mu_l,\xi_{jl}}+U_{\mu_l,\xi_{j(2m+2-l)}}\big)\right].
\]

As in the elliptic Yamabe construction in \cite[Section 6]{MM21}, we impose evenness in $x_3$ by pairing layers with opposite heights and equal scales.
The relation $\vth_{h+1-l}(t)=-\vth_l(t)$ does not imply that the heights are constant, so one height equation in \eqref{ode2} must still be solved for each non-equatorial pair.
This symmetry is optional in our construction, since Proposition \ref{prop:outer-mult} solves the outer problem without it. In contrast, the outer solvability argument presented in \cite[Proposition 4.2]{MM21} relies on this symmetry to handle the positive center bubble.

\begin{remark}
For $n\ge4$, the elliptic construction of Medina, Musso, and Wei \cite{MMW19} suggests considering linked or mutually orthogonal polygonal rings. For example, one may seek
\[
u\sim\sum_{j=1}^{k_1}U_{\mu_1(t),\xi_{j1}(t)} +\sum_{j=1}^{k_2}U_{\mu_2(t),\xi_{j2}(t)} \quad \text{for } k_1,k_2\ge2.
\]
Here, the centers are given by
\begin{align*}
\xi_{j1}(t) &= d_1(t)\Big(e^{\frac{2\pi i(j-1)}{k_1}},0_{\R^2},0_{\R^{n-4}}\Big) \quad \text{for } j=1,\dots,k_1,\\
\xi_{j2}(t) &= d_2(t)\Big(0_{\R^2},e^{\frac{2\pi i(j-1)}{k_2}},0_{\R^{n-4}}\Big) \quad \text{for } j=1,\dots,k_2.
\end{align*}
Also, for $l=1,2$, we put $\mu_l(t)\sim\kappa_l|t|^{-\frac1{n-2}}$ as $t\to-\infty$, where $\kappa_l>0$, and $d_l(t)=\sqrt{1-\mu_l(t)^2}$. The formal scale equations are obtained by imposing
\[
\mu_l^{\frac{n+2}{2}}\int_{B(0,R)} E(\mu_l y+\xi_{1l},t)Z_{n+1}(y)\,dy +\text{(higher-order terms)}=0
\]
for $l=1,2$.
\end{remark}

\section{Proof of Theorem \ref{thm:main-geo}}\label{sec:geo}
We fix integers $h\ge1$ and $k\ge2$ and a vector $\bvth^*\in\mca_h$, where $\mca_h$ is defined in \eqref{eq:mcah}.
Let $u$ be the corresponding solution of \eqref{pb} constructed in the preceding sections.
We use the notation of Theorem \ref{thm:main-an}, omitting the layer index when $h=1$, and restrict to $t\le t_0$, where $t_0<0$ will be chosen sufficiently negative below.
By Theorem \ref{thm:main-an}\textup{(iv)}, the conformal lift $\hu$ of $u$ extends smoothly across the north pole $\msN_\infty\in\Ss^n$.
We define $\hg(t):=\hu(\cdot,t)^{4/(n-2)}g_{\Ss^n}$ and denote its pullback to the stereographic chart by $g_E(t):=\pi^*\hg(t)=u(\cdot,t)^{4/(n-2)}g_{\R^n}$.

Reversing \eqref{ubarvrel} and translating the geometric time so that $t=t_0$ corresponds to $s=0$, we set
\be \label{eq:stcor}
s := \frac{e^{-t_0}-e^{-t}}{p(n-1)} \quad \text{and} \quad g(s)=g_{h,k,\bvth^*}(s) := e^{-t}\hg(t) \quad \text{for } t<t_0.
\ee
This change of variables maps $(-\infty,t_0)$ onto $(-\infty,0)$ and yields a solution $g$ of \eqref{YF-ori} on $\Ss^n\times(-\infty,0)$. We prove Theorem \ref{thm:main-geo} for this flow.

We write $\Sigma=\Sigma_{h,k}(\bvth^*)\subset\Ss^n$ as in \eqref{eq:Sigmahk}. Let $\mci := \{1,\dots,k\}\times\{1,\dots,h\}$,
\be \label{eq:Ftinf}
F_t(x) := \sqrt{|t|}\, u(x,t) \quad \text{and} \quad F_\infty(x) := \msc_n\sum_{l=1}^h\sum_{j=1}^k \kappa_l^{\frac{n-2}{2}}|x-{\bf q}_{jl}^*|^{2-n}
\ee
for $x\in\R^n\setminus P_{h,k}(\bvth^*)$, where $\msc_n > 0$ is the constant in \eqref{eq:bubble}, $\kappa_l > 0$ the constant in \eqref{eq:parasym1}, and ${\bf q}_{jl}^* \in \R^n$ the point in \eqref{eq:bfq}.

The following lemma is an immediate consequence of \eqref{eq:bubble2}, Theorem \ref{thm:main-an}, together with the estimates in Propositions \ref{prop:outer}, \ref{prop:nonlin}, \ref{prop:outer-mult}, and \ref{prop:nonlin-mult}.

\begin{lemma}[Pointwise profiles and comparison] \label{lem:core-profile}
For every compact subset $K\Subset\R^n\setminus P_{h,k}(\bvth^*)$, 
\be \label{eq:outer-profile}
F_t\to F_\infty \quad \text{in }C^2(K) \quad \text{as } t\to-\infty.
\ee
Besides,
\be \label{eq:geo-core-ratio}
\max_{(j,l)\in\mci} \sup_{|y|\le R(t)} \left|\frac{\mu_l(t)^{\frac{n-2}{2}} u(\xi_{jl}(t)+\mu_l(t)y,t)}{U(y)}-1\right| \to 0 \quad \text{as } t\to-\infty,
\ee
where $R(t)=|t|^{c_0}$ with $0<c_0<1/(n-2)$ is defined in \eqref{eq:Rt}. Furthermore, there exists $C>0$ such that
\be \label{eq:geo-global-upper}
u(x,t)\le C\sum_{l=1}^h\sum_{j=1}^kU_{\mu_l(t),\xi_{jl}(t)}(x) \quad \text{for } x\in\R^n \text{ and } t \in (-\infty,t_0].
\ee
\end{lemma}

For a path-connected subset $\Omega$ of a Riemannian manifold $(M,g)$, we define the \emph{intrinsic distance} between two points $z,z'\in\Omega$ and the \emph{intrinsic path diameter} of $\Omega$ by
\be \label{eq:diam}
\begin{aligned}
d^{\mint}_{\Omega,g}(z,z') &:= \inf_{\substack{\Gamma(0)=z,\ \Gamma(1)=z'\\ \Gamma\subset\Omega}}\operatorname{Length}_g(\Gamma),\\
\diam^{\mint}_g(\Omega) &:= \sup_{z,z'\in\Omega}d^{\mint}_{\Omega,g}(z,z').
\end{aligned}
\ee
Here the infimum is taken over rectifiable curves $\Gamma:[0,1]\to\Omega$, with value $+\infty$ if no such curve joins $z$ to $z'$.

\begin{lemma}[Collapse of the outer region]\label{lem:outer-collapse}
Let the open spherical cores $\mcc_{jl}(t)$ and the exterior region $\mco_t$ be defined by
\be \label{eq:mccjl}
\mcc_{jl}(t) := \pi\big(B(\xi_{jl}(t),R(t)\mu_l(t))\big), \qquad
\mco_t := \Ss^n\setminus \bigcup_{l=1}^h\bigcup_{j=1}^k\mcc_{jl}(t),
\ee
respectively. For sufficiently negative $t_0$, the cores $\mcc_{jl}(t)$ have pairwise disjoint closures, and $\mco_t$ is closed and contains all their boundaries. Moreover, $\mco_t$ is path-connected and
\be \label{eq:geo-outer-path-diameter}
\diam^{\mint}_{\hg(t)}(\mco_t) \le C\(R(t)^{-1}+\max_{l=1,\dots,h}\mu_l(t)\) \to 0 \quad \text{as } t\to-\infty.
\ee
Its volume also satisfies
\be \label{eq:geo-outer-volume}
\Vol_{\hg(t)}(\mco_t)\to0 \quad \text{as } t\to-\infty.
\ee
\end{lemma}
\begin{proof}
Let $\delta_0>0$ be the constant defined in \eqref{eq:delta0} when $h=1$ and \eqref{eq:delta02} when $h\ge2$.
For sufficiently negative $t_0$, the domains $\mcd_{jl}(t) := \pi(B(\xi_{jl}(t),\delta_0))$ have pairwise disjoint closures for every $t\le t_0$.
Since $R(t)\mu_l(t)\to0$, we may also assume $\overline{\mcc_{jl}(t)}\subset\mcd_{jl}(t)$ for every $(j,h)\in\mci$ and $t\le t_0$.
Each $\pp\mcd_{jl}(t)$ is a round $(n-1)$-sphere of radius at most $1$ with the metric induced by $g_{\Ss^n}$, and hence $\diam^{\mint}_{g_{\Ss^n}}(\pp\mcd_{jl}(t))\le\pi$.
Replacing each portion of a round minimizing geodesic inside these domains by a curve on the corresponding boundary shows that their complement $\Ss^n\setminus\bigcup_{j,l}\mcd_{jl}(t)$ is path-connected and has uniformly bounded intrinsic diameter in $g_{\Ss^n}$.
These connecting curves have $\hg(t)$-length at most $C\max_{l=1,\dots,h}\mu_l(t)$. In fact, \eqref{eq:geo-global-upper}, \eqref{eq:bubble2}, and \eqref{stereo-proj-1} imply
\[
|v|_{\hg(t)} =\hu(z,t)^{\frac2{n-2}}|v|_{g_{\Ss^n}} \le C\max_{l=1,\dots,h}\mu_l(t)\,|v|_{g_{\Ss^n}}
\]
for every $z \in \Ss^n\setminus\bigcup_{j,l}\mcd_{jl}(t)$ and $v\in T_z\Ss^n$, including $z=\msN_\infty$ by continuity.

Given $x\in B(\xi_{jl}(t),\delta_0)\setminus B(\xi_{jl}(t),R(t)\mu_l(t))$, we define
\[
\Gamma(r) := \xi_{jl}(t)+r\frac{x-\xi_{jl}(t)}{|x-\xi_{jl}(t)|} \quad \text{for } r \in [|x-\xi_{jl}(t)|,\delta_0].
\]
Its image under $\pi$ lies in $\mco_t$ and joins $\pi(x)$ to the set $\Ss^n\setminus\bigcup_{j,l}\mcd_{jl}(t)$.
Since $|\Gamma'(r)|=1$, the identity $g_E(t)=u(\cdot,t)^{4/(n-2)}g_{\R^n}$ and \eqref{eq:geo-global-upper} and \eqref{eq:bubble2} give
\begin{align*}
\operatorname{Length}_{\hg(t)}(\pi\circ\Gamma) &= \int_{|x-\xi_{jl}(t)|}^{\delta_0}u(\Gamma(r),t)^{\frac2{n-2}}\,dr \\
&\le C\int_{R(t)\mu_l(t)}^{\delta_0} \frac{\mu_l(t)}{\mu_l(t)^2+r^2}\,dr +C\max_{1\le l'\le h}\mu_{l'}(t) \\
&\le C\(R(t)^{-1}+\max_{1\le l'\le h}\mu_{l'}(t)\).
\end{align*}
Joining any two points of $\mco_t$ by at most two such radial paths and a connecting curve in $\Ss^n\setminus\bigcup_{j,l}\mcd_{jl}(t)$ proves path-connectedness and \eqref{eq:geo-outer-path-diameter}.

Finally, by \eqref{eq:geo-global-upper},
\[
\Vol_{\hg(t)}(\mco_t) = \int_{\pi^{-1}(\mco_t\setminus\{\msN_{\infty}\})}u(x,t)^{2^*} dx
\le C\sum_{l=1}^h\sum_{j=1}^k \int_{\{|y|\ge R(t)\}}U(y)^{2^*} dy \le CR(t)^{-n} \to 0,
\]
which proves \eqref{eq:geo-outer-volume}.
\end{proof}

\subsection{Backward concentration: proof of \textup{(i)}--\textup{(ii)}}
Let
\be \label{eq:mfm}
\mfm_n := \int_{\R^n}U(y)^{2^*}\,dy \quad \text{and} \quad dv_{\hg(t)}=\hu(z,t)^{2^*}\,dv_{\Ss^n}(z).
\ee
The conformal lift satisfies 
\be \label{eq:geo-critical-invariance}
\pi^*\big(\hu(\cdot,t)^{2^*}\,dv_{\Ss^n}\big)=u(x,t)^{2^*}\,dx.
\ee

\begin{prop}[Backward concentration and convergence of the Yamabe quotient]\label{prop:concentration}
We have
\be \label{eq:geo-measure-limit}
dv_{\hg(t)} \rightharpoonup \mfm_n\sum_{l=1}^h\sum_{j=1}^k \delta_{\pi({\bf q}_{jl}^*)} \quad \text{as } t \to -\infty.
\ee
Consequently,
\be \label{eq:geo-volume-limit}
\Vol_{\hg(t)}(\Ss^n) \to hk\mfm_n \quad \text{as } t \to -\infty.
\ee
The normalized volume measures \eqref{eq:nvm} of $g(s)=g_{h,k,\bvth^*}(s)$ satisfy Theorem \ref{thm:main-geo}\textup{(i)}, and
\be \label{eq:Yamabequoc}
\mcy(g(s)) \to (hk)^{2/n}\mcy(g_{\Ss^n}) \quad \text{as } s\to-\infty,
\ee
which is Theorem \ref{thm:main-geo}\textup{(ii)}.
\end{prop}
\begin{proof}
Choose and fix any $\chi\in C^0(\Ss^n)$. Since $R(t)\mu_l(t)\to0$ and $\xi_{jl}(t)\to{\bf q}_{jl}^*$, the uniform continuity of $\chi$, \eqref{eq:geo-critical-invariance}, and Lemma \ref{lem:core-profile} imply, for each $j,l$,
\begin{align*}
\int_{\mcc_{jl}(t)}\chi\,dv_{\hg(t)} &= \int_{|y|<R(t)} \chi\big(\pi(\xi_{jl}+\mu_ly)\big) \left[\mu_l^{\frac{n-2}{2}}u(\xi_{jl}+\mu_ly,t)\right]^{2^*}\,dy\\
&= \chi\big(\pi({\bf q}_{jl}^*)\big)\int_{|y|<R(t)}U(y)^{2^*}\,dy+o(1).
\end{align*}
By \eqref{eq:geo-outer-volume},
\[
\left|\int_{\mco_t}\chi\,dv_{\hg(t)}\right| \le \|\chi\|_{L^\infty}\Vol_{\hg(t)}(\mco_t) \to 0.
\]
Summing over the cores $\mcc_{jl}$ and using $R(t)\to\infty$, we obtain \eqref{eq:geo-measure-limit}. Taking $\chi\equiv1$ gives \eqref{eq:geo-volume-limit}.
As mentioned in \eqref{eq:stcor}, $g(s)=e^{-t}\hg(t)$, so their normalized volume measures $dv_{g(s)}/\Vol_{g(s)}(\Ss^n)$ and $dv_{\hg(t)}/\Vol_{\hg(t)}(\Ss^n)$ coincide.
Since $s\to-\infty$ if and only if $t\to-\infty$, \eqref{eq:geo-measure-limit}--\eqref{eq:geo-volume-limit} imply Theorem \ref{thm:main-geo}\textup{(i)}.

Next we treat the backward convergence of the Yamabe quotient $\mcy(g(s))$. By \eqref{eq:Xi}, \eqref{psi-norm-bound}, \eqref{est-phi-inner}, \eqref{neso}, \eqref{psi-norm-bound-mult}, \eqref{est-phi-inner-mult}, and Lemma \ref{lemma:eta_j}, it holds that $\|\nabla\Xi\|_{L^2(\R^n)}\to0$ as $t\to -\infty$.
The bubble interaction integrals satisfy $\int_{\R^n}\nabla U_{\mu_l,\xi_{jl}}\cdot\nabla U_{\mu_{l'},\xi_{j'l'}}\to0$ for $(j, l)\neq (j', l')$, so
\[
\int_{\R^n}|\nabla u(x,t)|^2\, dx \to hk\int_{\R^n}|\nabla U(y)|^2\, dy =\frac{n+2}{4}hk\,\mfm_n.
\]
Together with \eqref{eq:Yamabequo}, \eqref{eq:stcor}, \eqref{eq:geo-critical-invariance}, \eqref{eq:geo-volume-limit}, and \eqref{eq:bubble}, we get that
\[
\mcy(g(s)) = \mcy(\hg(t)) = \frac{4(n-1)}{n-2} \frac{\int_{\R^n}|\nabla u|^2}{(\int_{\R^n}u^{2^*})^{(n-2)/n}} \to (hk)^{2/n}\mcy(g_{\Ss^n})
\]
as $s \to -\infty$. This proves \eqref{eq:Yamabequoc} and Theorem \ref{thm:main-geo}\textup{(ii)}.
\end{proof}

\subsection{Type II behavior and indefinite Ricci curvature: proof of \textup{(iii)}--\textup{(iv)}}
\begin{lemma}[Ricci-flat harmonic conformal factors]
\label{lem:flat-harmonic}
Let $\Omega\subset\R^n$ be a connected open set, let $F>0$ be harmonic in $\Omega$, and suppose that $F^{4/(n-2)}g_{\R^n}$ is Ricci-flat. Then
\be \label{eq:flat-harmonic-classification}
F\equiv C \quad \text{or} \quad F(x)=C|x-x_*|^{2-n} \quad \text{in } \Omega
\ee
for some $C>0$, with $x_*\notin\Omega$ in the second alternative.
\end{lemma}
\begin{proof}
We set $f=F^{-2/(n-2)}$, so that $F^{4/(n-2)}g_{\R^n}=f^{-2}g_{\R^n}$. The conformal Ricci formula gives
\[
0=\Ric_{f^{-2}g_{\R^n}} = (n-2)f^{-1}D^2f+\bigl[f^{-1}\Delta f-(n-1)f^{-2}|\nabla f|^2\bigr]\Id_n.
\]
Taking the Euclidean trace-free part, we obtain
\[
0=(n-2)f^{-1}\(D^2f-\frac{\Delta f}{n}\Id_n\).
\]
Since $f>0$ and $n\ge3$, it follows that $D^2f=(\Delta f/n)\Id_n$.

For any $j\in\{1,\dots,n\}$, we choose $i\ne j$. Direct computations yield
\[
\pp_j\Delta f=n\pp_j f_{ii}=n\pp_i f_{ij}=0.
\]
Thus $\Delta f$ is constant in $\Omega$ and $f(x)=A|x|^2+\mathbf{b}\cdot x+c$ for $x \in \Omega$. Since $F=f^{-(n-2)/2}$ and $\Delta F=0$,
\[
0=\frac{n-2}{2}f^{-\frac{n+2}{2}}\(\frac{n}{2}|\nabla f|^2-f\Delta f\).
\]
It follows that $|\mathbf{b}|^2=4Ac$. If $A=0$, then $\mathbf{b}=0$ and $F$ is constant. If $A\ne0$, then $f(x)=A|x-x_*|^2$. Since $f>0$ in $\Omega$, one has $A>0$ and $x_*\notin\Omega$. This proves \eqref{eq:flat-harmonic-classification}.
\end{proof}

Using the previous lemma, we prove Theorem \ref{thm:main-geo}\textup{(iii)}--\textup{(iv)}.
\begin{prop}\label{prop:geo-typeII}
For sufficiently negative $t_0$, the flow $g$ is of Type II. Moreover, there exists a fixed point $z_0\in\Ss^n\setminus\Sigma$ at which $\Ric_{g(s)}$ is indefinite for every $s<0$.
\end{prop}
\begin{proof}
Recalling \eqref{eq:Ftinf} and \eqref{eq:Phk}, we set $g_\infty := F_\infty^{4/(n-2)}g_{\R^n}$ on $\Omega := \R^n\setminus P_{h,k}(\bvth^*)$. By Lemma \ref{lem:core-profile}, for every compact subset $K\Subset\Omega$,
\be \label{eq:metric-outer-limit}
|t|^{\frac2{n-2}}g_E(t)=F_t^{\frac4{n-2}}g_{\R^n}\to g_\infty \quad \text{in } C^2(K) \quad \text{as } t\to-\infty.
\ee
The function $F_\infty$ is positive and harmonic on $\Omega$, so
\be \label{eq:geo-scalar-flat}
R_{g_\infty}=-\frac{4(n-1)}{n-2}F_\infty^{-p}\Delta F_\infty=0.
\ee
Since $\Omega$ is connected and $F_\infty$ has $hk\ge2$ distinct poles with positive coefficients, Lemma \ref{lem:flat-harmonic} implies that $g_\infty$ is not Ricci-flat.
We choose $x_0\in\Omega$ with $\Ric_{g_\infty}(x_0)\ne0$. The convergence \eqref{eq:metric-outer-limit} and the relation $|\Rm_{cg}|_{cg}=c^{-1}|\Rm_g|_g$ prove 
\be \label{eq:typeII-growth}
|\Rm_{g_E(t)}|_{g_E(t)}(x_0)=\big(|\Rm_{g_\infty}|_{g_\infty}(x_0)+o(1)\big) |t|^{\frac2{n-2}} \to +\infty
\ee
as $t\to-\infty$. At $z_0=\pi(x_0)$, we infer from \eqref{eq:stcor} that
\be \label{eq:typeI-scaling}
|s|\,|\Rm_{g(s)}|_{g(s)}(z_0)=\frac{1-e^{t-t_0}}{p(n-1)}|\Rm_{g_E(t)}|_{g_E(t)}(x_0).
\ee
Since $(1-e^{t-t_0})/(p(n-1)) \to 1/(p(n-1))>0$ as $t \to -\infty$, the right-hand side of \eqref{eq:typeI-scaling} tends to infinity by \eqref{eq:typeII-growth}. Thus $g$ is of Type II.
	
By \eqref{eq:geo-scalar-flat}, the nonzero tensor $\Ric_{g_\infty}(x_0)$ has zero $g_\infty$-trace. Hence there are fixed vectors $v^+,v^-\in T_{x_0}\R^n$ such that
\[
\Ric_{g_\infty}(v^+,v^+)>0>\Ric_{g_\infty}(v^-,v^-).
\]
Since Ricci curvature as a covariant two-tensor is invariant under constant rescaling, \eqref{eq:metric-outer-limit} implies
\be \label{eq:ricci-approximation}
\Ric_{g_E(t)}(x_0)=\Ric_{|t|^{2/(n-2)}g_E(t)}(x_0) \to \Ric_{g_\infty}(x_0) \quad \text{as } t \to -\infty.
\ee
Take $|t_0|$ sufficiently large so that the strict signs persist for all $t\le t_0$. In view of $g(s)=e^{-t}\hg(t)$, $\Ric_{g(s)}$ is indefinite at $z_0$ for every $s<0$. This proves Theorem \ref{thm:main-geo}\textup{(iii)}--\textup{(iv)}.
\end{proof}

\subsection{Geometric equivalence and symmetry breaking: proof of
\textup{(v)}}
\begin{prop}[Concentration sets under geometric equivalence]\label{prop:geo-equiv}
Let $g_1$ be an ancient Yamabe flow produced by our construction, with $h_1$ layers, $k_1$ bubbles per layer, and spherical backward concentration set $\Sigma_1\subset\Ss^n$.
Suppose that a compact ancient Yamabe flow $g_2$ is geometrically equivalent to $g_1$, namely, $g_2(s)=c\,\Phi^*g_1(c^{-1}(s-s_0))$ for some $c>0$, $s_0\in\R$, and $\Phi\in\operatorname{Conf}(\Ss^n)$.
For $i=1,2$, we define the normalized volume measure $\nu_i(s)$ at time $s$ by
\be \label{eq:nui}
\nu_i(s) := \frac{dv_{g_i(s)}}{\Vol_{g_i(s)}(\Ss^n)}.
\ee
Then, as $s\to-\infty$,
\be \label{eq:nu2conv}
\nu_2(s) \rightharpoonup (\Phi^{-1})_*\(\frac1{h_1k_1}\sum_{z\in\Sigma_1}\delta_z\) = \frac1{h_1k_1}\sum_{z\in\Sigma_1}\delta_{\Phi^{-1}(z)}.
\ee
In particular, the concentration set $\Sigma_2$ of $g_2$ consists of exactly $h_1k_1$ points and equals $\Phi^{-1}(\Sigma_1)$.

If $g_2$ is also produced by our construction, with $h_2$ layers, $k_2$ bubbles per layer, and concentration set $\Sigma_2$, then $h_1k_1=h_2k_2$ and $\Sigma_2=\Phi^{-1}(\Sigma_1)$. Consequently, two constructed flows with different numbers of bubbles are geometrically inequivalent.
\end{prop}
\begin{proof}
By \eqref{eq:nui}, the normalized volume measure is invariant under constant rescaling of the metric.
Moreover, the Riemannian volume measures satisfy $dv_{\Phi^*g}=(\Phi^{-1})_*dv_g$. Thus $\nu_2(s)=(\Phi^{-1})_*\nu_1(c^{-1}(s-s_0))$.
Since $c^{-1}(s-s_0)\to-\infty$ as $s\to-\infty$, Proposition \ref{prop:concentration} and continuity of pushforward under weak convergence establish the first assertion \eqref{eq:nu2conv}.

If $g_2$ is also obtained by the construction, Proposition \ref{prop:concentration} identifies the same limiting measure as an equally weighted sum of $h_2k_2$ distinct Dirac masses. Consequently, $h_1k_1=h_2k_2$ and $\Sigma_2=\Phi^{-1}(\Sigma_1)$.
\end{proof}

The preceding result shows that Theorem \ref{thm:main-geo}\textup{(v)} holds.
\begin{prop}[Symmetry breaking]
If $hk\ge3$, the flow $g$ is not geometrically equivalent to a rotationally symmetric ancient Yamabe flow.
\end{prop}
\begin{proof}
Suppose that $hk\ge3$ and that $g$ is geometrically equivalent to a rotationally symmetric flow $g_0$.
Then, for some fixed $Q\in O(n+1)$, the flow $g_0$ is invariant at every time under the group $G_Q$ defined in Subsection \ref{subsec:intro1}.
By Proposition \ref{prop:geo-equiv}, its normalized volume measures converge weakly to a measure supported at exactly $hk$ distinct points as $s\to-\infty$.
Since these measures are $G_Q$-invariant, their weak limit is also $G_Q$-invariant, and hence so is its support.
Every $G_Q$-orbit outside the fixed-point set $\{\pm Q\mathbf e_{n+1}\}$ is an $(n-1)$-sphere, so the finite support of the limit must be contained in this two-point set, contradicting $hk\ge3$.
\end{proof}

\begin{remark}
For $(h,k)=(1,2)$, a two-point concentration set may be invariant under a rotational action, so the preceding argument does not exclude rotational symmetry.
\end{remark}

\subsection{The backward Gromov--Hausdorff limit: proof of \textup{(vi)}}
Fix $V_0>0$ and define
\be \label{eq:geo-vol-normal}
c(t) := \(\frac{V_0}{\Vol_{\hg(t)}(\Ss^n)}\)^{2/n} \quad \text{and} \quad \tg(t) := c(t)\hg(t).
\ee
With $s$ and $t$ related by \eqref{eq:stcor}, the fixed-volume normalization \eqref{eq:tg} satisfies
\[
\tg_{h,k,\bvth^*}(s)=\(\frac{V_0}{\Vol_{g(s)}(\Ss^n)}\)^{2/n}g(s)=\tg(t).
\]
Because $s\to-\infty$ if and only if $t\to-\infty$, establishing the asserted limits for $\tg(t)$ as $t\to-\infty$ is sufficient to prove Theorem \ref{thm:main-geo}\textup{(vi)}.

By \eqref{eq:geo-volume-limit},
\be \label{eq:geo-normal-limit}
c(t)\to c_\infty := \(\frac{V_0}{hk\mfm_n}\)^{2/n} \quad \text{as } t \to -\infty.
\ee
Let $\overline{\R^n} := \R^n\cup\{\infty\}$ and $g_{\bub} := c_\infty U^{\frac4{n-2}}g_{\R^n}$. By \eqref{eq:bubble}, \eqref{stereo-proj}, and \eqref{eq:mfm},
\be \label{eq:geo-round-radius}
U^{\frac4{n-2}}g_{\R^n}=\frac{n(n-2)}{n+2}\pi^*g_{\Ss^n} \quad \text{and} \quad \mfm_n=\left[\frac{n(n-2)}{n+2}\right]^{n/2}|\Ss^n|.
\ee
Thus $g_{\bub}$ extends smoothly to a round metric on $\overline{\R^n}\simeq\Ss^n$.

Let $X$ consist of $hk$ copies of $(\overline{\R^n},g_{\bub})$, with their points $\infty$ identified.
Denote the common point by $z_{\wedge}$ and the inclusion of the $(j,l)$-th branch by $\jmath_{jl}:\overline{\R^n}\to X$, where $\jmath_{jl}(\infty)=z_{\wedge}$ for $(j,h)\in\mci$. The \emph{path metric} on $X$ is
\be \label{eq:geo-wedge-dist}
d_X\big(\jmath_{jl}(y),\jmath_{j'l'}(y')\big) = \begin{cases}
d_{g_{\bub}}(y,y') &\text{if }(j,l)=(j',l'),\\
d_{g_{\bub}}(y,\infty)+d_{g_{\bub}}(y',\infty) &\text{if }(j,l)\ne(j',l').
\end{cases}
\ee
Each branch of $X$ is a round $n$-sphere of radius $\big(V_0/(hk|\Ss^n|)\big)^{1/n}$ and volume
$\Vol_{g_{\bub}}(\overline{\R^n}) =c_\infty^{n/2}\mfm_n=V_0/(hk)$.

\begin{definition}[Gromov--Hausdorff distance and distortion] \label{def:GH}
Let $(\msA,d_\msA)$ and $(\msB,d_\msB)$ be nonempty compact metric spaces. The \emph{distance distortion} of a map $F:\msA\to \msB$, not necessarily continuous, is
\[
\operatorname{dis}(F) := \sup_{\msa,\,\msa'\in \msA} \left|d_\msA(\msa,\msa')-d_\msB(F(\msa),F(\msa'))\right|.
\]
A \emph{correspondence} between $\msA$ and $\msB$ is a subset $\mcr\subset \msA\times \msB$ whose coordinate projections onto $\msA$ and $\msB$ are both surjective. Its distortion is
\[
\operatorname{dis}(\mcr) := \sup_{\substack{(\msa,\msb),\,(\msa',\msb')\in\mcr}} \left|d_\msA(\msa,\msa')-d_\msB(\msb,\msb')\right|.
\]
The \emph{Gromov--Hausdorff distance} can be defined by
\be \label{eq:GH-correspondence}
d_{\mathrm{GH}}(\msA,\msB) := \frac12\inf_{\mcr}\operatorname{dis}(\mcr),
\ee
where the infimum is taken over all correspondences $\mcr$ between $\msA$ and $\msB$.
\end{definition}
If every point of $\msB$ lies within distance $\vep>0$ of some point of $F(\msA)$, then
\be \label{eq:GH-map-bound}
d_{\mathrm{GH}}(\msA,\msB) \le \frac12\operatorname{dis}(F)+\vep.
\ee
Indeed,
\[
\mcr_\vep := \{(\msa,\msb)\in \msA\times \msB:d_\msB(F(\msa),\msb)\le\vep\}
\]
is a correspondence, and the triangle inequality implies $\operatorname{dis}(\mcr_\vep)\le\operatorname{dis}(F)+2\vep$. Thus \eqref{eq:GH-map-bound} follows from \eqref{eq:GH-correspondence}.

\begin{prop}[Backward wedge-of-spheres limit] \label{prop:geo-GH-limit}
As $t\to-\infty$,
\be \label{eq:geo-GH-limit}
(\Ss^n,\tg(t)) \xrightarrow{\mathrm{GH}} X, \quad \text{that is,} \quad d_{\mathrm{GH}}\big((\Ss^n,\tg(t)),X\big) \to 0.
\ee
Here, $\xrightarrow{\mathrm{GH}}$ denotes Gromov--Hausdorff convergence.
\end{prop}
\begin{proof}
For $(j,h)\in\mci$, we define $\iota_{jl,t}:\overline{B(0,R(t))} \to \overline{\mcc_{jl}(t)}$ (see \eqref{eq:mccjl}) by
\[
\iota_{jl,t}(y) := \pi\big(\xi_{jl}(t)+\mu_l(t)y\big) \quad \text{for } y \in \overline{B(0,R(t))}.
\]
Using $g_E(t)=u(\cdot,t)^{4/(n-2)}g_{\R^n}$ and the change of variables $x=\xi_{jl}(t)+\mu_l(t)y$, we obtain
\[
\iota_{jl,t}^*\tg(t) = c(t) \left[\mu_l(t)^{\frac{n-2}{2}}u(\xi_{jl}(t)+\mu_l(t)y,t)\right]^{\frac4{n-2}} g_{\R^n} \quad \text{on } B(0,R(t)).
\]
By \eqref{eq:geo-core-ratio} and \eqref{eq:geo-normal-limit}, there exists $\vep_{\met}(t)\to0$ as $t\to -\infty$ such that, uniformly in $j$ and $l$,
\be \label{eq:geo-core-bilipschitz}
(1-\vep_{\met}(t))g_{\bub} \le \iota_{jl,t}^*\tg(t) \le (1+\vep_{\met}(t))g_{\bub} \quad \text{on } \overline{B(0,R(t))}
\ee
for all sufficiently negative $t$.

The formula $g_{\bub}=c_\infty U^{4/(n-2)}g_{\R^n}$, \eqref{eq:bubble}, and \eqref{eq:diam} give
\be \label{eq:geo-cap-diameter}
\sup_{|y|\ge L}d_{g_{\bub}}(y,\infty)+\diam^{\mint}_{g_{\bub}}(\pp B(0,L)) \le C\int_L^\infty\frac{dr}{1+r^2} +\frac{CL}{1+L^2} \le\frac{C}{L}
\ee
for any $L \ge 1$. If a minimizing $g_{\bub}$-geodesic joining two points $y,y'\in B(0,L)$ meets $\pp B(0,L)$, we replace the portion between its first and last intersections with $\pp B(0,L)$ by a curve along the boundary of length at most $C/L$.
Radially contracting this curve by a factor of $1-\vep$ and joining $(1-\vep)y$ to $y$ and $(1-\vep)y'$ to $y'$ by radial segments, we obtain curves in $B(0,L)$ whose lengths converge to that of the original curve as $\vep\downarrow0$.
Since restricting the admissible curves to $B(0,L)$ cannot decrease distances, we obtain
\be \label{eq:geo-truncated-round-dist}
\sup_{y,y'\in B(0,L)} \left|d^{\mint}_{B(0,L),g_{\bub}}(y,y')-d_{g_{\bub}}(y,y')\right| \le \frac{C}{L}.
\ee
By \eqref{eq:geo-core-bilipschitz}, the length of the image of any curve in $B(0,R(t))$ lies between $\sqrt{1-\vep_{\met}(t)}$ and $\sqrt{1+\vep_{\met}(t)}$ times its length in $g_{\bub}$.
Taking infima over curves, applying \eqref{eq:geo-truncated-round-dist} with $L=R(t)$, and using the finite diameter of the round sphere $(\overline{\R^n},g_{\bub})$, we obtain
\be \label{eq:geo-int-core-dist}
\sup_{y,y'\in B(0,R(t))} \left|d^{\mint}_{\mcc_{jl}(t),\tg(t)}\big(\iota_{jl,t}(y),\iota_{jl,t}(y')\big) - d_{g_{\bub}}(y,y')\right| \le C\(\vep_{\met}(t)+R(t)^{-1}\).
\ee

The inequality \eqref{eq:geo-cap-diameter} also shows that
\be \label{eq:geo-model-bdry-dist}
\sup_{y\in B(0,L)} \left|\dist_{g_{\bub}}(y,\pp B(0,L)) -d_{g_{\bub}}(y,\infty)\right|\le\frac{C}{L}.
\ee
A path from a point in $\mcc_{jl}(t)$ to $\pp\mcc_{jl}(t)$ can be stopped at its first intersection with $\pp\mcc_{jl}(t)$ without increasing its length.
The distance to $\pp\mcc_{jl}(t)$ can therefore be computed using paths that remain in $\mcc_{jl}(t)$ before their endpoint.
Applying \eqref{eq:geo-core-bilipschitz} to these paths and using \eqref{eq:geo-model-bdry-dist}, we obtain
\be \label{eq:geo-dist-to-bdry}
\sup_{y\in B(0,R(t))} \left|\dist_{\tg(t)} \big(\iota_{jl,t}(y),\pp\mcc_{jl}(t)\big) - d_{g_{\bub}}(y,\infty)\right| \le C\(\vep_{\met}(t)+R(t)^{-1}\).
\ee
Since $c(t)\to c_\infty>0$, \eqref{eq:geo-outer-path-diameter} also gives
\be \label{eq:geo-normal-outer-path}
\diam^{\mint}_{\tg(t)}(\mco_t) \to 0.
\ee

\medskip
Let us define $\mcg_t:(\Ss^n,\tg(t))\to X$ by
\be \label{eq:geo-GH-map}
\mcg_t(z) := \begin{cases}
\jmath_{jl}\big(\iota_{jl,t}^{-1}(z)\big) &\text{for } z\in\mcc_{jl}(t),\\
z_{\wedge} &\text{for } z\in\mco_t.
\end{cases}
\ee
Since the open cores $\mcc_{jl}$ are pairwise disjoint and $\mco_t$ is their closed complement, $\mcg_t$ is well defined.
Moreover, $\mcg_t$ is Borel measurable because it is continuous on each core and constant on $\mco_t$. Nevertheless, it is discontinuous at every point of each $\pp\mcc_{jl}(t)$. We claim that
\be \label{eq:geo-GH-distort-quan}
\operatorname{dis}(\mcg_t) \le C\big(\vep_{\met}(t)+R(t)^{-1}+\diam^{\mint}_{\tg(t)}(\mco_t)\big) \to 0 \quad \text{as } t \to -\infty.
\ee

First, suppose that $z=\iota_{jl,t}(y)\in\mcc_{jl}(t)$ and $z'=\iota_{j'l',t}(y')\in\mcc_{j'l'}(t)$, where $(j,l)\ne(j',l')$.
Every path joining $z$ to $z'$ contains disjoint portions from $z$ to $\pp\mcc_{jl}(t)$ and from $\pp\mcc_{j'l'}(t)$ to $z'$. Thus \eqref{eq:geo-dist-to-bdry} and \eqref{eq:geo-wedge-dist} imply
\begin{align*}
d_{\tg(t)}(z,z') &\ge \dist_{\tg(t)}\big(z,\pp\mcc_{jl}(t)\big) +\dist_{\tg(t)}\big(z',\pp\mcc_{j'l'}(t)\big)\\
&\ge d_X\big(\mcg_t(z),\mcg_t(z')\big) -C\big(\vep_{\met}(t)+R(t)^{-1}\big).
\end{align*}
For the reverse inequality, choose paths from $z$ to $\pp\mcc_{jl}(t)$ and from $\pp\mcc_{j'l'}(t)$ to $z'$
whose lengths are arbitrarily close to the respective distances to these sets.
Since both boundaries lie in $\mco_t$, their endpoints can be joined by a path in $\mco_t$ of length at most $\diam^{\mint}_{\tg(t)}(\mco_t)$, up to an arbitrarily small error.
Applying \eqref{eq:geo-dist-to-bdry} to the concatenated path, we obtain
\[
d_{\tg(t)}(z,z') \le d_X\big(\mcg_t(z),\mcg_t(z')\big) +C\big(\vep_{\met}(t)+R(t)^{-1}\big) +\diam^{\mint}_{\tg(t)}(\mco_t).
\]

Next, suppose that $z=\iota_{jl,t}(y)$ and $z'=\iota_{jl,t}(y')$ both belong to $\mcc_{jl}(t)$. Since the distance in $(\Ss^n,\tg(t))$ is at most the intrinsic distance in $\mcc_{jl}(t)$, \eqref{eq:geo-int-core-dist} yields the upper bound.
For the lower bound, let $\Gamma$ be any rectifiable path joining $z$ to $z'$ in $\Ss^n$. If $\Gamma$ is contained in $\mcc_{jl}(t)$, the lower bound follows from \eqref{eq:geo-int-core-dist}.
Otherwise, $\Gamma$ meets $\pp\mcc_{jl}(t)$. In this case, we see from \eqref{eq:geo-dist-to-bdry} that
\begin{align*}
\operatorname{Length}_{\tg(t)}(\Gamma) &\ge \dist_{\tg(t)}\big(z,\pp\mcc_{jl}(t)\big) +\dist_{\tg(t)}\big(z',\pp\mcc_{jl}(t)\big)\\
&\ge d_{g_{\bub}}(y,\infty)+d_{g_{\bub}}(y',\infty)-C\big(\vep_{\met}(t)+R(t)^{-1}\big)\\
&\ge d_{g_{\bub}}(y,y')-C\big(\vep_{\met}(t)+R(t)^{-1}\big).
\end{align*}
Taking the infimum over $\Gamma$ establishes the lower bound.

For a spherical core point and an outer point, the bounds follow from \eqref{eq:geo-dist-to-bdry} and \eqref{eq:geo-normal-outer-path}. For two outer points, \eqref{eq:geo-normal-outer-path} suffices. These estimates prove \eqref{eq:geo-GH-distort-quan}.

\medskip
The image of $\mcg_t$ contains $\jmath_{jl}(B(0,R(t)))$ for every $j,l$, as well as $z_{\wedge}$. By \eqref{eq:geo-cap-diameter}, every $q\in X\setminus\mcg_t(\Ss^n)$ satisfies $d_X(q,z_{\wedge})\le C/R(t)$.
Since $z_{\wedge}\in\mcg_t(\Ss^n)$, every point of $X$ is at distance at most $C/R(t)$ from some point of $\mcg_t(\Ss^n)$.
Applying \eqref{eq:GH-map-bound} with $F=\mcg_t$ and using \eqref{eq:geo-GH-distort-quan}, we obtain
\[
d_{\mathrm{GH}}\big((\Ss^n,\tg(t)),X\big) \le \frac12\operatorname{dis}(\mcg_t)+\frac{C}{R(t)} \to 0 \quad \text{as } t \to -\infty.
\]
This proves \eqref{eq:geo-GH-limit}.
\end{proof}

\begin{corollary}[Measured Gromov--Hausdorff limit]
Let
\[
\nu_\infty := \sum_{l=1}^h\sum_{j=1}^k(\jmath_{jl})_*\,dv_{g_{\bub}}.
\]
Then 
\be \label{eq:geo-measured-GH-limit}
\(\Ss^n,d_{\tg(t)},\frac1{V_0}dv_{\tg(t)}\) \xrightarrow{\mathrm{mGH}} \(X,d_X,\frac1{V_0}\nu_\infty\) \quad \text{as } t \to -\infty.
\ee
Here, $\xrightarrow{\mathrm{mGH}}$ denotes measured Gromov--Hausdorff convergence. More precisely, the Borel measurable maps $\mcg_t:(\Ss^n,d_{\tg(t)})\to(X,d_X)$ defined in \eqref{eq:geo-GH-map} realize the Gromov--Hausdorff convergence in \eqref{eq:geo-GH-limit}, and the normalized volume measures $V_0^{-1}dv_{\tg(t)}$ satisfy
\[
(\mcg_t)_*\(\frac1{V_0}dv_{\tg(t)}\) \rightharpoonup \frac1{V_0}\nu_\infty \quad \text{weakly on } X \text{ as } t \to -\infty.
\]
Moreover,
\be \label{eq:geo-limit-measure-masses}
\nu_\infty(X)=V_0,\quad \nu_\infty(\{z_{\wedge}\})=0,\quad \text{and} \quad \nu_\infty\big(\jmath_{jl}(\overline{\R^n})\big)=\frac{V_0}{hk}
\ee
for $(j,h)\in\mci$.
\end{corollary}
\begin{proof}
The estimate \eqref{eq:geo-core-bilipschitz} and $(1\pm\vep)^{n/2}=1+O(\vep)$ provide
\[
\left|c(t)^{n/2}\mu_l(t)^nu^{2^*}(\xi_{jl}(t)+\mu_l(t)y,t)-c_\infty^{n/2}U^{2^*}(y)\right| \le C\vep_{\met}(t)U^{2^*}(y)
\]
for $y\in B(0,R(t))$. By \eqref{eq:geo-GH-map}, the map $\mcg_t$ is Borel measurable and equals $z_\wedge$ on $\mco_t$. For $\vph\in C^0(X)$, the change of variables $x=\xi_{jl}(t)+\mu_l(t)y$ give
\begin{align*}
\int_{\Ss^n}\vph\circ\mcg_t\,dv_{\tg(t)} &= \sum_{l=1}^h\sum_{j=1}^k \int_{|y|<R(t)}\vph\bigl(\jmath_{jl}(y)\bigr)c(t)^{n/2}\mu_l(t)^n u^{2^*}\bigl(\xi_{jl}(t)+\mu_l(t)y,t\bigr)\,dy \\
&\ +\vph(z_\wedge)\Vol_{\tg(t)}(\mco_t).
\end{align*}
The definition of $\nu_\infty$ also gives
\[
\int_X\vph\,d\nu_\infty =c_\infty^{n/2}\sum_{l=1}^h\sum_{j=1}^k \int_{\R^n}\vph\bigl(\jmath_{jl}(y)\bigr)U^{2^*}(y)\,dy.
\]
Subtracting these identities and applying the preceding density estimate, we obtain
\begin{align*}
&\ \left|\int_{\Ss^n}\vph\circ\mcg_t\,dv_{\tg(t)}-\int_X\vph\,d\nu_\infty\right|\\
&\le C\|\vph\|_{L^\infty(X)} \left[\vep_{\met}(t)\int_{\R^n}U^{2^*}\,dy +\int_{|y|\ge R(t)}U^{2^*}\,dy+\Vol_{\tg(t)}(\mco_t)\right]\\
&\le C\|\vph\|_{L^\infty(X)}\left[\vep_{\met}(t)+R(t)^{-n}\right] \to 0.
\end{align*}
Here, \eqref{eq:bubble} implies $\int_{|y|\ge R(t)}U^{2^*}\,dy\le CR(t)^{-n}$, while the proof of \eqref{eq:geo-outer-volume} and \eqref{eq:geo-normal-limit} give $\Vol_{\tg(t)}(\mco_t)\le CR(t)^{-n}$. Hence $(\mcg_t)_*dv_{\tg(t)}\rightharpoonup\nu_\infty$ on $X$.
Together with \eqref{eq:geo-GH-limit}, this proves \eqref{eq:geo-measured-GH-limit}. Each of the $hk$ copies of $(\overline{\R^n},g_{\bub})$ has volume $V_0/(hk)$, and $dv_{g_{\bub}}(\{\infty\})=0$.
Since their images under the inclusions $\jmath_{jl}$ intersect only at $z_\wedge$, the definition of $\nu_\infty$ implies \eqref{eq:geo-limit-measure-masses}.
\end{proof}

\subsection{Configuration moduli: proof of \textup{(vii)}}\label{subsec:equimod}
Throughout this subsection, we fix $n\ge3$ and $h,k\ge2$.

\medskip
For $z,z'\in\Ss^n\subset\R^{n+1}$, the chordal distance is $\delta_{\ch}(z,z'):=|z-z'|_{\R^{n+1}}$. For four pairwise distinct points, the \emph{chordal cross-ratio} is
\be \label{eq:ccr}
\cro(z_1,z_2;z_3,z_4) := \frac{\delta_{\ch}(z_1,z_3)\delta_{\ch}(z_2,z_4)} {\delta_{\ch}(z_1,z_4)\delta_{\ch}(z_2,z_3)}.
\ee
For a M\"obius transformation $\Phi$, let $f_\Phi$ be a positive function on $\Ss^n$ defined by $\Phi^*g_{\Ss^n}=f_\Phi^2g_{\Ss^n}$. Then, one has
\[
\delta_{\ch}(\Phi(z),\Phi(z'))^2=f_\Phi(z)f_\Phi(z')\delta_{\ch}(z,z')^2.
\]
The factors cancel in each cross-ratio. Hence $\cro$ is M\"obius invariant.

For $\bvth^*,\bvth^{*\prime}\in\mca_h$, we write $\bvth^*\sim_{\mathrm{conf}}\bvth^{*\prime}$ 
whenever there exists a M\"obius transformation mapping $\Sigma_{h,k}(\bvth^*)$ onto $\Sigma_{h,k}(\bvth^{*\prime})$ (see \eqref{eq:mcah} and \eqref{eq:Sigmahk}) and then denote
\[
\mfC_{h,k} := \mca_h/\!\sim_{\mathrm{conf}}.
\]
For $(j,h)\in\mci$ and the points ${\bf q}_{jl}^*$ in \eqref{eq:bfq}, let $z_{jl}^*(\bvth^*) := \pi({\bf q}_{jl}^*)=({\bf q}_{jl}^*,0)$.
A M\"obius transformation between two such configurations is \emph{label-preserving} if it maps $z_{jl}^*(\bvth^*)$ to $z_{jl}^*(\bvth^{*\prime})$ for every $(j,l)\in\mci$.

Given $\bvth^*\in\mca_h$, we define
\be \label{eq:config-gaps}
\begin{cases}
\lambda_l := \arctanh\vth_l^* &\text{for } l=1,\dots,h,\\
\Delta_l := \lambda_{l+1}-\lambda_l>0 &\text{for } l=1,\dots,h-1.
\end{cases}
\ee
Let $\bDelta:=(\Delta_1,\dots,\Delta_{h-1})$, and define $\lambda_l'$ and $\Delta_l'$ analogously for $\bvth^{*\prime}$.
For each $\bDelta\in(0,\infty)^{h-1}$, we set the normalized height vector $\bvth^*(\bDelta)$ by
\be \label{eq:normalized-heights}
\vth_1^*(\bDelta):=0,\qquad \vth_l^*(\bDelta) := \tanh\(\sum_{r=1}^{l-1}\Delta_r\) \quad \text{for } l=2,\dots,h.
\ee

For $i=(j,l)\in\mci$, we write $z_i^*(\bDelta) := z_{jl}^*(\bvth^*(\bDelta))$. The \emph{normalized configuration} $\{z_i^*(\bDelta)\}_{i\in\mci}$ determines the array
\[
\mbC(\bDelta) := \bigl(\cro(z_{i_1}^*(\bDelta),z_{i_2}^*(\bDelta); z_{i_3}^*(\bDelta),z_{i_4}^*(\bDelta))\bigr)_{\substack{i_1,i_2,i_3,i_4\in\mci\\ \text{pairwise distinct}}},
\]
whose $(i_1,i_2,i_3,i_4)$-component is denoted by $C_{i_1i_2i_3i_4}(\bDelta)$. Let $S_{\mci}$ denote the symmetric group of the label set $\mci$, consisting of all bijections $\mci\to\mci$.
For $\mfs\in S_{\mci}$, we define its action on the cross-ratio arrays by
\[
(\mfs\cdot\mbC(\bDelta))_{i_1i_2i_3i_4} := C_{\mfs^{-1}(i_1)\mfs^{-1}(i_2)\mfs^{-1}(i_3)\mfs^{-1}(i_4)}(\bDelta).
\]

\medskip
We prove Theorem \ref{thm:main-geo}\textup{(vii)} through the following stronger result, which also provides a detailed description of the conformal equivalence classes of concentration configurations.
\begin{prop}[Configuration moduli realized by the multi-layer family]
\label{prop:cont-moduli}
For every fixed $n\ge3$ and $h,k\ge2$, the following statements hold.
\begin{enumerate}
\item[(i)] The configurations for $\bvth^*,\bvth^{*\prime}\in\mca_h$ are related by a label-preserving M\"obius transformation if and only if $\bDelta=\bDelta'$, with $\bDelta,\bDelta'$ defined by \eqref{eq:config-gaps}.
    Thus $\bDelta\in(0,\infty)^{h-1}$ parametrizes the classes under label-preserving M\"obius equivalence.
\item[(ii)] We define
    \be \label{eq:gap-equiv}
    \bDelta \approx_k \bDelta' \quad \Longleftrightarrow \quad \mbC(\bDelta')=\mfs\cdot\mbC(\bDelta) \quad \text{for some } \mfs\in S_{\mci}.
    \ee
    With their quotient topologies, $\mfC_{h,k}$ and $(0,\infty)^{h-1}/\!\approx_k$ are homeomorphic. The homeomorphism sends $[\bvth^*]_{\sim_{\mathrm{conf}}}$ to $[\bDelta]_{\approx_k}$, where $\bDelta$ is defined in \eqref{eq:config-gaps}.
    Its inverse sends $[\bDelta]_{\approx_k}$ to $[\bvth^*(\bDelta)]_{\sim_{\mathrm{conf}}}$, with $\bvth^*(\bDelta)$ defined in \eqref{eq:normalized-heights}.
    Moreover, each $\approx_k$-class contains at most $(hk)!$ elements.
\item[(iii)] Every class in $\mfC_{h,k}$ occurs as the backward concentration configuration of a flow $g_{h,k,\bvth^*}$ constructed above.
    Flows whose concentration configurations lie in distinct classes are geometrically inequivalent.
    There exists a nonempty open set $\mcu\subset(0,\infty)^{h-1}$ such that any two flows $g_{h,k,\bvth^*(\bDelta)}$ and $g_{h,k,\bvth^*(\bDelta')}$ constructed for distinct $\bDelta,\bDelta'\in\mcu$, with the heights in \eqref{eq:normalized-heights}, are geometrically inequivalent.
\end{enumerate}
For fixed $V_0>0$, all flows in \textup{(iii)} have the same fixed-volume backward Gromov--Hausdorff limit given in Theorem \ref{thm:main-geo}\textup{(vi)}.
\end{prop}

\noindent\emph{Proof of \textup{(i)}.}
For $\theta\in(-1,1)$, we define
\be \label{eq:axial-mobius}
\Phi_\theta(z) := \frac{\sqrt{1-\theta^2}(z-z_3{\bf e}_3)+(z_3-\theta){\bf e}_3}{1-\theta z_3} \quad \text{for } z\in\Ss^n.
\ee
Since $1-\theta z_3>0$ on $\Ss^n$, the map is globally defined, with inverse $\Phi_{-\theta}$. Differentiation along $T_z\Ss^n$ shows that
\[
\Phi_\theta^*g_{\Ss^n}=\frac{1-\theta^2}{(1-\theta z_3)^2}g_{\Ss^n}.
\]
Thus $\Phi_\theta$ is conformal. On the configurations in \eqref{eq:bfq}, it preserves the angular directions $\omega_j$ in \eqref{eq:omegaj} and acts on the heights by
\[
\vth_l^*\mapsto\frac{\vth_l^*-\theta}{1-\theta\vth_l^*} \quad \text{and} \quad \lambda_l\mapsto\lambda_l-\arctanh\theta.
\]
If $\bDelta=\bDelta'$, then $\lambda_l-\lambda_l'=\lambda_1-\lambda_1'$ for all $l=1,\dots,h$.
Choosing $\theta=\tanh(\lambda_1-\lambda_1') \in (-1,1)$, we obtain a label-preserving M\"obius transformation $\Phi_\theta$ between the two configurations $\bvth^*$ and $\bvth^{*\prime}$.
In particular, $\Phi_{\vth_1^*}$ maps the prescribed heights to those in \eqref{eq:normalized-heights}.

Conversely, \eqref{eq:bfq} and \eqref{eq:config-gaps} imply
\be \label{eq:config-dist}
\delta_{\ch}(z_{jl}^*,z_{j'm}^*)^2 = \frac{2\big[\cosh(\lambda_l-\lambda_m) -\cos(2\pi(j-j')/k)\big]}{\cosh\lambda_l\cosh\lambda_m}.
\ee
For the four points with labels $(1,l),(2,l),(1,l+1),(2,l+1)$, we have
\be \label{eq:adj-layer-ratios}
R_l := \cro(z_{1l}^*,z_{2l}^*;z_{1,l+1}^*,z_{2,l+1}^*) =\frac{\sinh^2(\Delta_l/2)}{\sinh^2(\Delta_l/2)+\sin^2(\pi/k)}.
\ee
The right-hand side is strictly increasing in $\Delta_l>0$, with image $(0,1)$, and
\be \label{eq:rec-config-gaps}
\Delta_l=2\operatorname{arsinh}\(\sin\frac\pi k\sqrt{\frac{R_l}{1-R_l}}\).
\ee
Since a label-preserving M\"obius transformation preserves every $R_l$, \eqref{eq:rec-config-gaps} implies $\bDelta=\bDelta'$. This proves the equivalence stated in \textup{(i)}.

\medskip
By \eqref{eq:normalized-heights} and \eqref{eq:config-dist}, the map $\mbC$ is continuous.
Since \eqref{eq:rec-config-gaps} recovers every $\Delta_l$ from $\mbC(\bDelta)$, it is also injective. This also establishes the parametrization asserted in \textup{(i)}.

\medskip\noindent
\emph{Proof of \textup{(ii)}.}
Condition \eqref{eq:gap-equiv} means that a single permutation $\mfs\in S_{\mci}$ makes the cross-ratios of the two normalized configurations $\{z_i^*(\bDelta)\}_{i\in\mci}$ and $\{z_i^*(\bDelta')\}_{i\in\mci}$ agree for every ordered quadruple of distinct labels.
By Crane and Short \cite[Theorem 1.2]{CS11}, there exists a M\"obius transformation $\Phi$ satisfying 
\be \label{eq:perm}
\Phi(z_i^*(\bDelta))=z_{\mfs(i)}^*(\bDelta') \quad \text{for all } i\in\mci.
\ee
Conversely, any M\"obius transformation $\Phi$ between the two normalized configurations induces a unique permutation $\mfs\in S_{\mci}$ satisfying \eqref{eq:perm}, since their points are distinct.
M\"obius invariance of the ordered cross-ratios then implies $\mbC(\bDelta')=\mfs\cdot\mbC(\bDelta)$.
Together with the normalized representatives established in \textup{(i)}, this proves the bijection $[\boldsymbol\vth^*]_{\sim_{\mathrm{conf}}} \mapsto[\bDelta]_{\approx_k}$.

Moreover, the map $\bvth^*\mapsto(\lambda_1,\bDelta)$ is a homeomorphism from $\mca_h$ onto $\R\times(0,\infty)^{h-1}$.
Hence $\bvth^*\mapsto\bDelta$ is a continuous, open surjection. Since the two equivalence relations correspond under this map, the induced bijection is a homeomorphism with respect to the quotient topologies.

Finally, for fixed $\bDelta$ and $\mfs\in S_{\mci}$, the injectivity of $\mbC$ allows at most one $\bDelta'$ satisfying \eqref{eq:gap-equiv} with this permutation.
Since $|S_{\mci}|=(hk)!$, each $\approx_k$-class contains at most $(hk)!$ elements.

\medskip\noindent
\emph{Proof of \textup{(iii)}.}
We first find a nonempty open set $\mcu\subset(0,\infty)^{h-1}$ such that no two distinct parameters in $\mcu$ are equivalent under $\approx_k$.
Let the \emph{stabilizer} of $\bDelta\in(0,\infty)^{h-1}$ under a group action of $S_{\mci}$ be
\[
\Stab(\bDelta) := \{\mfs\in S_{\mci}:\mfs\cdot\mbC(\bDelta)=\mbC(\bDelta)\}.
\]
Since $|\Stab(\bDelta)|$ takes only finitely many values, one can choose $\bDelta^0$ at which it is minimal, and set $\Gamma_0 := \Stab(\bDelta^0)$.

For each $\mfs\in S_{\mci}\setminus\Gamma_0$, it holds that $\mfs\cdot\mbC(\bDelta^0) \ne \mbC(\bDelta^0)$.
The map $(\bDelta,\bDelta') \mapsto \mfs\cdot\mbC(\bDelta)-\mbC(\bDelta')$ is continuous and nonzero at $(\bDelta^0,\bDelta^0)$.
Hence there is an open neighborhood $\mcu_{\mfs}$ of $\bDelta^0$ such that $\mfs\cdot\mbC(\bDelta)\ne\mbC(\bDelta')$ for all $\bDelta,\bDelta'\in\mcu_{\mfs}$.
Intersecting these finitely many neighborhoods, we obtain an open neighborhood $\mcu$ of $\bDelta^0$ on which this inequality holds for every $\mfs\in S_{\mci}\setminus\Gamma_0$.

For $\bDelta\in\mcu$, taking $\bDelta'=\bDelta$ in this inequality shows that no permutation outside $\Gamma_0$ fixes $\mbC(\bDelta)$. Hence $\Stab(\bDelta)\subseteq\Gamma_0$.
From the minimality of $|\Stab(\bDelta^0)|$, it follows that $|\Stab(\bDelta)| \ge |\Stab(\bDelta^0)|=|\Gamma_0|$. Therefore $\Stab(\bDelta)=\Gamma_0$ for every $\bDelta\in\mcu$.

Now suppose that $\bDelta,\bDelta'\in\mcu$ satisfy $\bDelta\approx_k\bDelta'$. By \eqref{eq:gap-equiv}, there exists an $\mfs\in S_{\mci}$ such that $\mbC(\bDelta')=\mfs\cdot\mbC(\bDelta)$.
By the choice of $\mcu$, we have $\mfs\in\Gamma_0=\Stab(\bDelta)$. Thus $\mbC(\bDelta')=\mfs\cdot\mbC(\bDelta)=\mbC(\bDelta)$. Since $\mbC$ is injective, we conclude that $\bDelta'=\bDelta$.

Consequently, by \textup{(ii)} and Proposition \ref{prop:geo-equiv},
the flows $g_{h,k,\bvth^*(\bDelta)}$ and $g_{h,k,\bvth^*(\bDelta')}$
are geometrically inequivalent whenever $\bDelta,\bDelta'\in\mcu$ are distinct.

\medskip
By \textup{(i)}, Theorem \ref{thm:main-an}, and Proposition \ref{prop:concentration}, every class in $\mfC_{h,k}$ is realized by a constructed flow with heights normalized as in \eqref{eq:normalized-heights}.
By Proposition \ref{prop:geo-equiv}, if no M\"obius transformation of $\Ss^n$ maps $\Sigma_{h,k}(\boldsymbol\vth^*)$ onto $\Sigma_{h,k}(\boldsymbol\vth^{*\prime})$, then the corresponding ancient Yamabe flows are geometrically inequivalent.

\medskip\noindent
\emph{Proof of the final assertion.}
For fixed $V_0>0$, Proposition \ref{prop:geo-GH-limit} shows that all flows in \textup{(iii)} have the same fixed-volume backward Gromov--Hausdorff limit, namely the wedge of $hk$ round $n$-spheres, each of volume $V_0/(hk)$. \qed

\begin{proof}[Completion of the proof of Theorem \ref{thm:main-geo}\textup{(vii)}]
The map $\bDelta\mapsto\bvth^*(\bDelta)$ is a diffeomorphism from $(0,\infty)^{h-1}$ onto $\{\bvth^*\in\mca_h:\vth_1^*=0\}$.
Hence $\mcv:=\{\bvth^*(\bDelta):\bDelta\in\mcu\}$ is nonempty and relatively open in this space, and the flows corresponding to distinct elements of $\mcv$ are geometrically inequivalent.
This proves the second assertion of Theorem \ref{thm:main-geo}\textup{(vii)}; the first follows from Proposition \ref{prop:geo-equiv}.
\end{proof}

\begin{remark}[A continuous modulus]\label{rem:continuous-modulus}
Define a positive continuous function $\mfc$ on $(0,\infty)^{h-1}$ by
\[
\mfc(\bDelta) := \min_{\substack{i_1,i_2,i_3,i_4\in\mci \\ \text{pairwise distinct}}} C_{i_1i_2i_3i_4}(\bDelta).
\]
Permuting the labels merely reorders the entries of $\mbC(\bDelta)$, so $\mfc$ is constant on every $\approx_k$-class.
Keeping $\Delta_2,\dots,\Delta_{h-1}$ fixed when $h>2$, \eqref{eq:adj-layer-ratios} yields
\[
0<\mfc(\bDelta) \le \frac{\sinh^2(\Delta_1/2)}{\sinh^2(\Delta_1/2)+\sin^2(\pi/k)} \to 0 \quad \text{as } \Delta_1 \downarrow 0.
\]
Thus $\mfc$ is nonconstant along this one-parameter family, and its image contains a nondegenerate interval.
By Propositions \ref{prop:geo-equiv} and \ref{prop:cont-moduli}\textup{(ii)}, geometrically equivalent flows $g_{h,k,\bvth^*(\bDelta)}$ and $g_{h,k,\bvth^*(\bDelta')}$ satisfy $\mfc(\bDelta)=\mfc(\bDelta')$.
Hence $\mfc$ is a continuous modulus in the sense defined in Subsection \ref{subsec:intro1}.
\end{remark}

\begin{remark}[Configuration and flow classes] \label{rem:config-versus-flow-moduli}
For fixed $n\ge3$ and $h,k\ge2$, let $\mfM_{h,k}$ denote the geometric-equivalence classes of ancient Yamabe flows constructed in Theorem \ref{thm:main-an}.
By Proposition \ref{prop:geo-equiv} and Theorem \ref{thm:main-an}, the map $\mfM_{h,k} \to \mfC_{h,k}$, defined by $[g(s)] \mapsto [\bvth^*]_{\sim_{\mathrm{conf}}}$, is well defined and surjective, where $\Sigma_{h,k}(\bvth^*)$ is the backward concentration set of $g(s)$.
Proposition \ref{prop:cont-moduli} classifies these concentration configurations, but does not classify flows sharing the same concentration set.
\end{remark}

\appendix
\section{Technical Arguments}\label{app}
\subsection{Proof of Lemma \ref{lemma:eta_j}}\label{app1}
For $|y|\le R/4$, \eqref{def-cutoff-j} gives
\[
D_j^2(\mu y+\xi_j,t) = \frac{2\mu^2(1+|y|^2)}{1+|\mu y+\xi_j|^2} \le 2\mu^2\(1+\frac{R^2}{16}\) \le (R\mu)^2,
\]
so $\eta_j(\mu y+\xi_j,t)=1$. On the other hand, $1+|x|^2\le3+2|x-\xi_j|^2$ implies
\begin{align*}
D_j(x,t)\le2R\mu \ &\Longrightarrow \ \big(1-4(R\mu)^2\big)|x-\xi_j|^2\le6(R\mu)^2 \\ &\Longrightarrow\ |x-\xi_j|\le3R\mu.
\end{align*}
Thus $\eta_j(\mu y+\xi_j,t)=0$ for $|y|\ge5R$. Differentiating \eqref{def-cutoff-j} in the transition region, we obtain
\[
|\nabla_x^{\ell}\eta_j|\le C(R\mu)^{-\ell}\quad \text{for } \ell=1,2 \quad \text{and} \quad
|\pp_t\eta_j|\le C\(\frac{|\dot R|}{R} +\frac{|\dot\mu|}{\mu}+\frac{|\dot\xi_j|}{R\mu}\).
\]
The bounds for $\eta_{j,2}$ follow by replacing $R$ with $2R$. \qed

\subsection{\texorpdfstring{Estimates of the nonlinear term $N(\Xi)$}{Estimates of the nonlinear term}}\label{app:nonlinear-estimates}
In this subsection, we use the single-layer notation from Sections \ref{sec:sing} to \ref{sec:inner}. We also assume \eqref{mu}--\eqref{dt}, and \eqref{mu2}, as well as $\|\phi\|_{\sharp\sharp,a,1;\sigma}\le\msM_1$ and $\|\psi\|_{\widetilde{**},\alpha,\beta;\sigma}\le\msM_2$.

By the norms in Subsection \ref{subsec:norms},
\be \label{ppt}
\begin{aligned}
|\nabla_y^m\phi(y,t)| &\le C\|\phi\|_{\sharp\sharp,a,1;\sigma}|t|^{-1}\la y\ra^{-a-m} \quad \text{for } m=0,1,2,\\
|\pp_t\phi(y,t)| &\le C\|\phi\|_{\sharp\sharp,a,1;\sigma}|t|^{-1}\la y\ra^{2-a}.
\end{aligned}
\ee
For $x\in B(0,\msR_0)$, we choose a center $\xi_j$ closest to $x$ and write $y=\mu^{-1}(x-\xi_j)$. Then
\be \label{global-psi}
\begin{aligned}
|\nabla_x^m\psi(x,t)| &\le C\|\psi\|_{\widetilde{**},\alpha,\beta;\sigma}|t|^{-\beta}\mu^{-m} \la y\ra^{-\alpha-m} \quad \text{for } m=0,1,2,\\
|\pp_t\psi(x,t)| &\le C\|\psi\|_{\widetilde{**},\alpha,\beta;\sigma}|t|^{-\beta}\la y\ra^{2-\alpha}.
\end{aligned}
\ee
On $\mcb^{\ep_0}$, one has
\be \label{global-psit}
\begin{aligned}
|\nabla_{\Ss^n}^m\hps| &\le C\|\psi\|_{\widetilde{**},\alpha,\beta;\sigma} |t|^{-\beta-\alpha/(n-2)}\quad \text{for } m=0,1,2,\\
|\pp_t\hps|&\le C\|\psi\|_{\widetilde{**},\alpha,\beta;\sigma} |t|^{-\beta-(\alpha-2)/(n-2)}.
\end{aligned}
\ee

As in Subsection \ref{subsec:comp}, we have that $|\Xi|\le u_*/2$ for sufficiently negative $t_0$. We set $f_1(w) := (1+w)^{p-1}-1$ and $f_2(w) := (1+w)^p-1-pw$, and rewrite \eqref{N} as
\be \label{esnp}
N(\Xi)=-pu_*^{p-1}f_1(u_*^{-1}\Xi)\pp_t(u_*+\Xi)+\frac{n+2}{4}u_*^pf_2(u_*^{-1}\Xi).
\ee
For $|w|\le1/2$, we obtain
\be \label{high}
|f_1(w)|\le C|w|,\quad |f_2(w)|\le Cw^2,\quad |f_1'(w)|\le C,\quad |f_2'(w)|\le C|w|.
\ee
For $|w_1|,|w_2|\le1/2$ and $i=1,2$,
\be \label{low3}
f_i(w_1)-f_i(w_2)=(w_1-w_2)\int_0^1f_i'\big(w_2+v(w_1-w_2)\big)\,dv.
\ee
By \eqref{esnp}--\eqref{high}, we have
\[
|N(\Xi)|\le Cu_*^{p-2}\(\Xi^2+|\Xi|\,|\pp_t\Xi|+|\Xi|\,|\pp_tu_*|\).
\]
Moreover, $|\pp_tu_*|\le C|t|^{-1}u_*$.

\subsubsection{Proof of estimate \eqref{poho2}}\label{app4}
We recall the definition of $\mch_{\mathrm{outer},2}$ in \eqref{hout2}.

\medskip
Let $x\in B(0,\msR_0)$. By the polygonal symmetry \eqref{pp2}, we may assume that $\xi_1$ is a nearest center to $x$ and write $x=\xi_1+\mu y$.
Definition \ref{defpr} then implies $\mcw^*_{\alpha,\beta,2}\asymp |t|^{-\beta}\mu^{-2}\la y\ra^{-\alpha-2}$.
By Lemma \ref{lemma:eta_j}, on the support of $1-\sum_{j=1}^k\eta_j$, each product $\eta_j\tph_j$ vanishes unless $R(t)/4\le|\mu^{-1}(x-\xi_j)|\le5R(t)$. Thus, on this support, $\Xi=\psi$ outside the union of these annuli.

For $R(t)/4\le |y|\le5R(t)$, \eqref{ppt} and Lemma \ref{lemma:eta_j} yield
\[
|t|^\beta\mu^2\la y\ra^{\alpha+2}u_*^{p-1} |\pp_t(\eta_1\tph_1)| \le C\|\phi\|_{\sharp\sharp,a,1;\sigma}|t|^{c'(\alpha-a)}.
\]
Likewise, \eqref{global-psi} implies
\[
|t|^\beta\mu^2\la y\ra^{\alpha+2}u_*^{p-1}|\pp_t\psi| \le C\|\psi\|_{\widetilde{**},\alpha,\beta;\sigma}.
\]

To estimate $u_*^{-1}|\Xi|$, we first use \eqref{ppt} and Lemma \ref{lemma:eta_j} to obtain
\be \label{eq:etajtphj}
u_*^{-1}|\eta_j\tph_j| \le C\|\phi\|_{\sharp\sharp,a,1;\sigma} |t|^{-1+c_0(n-2-a)} \quad \text{for } j=1,\dots,k.
\ee
By Definition \ref{def:num},
\[
1-c_0(n-2-a)-\left[\beta-\frac12+\frac{\alpha}{n-2}+c'(a-\alpha)\right] =\(\frac1{n-2}-c_0\)(n-2-\alpha)>0.
\]
Thus, summing \eqref{eq:etajtphj} over $j=1,\dots,k$ and applying \eqref{global-psi} to $\psi$, we find
\[
u_*^{-1}|\Xi| \le C|t|^{-\beta+\frac12-\frac{\alpha}{n-2}} \Bigl(|t|^{c'(\alpha-a)}\|\phi\|_{\sharp\sharp,a,1;\sigma} +\|\psi\|_{\widetilde{**},\alpha,\beta;\sigma}\Bigr).
\]

Combining these bounds with \eqref{hout2} and \eqref{esnp}--\eqref{high}, we obtain
\begin{align*}
|t|^\beta\mu^2\la y\ra^{\alpha+2} |u_*^{p-1}\mch_{\mathrm{outer},2}| &\le C|t|^{-1} \Bigl(|t|^{c'(\alpha-a)}\|\phi\|_{\sharp\sharp,a,1;\sigma} +\|\psi\|_{\widetilde{**},\alpha,\beta;\sigma}\Bigr)\\
&\ +C|t|^{-\beta+\frac12-\frac{\alpha}{n-2}} \Bigl( |t|^{c'(\alpha-a)}\|\phi\|_{\sharp\sharp,a,1;\sigma} +\|\psi\|_{\widetilde{**},\alpha,\beta;\sigma}\Bigr)^2.
\end{align*}
The same estimates hold for the weighted spatial and temporal H\"older seminorms.

\medskip
We next consider $\mcb^{\ep_0}_{t_0}$, where all cutoff functions $\eta_j$ vanish for sufficiently negative $t_0$. By \eqref{hout2} and \eqref{esnp},
\[
\whmch_{\mathrm{outer},2} = -pf_1\big(\hu_*^{-1}\hps\big)\pp_t\big(\hu_*+\hps\big) +\frac{n+2}{4}\hu_*f_2\big(\hu_*^{-1}\hps\big).
\]
Using $\hu_*\asymp|t|^{-1/2}$, $|\pp_t\hu_*|\le C|t|^{-3/2}$, \eqref{global-psit}, and \eqref{high}, together with the weighted H\"older bounds in Definition \ref{def:norm-refined}, we obtain
\[
\|\whmch_{\mathrm{outer},2}\|_{*',\alpha,\beta;\sigma(\mcb^{\ep_0}_{t_0})} \le C|t_0|^{-1}\|\psi\|_{\widetilde{**},\alpha,\beta;\sigma} +C|t_0|^{-\beta+\frac12-\frac{\alpha}{n-2}} \|\psi\|_{\widetilde{**},\alpha,\beta;\sigma}^2.
\]
Definition \ref{def:num} ensures that $\min\{1,\beta-\frac12+\frac{\alpha}{n-2}\}>2\vep_*$, so \eqref{poho2} is valid.

\medskip
Let us fix $\mu$ and $d$, and consider pairs $(\phi^{(i)},\psi^{(i)})$, $i=1,2$, satisfying the assumed norm bounds.
Subtracting \eqref{esnp} for the corresponding corrections $\Xi^{(i)}$ in \eqref{eq:Xi}, we apply \eqref{low3} with $w_i=u_*^{-1}\Xi^{(i)}$ and use the preceding weighted estimates to obtain
\begin{multline} \label{nonh2}
\left\|\mch_{\mathrm{outer},2}[\phi^{(1)},\psi^{(1)}] -\mch_{\mathrm{outer},2}[\phi^{(2)},\psi^{(2)}]\right\|_{\tilde *,\alpha,\beta;\sigma}\\
\le C|t_0|^{-2\vep_*} \Big[|t_0|^{c'(\alpha-a)} \|\phi^{(1)}-\phi^{(2)}\|_{\sharp\sharp,a,1;\sigma} +\|\psi^{(1)}-\psi^{(2)}\|_{\widetilde{**},\alpha,\beta;\sigma}\Big].
\end{multline}

\subsubsection{Proof of estimate \eqref{eq:H3}}\label{app2}
On $x=\xi_1+\mu y\in\operatorname{supp}\eta_{1,2}$, we have $\la y\ra\le CR$. Using
$|t|^{1/2-\beta}R^{a-\alpha}=|t|^{c'(a-\alpha)}$ in \eqref{ppt}--\eqref{global-psi}, we obtain
\[
u_*^{-1}|\Xi| \le C|t|^{-1+c_0(n-2-a)}  \big(\|\phi\|_{\sharp\sharp,a,1;\sigma} +|t|^{c'(a-\alpha)}\|\psi\|_{\widetilde{**},\alpha,\beta;\sigma}\big).
\]
Since $u_*^{p-1}(x)\asymp\mu^{-2}\la y\ra^{-4}$, we have
\begin{align*}
|t| \la y\ra^{a+2}\mu^{\frac{n+2}{2}} (u_*^{p-1}|\pp_t\Xi|)(x) &\le C\(\|\phi\|_{\sharp\sharp,a,1;\sigma}+|t|^{1/2-\beta}\la y\ra^{a-\alpha}\|\psi\|_{\widetilde{**},\alpha,\beta;\sigma}\)\\
&\le C\(\|\phi\|_{\sharp\sharp,a,1;\sigma}+|t|^{c'(a-\alpha)} \|\psi\|_{\widetilde{**},\alpha,\beta;\sigma}\).
\end{align*}
Thus the $\pp_t\Xi$ term in \eqref{esnp} satisfies
\[
|t|\la y\ra^{a+2}\mu^{\frac{n+2}{2}}(u_*^{p-1} \big|f_1(u_*^{-1}\Xi)\pp_t\Xi\big|)(x)
\le C|t|^{-1+c_0(n-2-a)} \(\|\phi\|_{\sharp\sharp,a,1;\sigma} +|t|^{c'(a-\alpha)}\|\psi\|_{\widetilde{**},\alpha,\beta;\sigma}\)^2.
\]
The same estimate holds for the $f_2$ term with an additional factor $\la y\ra^{-2}$. For the $\pp_tu_*$ term,
\begin{align*}
&\ |t|\la y\ra^{a+2}\mu^{\frac{n+2}{2}}(u_*^{p-1}\big|f_1(u_*^{-1}\Xi)\pp_tu_*\big|)(x)\\
&\le C|t|^{-1} \(\|\phi\|_{\sharp\sharp,a,1;\sigma}\la y\ra^{-2} +|t|^{1/2-\beta}\|\psi\|_{\widetilde{**},\alpha,\beta;\sigma}\la y\ra^{a-\alpha-2}\)\\
&\le C|t|^{-1} \(\|\phi\|_{\sharp\sharp,a,1;\sigma} +|t|^{1/2-\beta}\|\psi\|_{\widetilde{**},\alpha,\beta;\sigma}\),
\end{align*}
since $a-\alpha<2$. The corresponding weighted spatial $C^\sigma$ and temporal $C^{\sigma/2}$ estimates follow by applying the standard product estimates in H\"older spaces to \eqref{esnp}, using \eqref{high}--\eqref{low3} and the weighted H\"older bounds for $\phi$ and $\psi$.
Definition \ref{def:num} tells us that $1-c_0(n-2-a)-2c'(a-\alpha) > 1-c_0(n-2)>2\vep_*$. We conclude that
\[
\left\|U^{1-p}\mu^{\frac{n+2}{2}} [\eta_{1,2}N(\Xi)](\xi_1+\mu y,t)\right\|_{\sharp,a+2,1;\sigma}
\le C\bigl(|t_0|^{-1} +|t_0|^{-1+c_0(n-2-a)+2c'(a-\alpha)}\bigr) \le C|t_0|^{-2\vep_*}.
\]
This proves \eqref{eq:H3}.

\subsection{\texorpdfstring{Justification of \eqref{psi-lip-2}, \eqref{plim}, and \eqref{plimd}}{Justification}}\label{app5}
Fix $t_0$ and $\msM_1,\msM_2$ as in Propositions \ref{prop:nonlin} and \ref{prop:outer}.
Suppose that $\mu^{(m)}\to\mu$ in $C^1_{\loc}((-\infty,t_0])$, with the uniform bounds \eqref{mu} and \eqref{mu2}.
The weighted H\"older bounds \eqref{est-phi-inner} and \eqref{psi-norm-bound} yield subsequential convergence of $\phi^{(m)},\psi^{(m)}$ in $C^{2,1}_{\loc}(\R^n\times(-\infty,t_0])$,
of $\hps^{(m)}$ in $C^{2,1}_{\loc}(\Ss^n\times(-\infty,t_0])$, and of $c_{n+1}^{(m)}$ in $C^0_{\loc}((-\infty,t_0])$.
The coefficients and cutoffs in \eqref{inner} and \eqref{def-cutoff-j}, together with the nonlinear terms in \eqref{mh1} and \eqref{esnp}, converge locally.
Thus the limits satisfy \eqref{eq:nonlin} and \eqref{outer} with parameter $\mu$, and the prescribed symmetries are preserved.

On each compact time interval $I\subset(-\infty,t_0]$, \eqref{ppt}, \eqref{eq:bubble}, and \eqref{eq:Zi} imply the uniform bound $|\phi^{(m)}|(U^{p-1}|Z_i|+U^p)\le C_I\la y\ra^{-n-a-2}$.
Dominated convergence therefore preserves \eqref{eq:inhom-i21}--\eqref{eq:inhom-i22}, including the normalization at $t_0$.
Lower semicontinuity preserves \eqref{est-phi-inner} and \eqref{psi-norm-bound}, with the weights associated with $\mu$.
Proposition \ref{prop:outer} identifies the outer limit as the solution corresponding to the limiting inner profile.
Uniqueness in Proposition \ref{prop:nonlin} then identifies the limiting pair $(\phi,c_{n+1})$.
Since every subsequential limit is thereby determined, \eqref{plim}--\eqref{plimd} hold for the full sequence.
For data satisfying \eqref{psi-lip-1}, the same argument applied to the outer problem proves \eqref{psi-lip-2} and convergence of the conformal lifts.

\section{A Compact Quadratic Type II Ancient Flow on \texorpdfstring{$\Ss^4$}{the 4-sphere}} \label{app:quadratic-S4}
We give a non-rotational compact Type II ancient Yamabe flow $g(s)$ on $\Ss^4$, satisfying \eqref{YF-ori}, whose Yamabe quotient $\mcy(g(s))$ in \eqref{eq:Yamabequo} tends to infinity as $s\to-\infty$.
The construction is motivated by the polynomial solutions of nonlinear diffusion equations in \cite{King93} and by the noncompact Yamabe flows with quadratic reciprocal conformal factors in \cite{BMS08}.

\medskip
Let $A:(-\infty,0]\to\operatorname{Sym}^+(n+1)$ be a $C^1$ curve of positive-definite symmetric $(n+1)\times(n+1)$ matrices, and define
\[
q(z,s)=z^TA(s)z \quad \text{and} \quad g(s)=q(\cdot,s)^{-1}g_{\Ss^n}.
\]
For each fixed $s$, the identities
\[
\Delta_{\Ss^n}q=2\operatorname{tr}A-2(n+1)q \quad \text{and} \quad |\nabla_{\Ss^n}q|^2=4\big(z^TA^2z-q^2\big)
\]
and the conformal scalar-curvature formula give
\[
qR_g=(n-1)\left[2(\operatorname{tr}A)q-(n+2)z^TA^2z\right].
\]
Since \eqref{YF-ori} is equivalent to $\pp_s q =q R_g$, and $\pp_s q(z,s)=z^TA'(s)z$, the quadratic ansatz is preserved by the Yamabe flow precisely when
\be \label{eq:A'}
A'=(n-1)\left[2(\operatorname{tr}A)A -(n+2)A^2\right].
\ee

If we set $n=4$ and $A(s) =\operatorname{diag}(\lambda_+(s) \Id_3,\lambda_-(s) \Id_2) \in \operatorname{Sym}^+(5)$, where $0<\lambda_-(s)<\lambda_+(s)$ for all $s \in (-\infty,0]$, then \eqref{eq:A'} becomes
\[
\begin{cases}
\lambda_+'=12\lambda_+\lambda_-,\\
\lambda_-'=18\lambda_+\lambda_--6\lambda_-^2.
\end{cases}
\]
Writing $r =\lambda_-/\lambda_+$, we obtain $\lambda_+=\lambda_*(1-r )^{-2/3}$ and $r'=18\lambda_*r (1-r )^{1/3}$, where $\lambda_*>0$ is a constant. For $r(0)=r_0\in(0,1/4)$, the relation
\[
s=\frac1{18\lambda_*}\int_{r_0}^{r(s)}\frac{du}{u(1-u)^{1/3}}
\]
defines a smooth compact ancient Yamabe flow $g(s)$ for every $s\le0$, with $r(s)>0$ and $r(s)\to0$, $\lambda_+(s)\to \lambda_*$, $\lambda_-(s)\to0$ as $s\to-\infty$.

At $z=(z',0)\in\Ss^4\subset\R^3\oplus\R^2$, we have $q(z,s)=\lambda_+$ and $dq(z,s)=0$. Furthermore, $T_z\Ss^4=T_{z'}\Ss^2\oplus\R^2$, and
\[
\nabla^2q=0\quad \text{on } T_{z'}\Ss^2,\qquad
\nabla^2q=2(\lambda_--\lambda_+)g_{\Ss^4}\quad \text{on } \R^2.
\]
Since $g=q^{-1}g_{\Ss^4}$, the conformal Ricci formula at a critical point of $q$ implies
\[
\Ric_g = 3g_{\Ss^4} +\frac1q\nabla^2q +\frac{\Delta_{\Ss^4}q}{2q}g_{\Ss^4}.
\]
Here, $\Delta_{\Ss^4}q =2\operatorname{tr}A-10q =2(3\lambda_++2\lambda_-)-10\lambda_+ =4(\lambda_--\lambda_+)$.
It follows that, with respect to $g$, the Ricci tensor at $z=(z',0)$ has eigenvalues
\[
\lambda_++2\lambda_- \quad \text{on } T_{z'}\Ss^2, \qquad -\lambda_++4\lambda_- \quad \text{on } \R^2,
\]
each of multiplicity $2$. Since $\lambda_+>\lambda_-$, these eigenvalues are distinct.
This multiplicity pattern excludes rotational symmetry, even after a diffeomorphism and constant scaling.
Indeed, away from the rotation poles, a rotationally symmetric metric on $\Ss^4$ has one radial Ricci eigenvalue and one tangential eigenvalue of multiplicity $3$;
at a rotation pole, the Ricci tensor is a multiple of the metric. Hence two distinct Ricci eigenvalues of multiplicity $2$ cannot occur.
Moreover, $r(s)<1/4$ implies that the two eigenvalues have opposite signs for every $s\le0$, and $\lambda_++2\lambda_-\to\lambda_*$ and $-\lambda_++4\lambda_-\to-\lambda_*$ as $s\to-\infty$.
Thus the curvature stays bounded away from zero backward in time, so the Type I bound $|\Rm_{g(s)}|_{g(s)}=O(|s|^{-1})$ fails. Hence the ancient flow is Type II.

Finally, setting $\rho=|z'|$, we have
\begin{align*}
\Vol_{g(s)}(\Ss^4) &= 8\pi^2\int_0^1 \frac{\rho^2\,d\rho}{[\lambda_-+(\lambda_+-\lambda_-)\rho^2]^2}, \\
R_g &= \frac{6\lambda_-[3\lambda_+-\lambda_--(\lambda_+-\lambda_-)\rho^2]} {\lambda_-+(\lambda_+-\lambda_-)\rho^2},
\end{align*}
which yields
\be \label{eq:Vs4dim}
\Vol_{g(s)}(\Ss^4)\asymp\lambda_-^{-1/2} \quad \text{and} \quad \frac1{\Vol_{g(s)}(\Ss^4)}\int_{\Ss^4}R_g\,dv_g\to\frac92\lambda_*
\ee
as $s\to-\infty$. Employing \eqref{eq:Yamabequo} and \eqref{eq:Vs4dim}, we obtain $\mcy(g(s))\asymp\lambda_-^{-1/4}\to\infty$ as $s\to-\infty$.
Besides, $q(\cdot,s)\to\lambda_*|z'|^2$, whose zero set is the circle $\{(0,z'')\in\Ss^4:|z''|=1\}\simeq\Ss^1$.
Therefore, the conformal factor $q^{-1}$ becomes singular along a one-dimensional set rather than at finitely many points.


\begin{thebibliography}{99}
\bibitem{ADS20}
S. Angenent, P. Daskalopoulos, and N. Sesum, Uniqueness of two-convex closed ancient solutions to the mean curvature flow, Ann. of Math. \textbf{192} (2020), 353--436.

\bibitem{BKN12}
I. Bakas, S. Kong, and L. Ni, Ancient solutions of Ricci flow on spheres and generalized Hopf fibrations, J. Reine Angew. Math. \textbf{663} (2012), 209--248.

\bibitem{brendle05}
S. Brendle, Convergence of the Yamabe flow for arbitrary initial energy, J. Differential Geom. \textbf{69} (2005), 217--278.

\bibitem{brendle07}
\bysame, Convergence of the Yamabe flow in dimension 6 and higher, Invent. Math. \textbf{170} (2007), 541--576.

\bibitem{BDNS23}
S. Brendle, P. Daskalopoulos, K. Naff, and N. Sesum, Uniqueness of compact ancient solutions to the higher-dimensional Ricci flow, J. Reine Angew. Math. \textbf{795} (2023), 85--138.

\bibitem{BDS21}
S. Brendle, P. Daskalopoulos, and N. Sesum, Uniqueness of compact ancient solutions to three-dimensional Ricci flow, Invent. Math. \textbf{226} (2021), 579--651.

\bibitem{BK17}
S. Brendle and N. Kapouleas, Gluing Eguchi--Hanson metrics and a question of Page, Comm. Pure Appl. Math. \textbf{70} (2017), 1366--1401.

\bibitem{BMS08}
A. Burchard, R. J. McCann, and A. Smith, Explicit Yamabe flow of an asymmetric cigar, Methods Appl. Anal. \textbf{15} (2008), 65--80.

\bibitem{CGS89}
L. A. Caffarelli, B. Gidas, and J. Spruck, Asymptotic symmetry and local behavior of semilinear elliptic equations with critical Sobolev growth, Comm. Pure Appl. Math. \textbf{42} (1989), 271--297.

\bibitem{chow92}
B. Chow, The Yamabe flow on locally conformally flat manifolds with positive Ricci curvature, Comm. Pure Appl. Math. \textbf{45} (1992), 1003--1014.

\bibitem{CDM20}
C. Cort\'azar, M. del Pino, and M. Musso, Green's function and infinite-time bubbling in the critical nonlinear heat equation, J. Eur. Math. Soc. \textbf{22} (2020), 283--344.

\bibitem{CS11}
E. Crane and I. Short, Rigidity of configurations of balls and points in the $N$-sphere, Q. J. Math. \textbf{62} (2011), 351--362.

\bibitem{ddks}
P. Daskalopoulos, M. del Pino, J. King, and N. Sesum, New type I ancient compact solutions of the Yamabe flow, Math. Res. Lett. \textbf{24} (2017), 1667--1691.

\bibitem{dds}
P. Daskalopoulos, M. del Pino, and N. Sesum, Type II ancient compact solutions to the Yamabe flow, J. Reine Angew. Math. \textbf{738} (2018), 1--71.

\bibitem{DHS10}
P. Daskalopoulos, R. Hamilton, and N. Sesum, Classification of compact ancient solutions to the curve shortening flow, J. Differential Geom. \textbf{84} (2010), 455--464.

\bibitem{DHS12}
\bysame, Classification of ancient compact solutions to the Ricci flow on surfaces, J. Differential Geom. \textbf{91} (2012), 171--214.

\bibitem{DDW20}
J. D\'avila, M. del Pino, and J. Wei, Singularity formation for the two-dimensional harmonic map flow into $\Ss^2$, Invent. Math. \textbf{219} (2020), 345--466.

\bibitem{dmpp}
M. del Pino, M. Musso, F. Pacard, and A. Pistoia, Large energy entire solutions for the Yamabe equation, J. Differential Equations \textbf{251} (2011), 2568--2597.

\bibitem{ha89}
R. Hamilton, Lectures on geometric flows, unpublished manuscript (1989).

\bibitem{KME}
K. Kim and F. Merle, On classification of global dynamics for energy-critical equivariant harmonic map heat flows and radial nonlinear heat equation, Comm. Pure Appl. Math. \textbf{78} (2025), 1783--1842.

\bibitem{KME2}
\bysame, Rigidity results in multi-bubble dynamics for non-radial energy-critical heat equation, preprint, arXiv:2601.12517.

\bibitem{km}
S. Kim and M. Musso, Infinite-time blowing-up solutions to small perturbations of the Yamabe flow, Adv. Math. \textbf{443} (2024), Paper No. 109611, 77 pp.

\bibitem{King93}
J. R. King, Exact polynomial solutions to some nonlinear diffusion equations, Phys. D \textbf{64} (1993), 39--65.

\bibitem{LS26}
Y. Li and L. Sun, Planar doubling nodal solutions to the Yamabe equation with maximal rank, preprint, arXiv:2604.02978.

\bibitem{MM21}
M. Medina and M. Musso, Doubling nodal solutions to the Yamabe equation in $\R^n$ with maximal rank, J. Math. Pures Appl. \textbf{152} (2021), 145--188.

\bibitem{MMW19}
M. Medina, M. Musso, and J. Wei, Desingularization of Clifford torus and nonradial solutions to the Yamabe problem with maximal rank, J. Funct. Anal. \textbf{276} (2019), 2470--2523.

\bibitem{MN75}
S. A. Morris and E. S. Noussair, The Schauder--Tychonoff fixed point theorem and applications, Matematick\'y \v{c}asopis \textbf{25} (1975), 165--172.

\bibitem{Ro}
P. Rosenau, Fast and superfast diffusion processes, Phys. Rev. Lett. \textbf{74} (1995), 1056--1059.

\bibitem{SS03}
H. Schwetlick and M. Struwe, Convergence of the Yamabe flow for large energies, J. Reine Angew. Math. \textbf{562} (2003), 59--100.

\bibitem{ye94}
R. Ye, Global existence and convergence of Yamabe flow, J. Differential Geom. \textbf{39} (1994), 35--50.

\end{thebibliography}
\end{document}